\documentclass{article}
\usepackage{geometry}
\usepackage{amsthm,amsmath,amsfonts,amssymb}
\usepackage[round]{natbib}
\usepackage[colorlinks,citecolor=blue]{hyperref}

\usepackage{bm}
\usepackage{bbm}
\usepackage[normalem]{ulem}
\usepackage[T1]{fontenc}
\usepackage{enumitem}

\newtheorem{theorem}{Theorem}[section]
\newtheorem{corollary}[theorem]{Corollary}
\newtheorem{lemma}[theorem]{Lemma}
\newtheorem{proposition}[theorem]{Proposition}

\theoremstyle{definition}
\newtheorem{definition}[theorem]{Definition}
\newtheorem{remark}[theorem]{\textbf{Remark}}
\newtheorem{example}[theorem]{\textbf{Example}}

\numberwithin{equation}{section}
\numberwithin{theorem}{section}

\newcommand{\E}{\mathbb{E}}

\renewcommand{\P}{\mathbb{P}}
\newcommand{\Pp}{\mathbb{Q}}
\newcommand{\Q}{\mathbb{Q}}

\newcommand{\di}{\mathrm{d}}

\newcommand{\pian}[2]{\dfrac{\partial #1}{\partial #2}}
\newcommand{\piann}[2]{\dfrac{\partial^{2}#1}{\partial #2^{2}}}

\newcommand{\R}{\mathbb{R}}

\newcommand{\Lc}{\mathcal{L}}
\newcommand{\as}{a.s.}

\newcommand{\eg}{e.g.}
\newcommand{\cf}{c.f.}
\newcommand{\dx}{\mathrm{d}x}
\newcommand{\dy}{\mathrm{d}y}
\newcommand{\ds}{\mathrm{d}s}

\newcommand{\du}{\mathrm{d}u}
\newcommand{\dt}{\mathrm{d}t}

\newcommand{\Fb}{\mathbb{F}}
\newcommand{\Ec}{\mathcal{E}}
\newcommand{\Fc}{\mathcal{F}}

\newcommand{\Hc}{\mathcal{H}}
\newcommand{\Hb}{\mathbb{H}}

\newcommand{\T}{\mathrm{T}}
\DeclareMathOperator{\supp}{supp}

\renewcommand{\r}{R}
\renewcommand{\Pr}{\P}

\newcommand{\Mb}{\mathbb{M}}
\newcommand{\Nb}{\mathbb{N}}

\newcommand{\Pc}{\mathcal{P}}

\newcommand{\Bc}{\mathcal{B}}

\newcommand{\cadlag}{c\`adl\`ag}

\newcommand{\Cov}{\mathbb{C}\mathrm{ov}}

\newcommand{\Var}{\mathbb{V}\mathrm{ar}}
\newcommand{\Corr}{\mathbb{C}\mathrm{orr}}

\newcommand{\Pk}{\mathfrak{P}}

\renewcommand{\d}{{\rm d}}
\newcommand{\e}{\varepsilon}

\renewcommand{\P}{{\mathbb P}}

\newcommand{\N}{{\mathbb N}}

\newcommand{\Acal}{{\mathcal A}}
\newcommand{\Bcal}{{\mathcal B}}

\newcommand{\Fcal}{{\mathcal F}}

\newcommand{\Hcal}{{\mathcal H}}
\newcommand{\Ical}{{\mathcal I}}

\newcommand{\Pcal}{{\mathcal P}}

\newcommand{\fdot}{{\,\cdot\,}}

\newcommand{\id}{{\rm id}}

\DeclareMathOperator{\diam}{diam}

\begin{document}

\title{Controlled measure-valued martingales: a viscosity solution approach}
\author{Alexander M.~G.~Cox\thanks{Department of Mathematical
    Sciences, University of Bath, Claverton Down, Bath BA2 7AY, U.~K., \url{a.m.g.cox@bath.ac.uk}.}
\and Sigrid K{\"a}llblad\thanks{Department of Mathematics, KTH Royal Institute of Technology, Lindstedtsv{\"a}gen 25, 100 44 Stockholm, Sweden, \url{sigrid.kallblad@math.kth.se}.
%The author gratefully acknowledges support from the Swedish Research Council under grant 2020-03449.
}
\and Martin Larsson\thanks{Department of Mathematical Sciences, Carnegie Mellon University, Wean Hall, 5000 Forbes Ave, Pittsburgh, Pennsylvania 15213, USA, \url{larsson@cmu.edu}.}
\and Sara Svaluto-Ferro\thanks{Department of Economics, University of Verona, Via Cantarane 24, 37129 Verona, Italy, \url{sara.svalutoferro@univr.it}.
The author gratefully acknowledges financial support by the Vienna Science and Technology Fund (WWTF) under grant MA16-021.}}
%\date{}
\maketitle

\begin{abstract}
%\noindent$\qquad$\hrulefill$\qquad$
	
	We consider a class of stochastic control problems where the state process is a probability measure-valued process satisfying an additional martingale condition on its dynamics, called measure-valued martingales (MVMs). We establish the `classical' results of stochastic control for these problems: specifically, we prove that the value function for the problem can be characterised as the unique solution to the Hamilton-Jacobi-Bellman equation in the sense of viscosity solutions. In order to prove this result, we exploit structural properties of the MVM processes. Our results also include an appropriate version of It\^o's formula for controlled MVMs.

	We also show how problems of this type arise in a number of applications, including model-independent derivatives pricing, the optimal Skorokhod embedding problem, and two player games with asymmetric information.
	
\end{abstract}

{\hypersetup{linkcolor=black}
\tableofcontents
}

\section{Introduction}

Recently there has been substantial interest in understanding
stochastic control of processes which take values in the set of probability
measures. In particular, the study of stochastic control problems
where the underlying state variable is a probability measure have been
studied in a number of contexts such as mean-field games, and
McKean-Vlasov dynamics. In this paper, we consider
stochastic control problems where the state process is a
probability measure-valued process, satisfying an additional
martingale condition which restricts the possible dynamics of the
process. The restrictions on the dynamics of the process provide
enough regularity to prove the `classical' theorems of stochastic
control, specifically, dynamic programming, identification of the
value function as a solution (in an appropriate sense) to a
Hamilton-Jacobi-Bellman (HJB) equation and a verification theorem for
`classical' solutions. Under stronger conditions, we are also able to
prove comparison for the HJB equation, allowing characterisation of
the value function as the unique solution to this equation.

The probability measure-valued evolution we wish to study as our underlying state
variable is the class of \emph{measure-valued martingales}, or
\emph{MVMs}, introduced in \cite{cox2017}. A process $(\xi_t)_{t \ge
  0}$, taking values in the space of probability measures on $\R^d$ is an MVM if $\xi_t(\varphi) := \int_{\R^d} \varphi(x)\xi_t(\dx)$ is a
martingale for every bounded continuous function $\varphi \in C_b(\R^d)$. Such processes arise naturally
in a number of contexts, and we outline some of these applications
below.

In \cite{cox2017}, MVMs were introduced in the context of model-independent pricing and hedging of financial derivatives. In this application, the measure $\mu$ has an interpretation as the implied distribution of the asset price $S_T$ given the information at time $t$, $\xi_t(A) = \Q(S_T \in A|\Fc_t)$, where $\Q$ is the risk-neutral measure. In the model-independent pricing literature, initiated by \cite{hobson1998}, one typically does not assume that the law of the process $S$ is known, but rather one observes market information in terms of the European call prices with maturity time $T$, and tries to find bounds on the prices of exotic derivatives as the maximum/minimum over all models which fit with the market information. In practice, since the market prices of call options imply that the law of $S_T$ is known at time zero via the Breeden-Litzenberger formula, \citep{breeden_prices_1978}, this turns out to be equivalent to knowing $\xi_0$, the starting point of the MVM from market information; the risk-neutral assumption additionally grants that the process $\xi$ will then be an MVM under any risk-neutral measure. Optimising over all models for $S$ which have terminal law $\xi_0$ can be shown to be equivalent to optimising over the laws of MVMs which start at $\xi_0$ and satisfy an additional terminal condition. While increasing the complexity of the optimisation problem by making the state variable infinite dimensional, this avoids the tricky distributional constraint on the terminal law of the process. In \cite{cox2017} and \cite{bayraktar_martingale_2018}, this connection was used to characterise the model-independent bounds of Asian and American-type options.
See also e.g. \cite{kallblad2017dynamic} for the use of MVMs to address distribution-constrained optimal stopping problems. 

Further related to this problem, although also of interest in its own
right, is the problem of finding optimal solutions to the Skorokhod
Embedding Problem. Given an integrable measure $\mu$ and a Brownian
motion $B$, the Skorokhod Embedding problem (SEP) is to find a
stopping time $\tau$ such that the process
$(B_{t \wedge \tau})_{t \ge 0}$ is uniformly integrable, and
$B_\tau \sim \mu$. By introducing the conditioned, probability
measure-valued process $\xi_t(A) := \Pr(B_{\tau} \in A|\Fc_t)$, it
follows that $B_{t \wedge \tau} = \int_{\R} x \xi_t(\dx)$. In this case,
the process $\xi_t$ is evidently an MVM, and in fact,
it can be shown that there is an equivalence between solutions to the
SEP and MVMs which \emph{terminate}, that is, converge to a (random)
point mass (see \cite{cox2017}). In many applications of the SEP, one is interested in finding
\emph{optimal} solutions to the SEP (see
\cite{obloj2004,beiglboeck2017}), and one approach is to
reformulate this problem in terms of the MVM, and to optimise over the
class of MVMs. Approaches to the SEP using an MVM-like perspective can
be traced back (indirectly) to the construction of
\cite{bass_skorokhod_1983}. More recent developments in this direction
include \cite{eldan2016} and \cite{beiglboeck2017mvm}.

A second class of problems in which MVMs naturally arise is in the
setting of two-player, zero sum games with asymmetric
information. These games were initially introduced in discrete time by
\cite{aumann1995}, and subsequently have been the subject of
systematic investigation by Cardialiguet, Rainer and Gr\"un, among
others
(\cite{cardaliaguet_rainer2009_no_mvm,cardaliaguet_rainer2009,cardaliaguet2009,cardaliaguet2012,gensbittel2018,grun_dynkin_2013}). In
these games, the payoff of the game depends on a parameter $\theta$
which is known at the outset to the first player, but which is unknown
to the second player, whose belief in the value of the parameter is
known to be some probability measure $\xi_0$. In the game, both players
act to optimise their final reward, and the actions of the first player
may inform the second player about the value of the parameter. It
follows that the posterior belief of the second player at time $t$,
$\xi_t$ follows the dynamics of an MVM. Moreover, the strategies of
the first player can be reformulated into a control problem, where the
state variable of the problem is the posterior belief of the second
player, $\xi_t$. Consequently, the game formulation fits into the
setup of a controlled MVM problem. 

Our main results follow the classical approach to stochastic
control. We will make one major restriction to the full generality of
the problem by assuming that we can restrict our MVM to processes
driven by a Brownian motion. In this framework, we will postulate
dynamics for the MVM in terms of an SDE where we are able to identify
a natural class of (function-valued) controls. Once this natural set
of controls is established, we are able to formulate the control
problem for a controlled measure-valued process. In this setting, we
then proceed to establish a corresponding Hamilton-Jacobi-Bellman
(HJB) equation which we expect our value function to satisfy. In order
to uniquely characterise the value function, it is necessary to
introduce an appropriate sense of weak solution to the HJB equation,
which we do using viscosity theory. Specifically, we introduce a
notion of viscosity solution which, in our setting and under
appropriate conditions on the problem, allows us to show that the
value function is a viscosity solution to the HJB equation, and also
prove a comparison result, under which we further conclude that the
value function is the unique such solution. Our notion of viscosity
solution will exploit the specific nature of the dynamics of the MVM
and allows to prove some of the viscosity results above, which are
notoriously hard to prove in the general setting of measure-valued
processes. Our proof of comparison depends on a continuity assumption on the value function, which is not required for our other results. This is needed for a reduction to the case of finitely supported measures, where finite-dimensional viscosity theory can be applied. It would be of great interest to find a proof of comparison without any \emph{a priori} continuity of the value function.

Our results have connections with existing results
in the literature. Broadly, we believe that a special case of our class of MVMs corresponds to a controlled filtering problem, where the process being filtered is constant. There is an
existing literature on these problems, culminating with e.g.~\cite{fabbri_stochastic_2017,gozzi_hamiltonjacobibellman_2000, nisio_book_2015}. In
comparison with our approach, these works formulate the dynamics of
the problem in terms of an (unnormalised) density function, which is
embedded in an appropriate vector space. In comparison, we formulate
our problem directly in the underlying (metric) space of probability
measures. More recently, \citep{bandini_randomized_2019} considered a
related problem in metric space setting, however their control problem
arises in the context of partial observation of a diffusion, and the
two problems do not appear to be directly comparable.

More recently, there has been substantial interest in McKean-Vlasov
equations, including viscosity solutions for
control problems where the state variables take values in the space of
probability measures. In particular, this involves obtaining It\^o formulas for probability measure-valued processes arising as the (conditional or unconditional) laws of an underlying state process; see \cite{cha_cri_del_14,MR3630288,pham_bellman_2018,car_del_18_I,carmona_probabilistic_2018,burzoni_viscosity_2020,guo_pha_wei_20,tal_tou_zha_21,cosso_optimal_2020,wu_viscosity_2020}. However, these probability measure-valued processes are not MVMs except in degenerate instances, and these papers therefore have limited bearing on the results we develop here.
To see this, we observe that a key property of MVMs is to always decrease the support of the measures. As such, measure-valued dynamics such as McKean-Vlasov are generally excluded from our analysis, since they are the limits of particle approximations, where the particles naturally spread out on account of their diffusive nature. Trivially, any MVM which is started in an atomic measure will never gain support outside the initial atoms, and hence any attempt to interpret it as the limit of diffusive particle models such as McKean-Vlasov will fail unless the particles are all assumed to be constant.

%\fbox{OLD:}
%In
%comparison to our work, the dynamics established in these papers are
%motivated by very different concerns: typically the evolution of the
%measure is specified in terms of a fixed point equation, where the
%current state evolves locally according to SDE-like dynamics whose
%dynamics may themselves depend on the global state of the problem. In
%short, one would typically expect the evolution of the measures to be
%dominated locally by diffusive behaviour, while for the MVM dynamics
%we consider, there is typically no similar diffusive behaviour;
%instead the evolution of the measure is dominated by the relative
%likelihood of the probability measure in different locations. As a
%simple example, in the setting of our paper, an MVM started in a
%purely atomic distribution can evolve stochastically without increasing
%its support. In all the papers cited above, restricting the initial
%value of the measure to have finite support would force the dynamics
%to be constant.

The rest of the paper is structured as follows. In
Section~\ref{sec:meas-valu-mart} we give a formal definition of an MVM
and establish certain helpful properties, including giving a natural
notion of control of MVMs. In Section~\ref{sec:contr-probl-dynam} we formally state our stochastic control
problem, and show the important, non-trivial fact that constant
controls exist in our formulation. In Section~\ref{sec:diff-calc} we
establish an appropriate differential calculus in our setting, which
enables us, in Section~\ref{S_Ito} to prove a version of It\^o's formula
in our setting. In Section~\ref{S_viscosity} we state our main result,
including our definition of a viscosity solution and a verification
result for classical solutions. The proofs of the
main result are then detailed in Sections~\ref{sec:visc-subs-prop},
\ref{sec:visc-supers-prop} and \ref{sec:comparison-principle} where we
prove the sub-, and super-solution properties, and a comparison
principle; the proof of the dynamic programming principle is deferred to Appendix~\ref{app:DPP}.
Finally, in Section~\ref{sec:applications} we give some
concrete examples of solvable control problems, and also explain how
our main results relate to the applications set out above.
Appendix \ref{secC} reports some auxiliary properties of the notion of derivative used in this paper. \\

\noindent\textbf{Notation.}
The following notation will feature throughout the paper. We fix $d \in \N$.

\begin{itemize}
\item $\Pcal$ denotes the space of probability measures on $\R^d$ with the topology of weak convergence. $\Pcal_p$ for $p \in [1,\infty)$ denotes the probability measures whose $p$-th moment is finite, endowed with the Wasserstein-$p$ metric. We set $\Pcal_0 = \Pcal$ by convention. All these spaces are Polish. Finally, $\Pcal^s$ denotes the (closed) subset of probability measures supported in one single point.
%\item $\Pcal=\Pcal(\R^d)=\Pcal_0=\Pcal_0(\R^d)$: the Polish space of all probability measures  on $\R^d$ endowed with the topology of weak convergence.
%\item $\Pcal_p=\Pcal_p(\R^d)$, $p \in \{0\} \cup [1,\infty)$: the subset of $\Pcal$ with
%  finite $p$-th moment endowed with the topology induced by the
%  Wasserstein-$p$ metric, which is also a Polish space. 
%  \comment{Martin to rewrite.}
%\item $\Pcal^s=\Pcal^s(\R^d)$: the (closed) set of probability measures on $\R^d$ supported in one point.
\item $C_b(\R^d)$ and $C_c(\R^d)$ are the bounded continuous and compactly supported continuous functions on $\R^d$, respectively. They are frequently abbreviated as $C_b$ and $C_c$. We also write $C(\Pcal_p)$ for the real-valued continuous functions on $\Pcal_p$.
%\item $C_b=C_b(\R^d)$: the set of bounded continuous functions on
%  $\R^d$, and $C_c=C_c(\R^d)$: the set of compactly supported continuous functions on $\R^d$.
\item %We introduce various definitions of standard operations for measures and functions.
For $\mu\in \Pcal_p$ and $\varphi:\R^d\to \R$ such that $\int_{\R^d}|\varphi(x)| \mu(\d x)<\infty$ we set 
$$\mu(\varphi):=\int_{\R^d} \varphi(x) \mu(\d x).$$
When $d = 1$ we write $\Mb(\mu):=\mu(\id)$ (if $p \ge 1$), where $\id \colon x \mapsto x$ is the identity function, and $\Var(\mu) := \int_{\R^d} x^2 \mu(\dx) - (\mu(\id))^2$ (if $p \ge 2$). In addition, we write the covariance under $\mu$ of two functions $\varphi$ and $\psi$ as $\Cov_\mu(\varphi,\psi) := \mu(\varphi \psi) - \mu(\varphi) \mu(\psi)$ and similarly $\Var_\mu(\varphi) := \Cov_\mu(\varphi,\varphi)$. Note that $\Var(\mu) = \Var_\mu(\id)$.
%For $p\geq 1$ set then $\Mb(\mu):=\mu(\id)$, where $\id$ is
%  the identity function.
%For $p\geq 2$ we also set 
 % $\Var(\mu) := \int_{\R^d} x^2 \mu(\dx) - (\mu(\id))^2$.
%  In addition, we extend to define the
%  covariance of functions $\varphi$ and $\psi$ relative to $\mu$ by
%  $\Cov_\mu(\varphi,\psi) := \mu(\varphi \psi) - \mu(\varphi) \mu(\psi)$. In this notation,
%  we also write $\Var_\mu(\varphi) := \Cov_\mu(\varphi,\varphi)$. Note that $\Var(\mu) = \Var_\mu(\id)$.
\end{itemize}

\section{Measure-valued martingales}
\label{sec:meas-valu-mart}

\begin{definition}
A \emph{measure-valued martingale} (MVM) is a $\Pcal$-valued adapted stochastic process $\xi=(\xi_t)_{t\ge0}$, defined on some filtered probability space $(\Omega,\Fcal,(\Fcal_t)_{t\ge0},\P)$, such that $\xi(\varphi)$ is a real-valued martingale for every $\varphi\in C_b$. We say that an MVM is \emph{continuous} if it has weakly continuous trajectories, or equivalently, if $\xi(\varphi)$ is continuous for every $\varphi\in C_b$.
\end{definition}

%Here $\xi$ is defined on some filtered probability space $(\Omega,\Fcal,(\Fcal_t)_{t\ge0},\P)$, and the martingale property is understood with respect to this space. 
In this paper we consider control problems and stochastic equations in a weak formulation, meaning that the probability space is not fixed, but rather constructed as needed. Note that there is a connection to the class of `martingale measures' as defined in \eg{} \cite{dawson1993}. However in contrast to the definition there, we make the additional restriction that our processes remain as probability measures.

The following lemma shows that the martingale property of $\xi(\varphi)$ extends beyond bounded continuous functions. It applies to arbitrary MVMs with continuous trajectories.

\begin{lemma}\label{L_lin_fcn}
Let $\xi$ be a continuous MVM, and let $\varphi$ be any nonnegative measurable function such that $\E[\xi_0(\varphi)]<\infty$. Then $\xi(\varphi)$ is a uniformly integrable continuous martingale.
\end{lemma}

\begin{proof}
Let $\Hcal$ be the set of all bounded measurable functions $\varphi$
such that $\xi(\varphi)$ is a continuous
martingale (necessarily uniformly bounded). Let $\varphi_n\in\Hcal$, and assume that the $\varphi_n$
increase pointwise to a bounded function $\varphi$. Since $\xi_t(\varphi) = \lim_{n\to\infty} \xi_t(\varphi_n)$, the process $\xi(\varphi)$
is adapted.  The stopping theorem yields $\E[\xi_\tau(\varphi_n)]=\E[\xi_0(\varphi_n)]$ for every finite stopping time $\tau$ and all $n\in\N$, and sending $n\to\infty$ gives $\E[\xi_\tau(\varphi)]=\E[\xi_0(\varphi)]$ by monotone convergence. This implies that $\xi(\varphi)$ is a martingale; see e.g.\ \cite[Proposition~II.1.4]{MR1725357}. Next, since $\xi(\varphi_n)$ is a continuous martingale for every $n$, Doob's inequality yields
\begin{align*}
\P\left( \sup_{t\le T}|\xi_t(\varphi_m)-\xi_t(\varphi_n)| > \varepsilon \right) &\le \frac1\varepsilon \E[|\xi_T(\varphi_m - \varphi_n)|] \\
&\le \frac1\varepsilon \E[\xi_T(\varphi - \varphi_{m\wedge n})]
\end{align*}
for all $T\ge0$, $m,n\in\N$, $\varepsilon>0$. Keeping $\varepsilon>0$ fixed, the dominated convergence theorem implies that the right-hand side vanishes as $m,n\to\infty$. Since $\xi(\varphi_n)$ is continuous for each $n$, so is the limit $\xi(\varphi)$. We have proved that $\varphi\in\Hcal$, and deduce from the monotone class theorem that $\Hcal$ consists of all bounded measurable $\varphi$. Next, let $\varphi$ be nonnegative and measurable with $\E[\xi_0(\varphi)]<\infty$. The same argument as above with $\varphi_n=\varphi\wedge n$ shows that $\E[\xi_\tau(\varphi)]=\E[\xi_0(\varphi)]$ for every finite stopping time $\tau$. Thanks to \cite[Theorem~5.1]{MR2233537}, this implies that $\xi(\varphi)$ is a uniformly integrable martingale, and it is continuous by the same argument as above.
\end{proof}

\begin{remark} \label{rem:MVMCompact}
Lemma~\ref{L_lin_fcn} has several very useful consequences, which are crucial for the methods we use in this paper. In this way the MVM structure is essential. In the following, let $\xi$ be a continuous MVM.
\begin{enumerate}
\item\label{rem:MVMCompact_1} If $\xi_0$ lies in $\Pcal_p$ for some $p \in [1,\infty)$ then, with probability one, so does $\xi_t$ for all $t\ge0$, and the trajectories of $\xi$ are continuous in $\Pcal_p$. To see this, apply Lemma~\ref{L_lin_fcn} with $\varphi(x) = |x|^p$.
\item\label{it2iii} Any continuous MVM $\xi$ has decreasing support in the sense that, with probability one,
\begin{equation}\label{eq_supp_cont}
\text{$\supp(\xi_t)\subseteq\supp(\xi_s)$ whenever $t\ge s$.}
\end{equation}
To see this, let $\Ical$ be the countable collection of all open  balls in $\R^d$ with rational centre and radius, and define $\Ical(\mu)=\{I\in\Ical\colon \mu(I)=0\}$ for $\mu\in\Pcal$. Then $\supp(\mu)=\R^d\setminus\bigcup_{I\in\Ical(\mu)} I$. Now, for every $I\in\Ical$, $\xi(I)$ is a nonnegative martingale that stops once it hits zero, at least off a nullset $N$ that does not depend on $I\in\Ical$. Therefore, off $N$, $\Ical(\xi_s)\subseteq\Ical(\xi_t)$ for all $s\le t$. This yields \eqref{eq_supp_cont}.
%\item  If $\xi_0$ lies in $\Pcal_p(E)$ for some $p\geq0$ then, with probability one, so does $\xi_t$ for all $t\ge0$. This follows from the previous two points.

\item\label{dlvp} (De~la~Vall\'ee-Poussin) For each $a>0$ and each $\varphi:\R^d\to \R$ given by
  $\varphi(x):=G(|x|)$ for some measurable function $G:\R_+\to\R_+$ with
  $\lim_{t\to\infty}G(t)/t^p=\infty$ the set 
\begin{equation}\label{eqn5}
K^\varphi_a:=\{\mu\in \Pcal_p\colon\mu(\varphi)\leq a\}
\end{equation}
is compact in $\Pcal_p$. Moreover, for each compact set $K\subset \Pcal_p$ there is a function $\varphi$ as before such that $K\subseteq K_a^\varphi$ for some $a>0$.

We provide a few details about these results. By Prohorov's theorem %(see for instance Theorem 13.29 in \cite{K:13} for $p=0$)
 we know that a closed set $K\subseteq \Pcal_p$ is compact if and only if for each $\e>0$ there is a compact set $C\subset \R^d$ such that $\int_{{\R^d}\setminus C} |x|^p\mu(\dx)<\e$ for all $\mu\in K$. The criterion of de~la~Vall\'ee-Poussin then states that this condition is satisfied if and only if there is a function $\varphi$ as before such that 
$$\sup\{\mu(\varphi)\colon \mu\in K\} < \infty.$$
In this case one can choose the function $G$ to be continuous. Since $K^\varphi_a$ is closed for each $a>0$ by the monotone convergence theorem, the claim follows.
% In particular, observe that if $E$ is bounded, every closed subset of $\Pcal_p$ is a compact set.

\item\label{page:mvm_localise_compact} MVMs can be localised in compact sets. More specifically, if $\xi$ is a continuous MVM starting at $\xi_0=\bar\mu\in\Pcal_p$, 
Remark~\ref{rem:MVMCompact}\ref{dlvp} (De~la~Vall\'ee-Poussin)
 gives a  measurable function $\varphi:\R^d\to\R_+$ such that $\bar\mu(\varphi) < \infty$ and the set
$
K^\varphi_n 
$
given by \eqref{eqn5} is a compact subset of $\Pcal_p$ for each $n\in \N$. With $\tau_n=\inf\{t\ge0\colon\xi_t(\varphi)\ge n\}$ we have $\xi_t\in K_n^\varphi$ for all $t<\tau_n$, and since $\xi(\varphi)$ is a continuous process by Lemma~\ref{L_lin_fcn}, we have that $\xi_{\tau_n}\in K_n^\varphi$ for each $n$ and $\tau_n\to\infty$ as $n\to\infty$.
\end{enumerate}
\end{remark}

In this paper we are interested in MVMs driven by a single Brownian motion. More specifically, our goal is to consider optimal control problems where the controlled state is an MVM $\xi$ given as a weak solution of the equation
\begin{equation}\label{eq_MVM_SDE}
\xi_t(\varphi) = \xi_0(\varphi) + \int_0^t \Cov_{\xi_s}(\varphi,\rho_s) \d W_s \quad \text{for all } \varphi \in C_b
\end{equation}
in a sense to be made precise below, where $\rho$ is a progressively measurable function acting as the control.

\begin{remark} \label{R_prog_meas_fcn}
A \emph{progressively measurable function} from $\R^d$ to $\R$ on a filtered measurable space $(\Omega, \Fcal, (\Fcal_t)_{t \ge 0})$ is a map $\rho \colon \Omega \times \R_+ \times \R^d$ that is $\mathfrak P \otimes \Bcal(\R^d)$-measurable, where $\mathfrak P$ is the $\sigma$-algebra on $\Omega\times\R_+$ generated by all progressively measurable processes, and $\Bcal(\R^d)$ is the Borel $\sigma$-algebra on $\R^d$.
\end{remark}

\begin{remark}
Although we will not use it directly in this paper, let us indicate how this type of MVM can be derived from first principles.
%The  properties above will be key in allowing us to prove many of our results, and w
%We will use these technical properties later in the paper to prove many of our main results. We can also use these properties to characterise a natural class of controls, as follows.
%\comment{AC: We should emphasise here that this is a key property of MVMs that allows us to do this, which we might not be able to do in general!\\ AC: Done?!}
%In Brownian filtrations---or more generally, filtrations where a martingale representation theorem holds---one can say more about the structure of MVMs. 
Suppose $\xi$ is an MVM on a space whose filtration is generated by a Brownian motion $W$. For any $\varphi\in C_b$, the martingale representation theorem yields
\begin{equation}\label{eq_mg_repr}
\xi_t(\varphi) = \xi_0(\varphi) + \int_0^t \sigma_s(\varphi) \d W_s
\end{equation}
for some progressively measurable process $\sigma(\varphi)$ with $\int_0^t \sigma_s(\varphi)^2  \d s<\infty$ for all $t$. % With some work it is possible to construct a single process $\sigma=(\sigma_t)_{t\ge0}$, taking values among the signed measures, such that the integrand appearing above is indeed $\sigma_t(\varphi)=\int \varphi(x) \sigma_t(\dx)$. \fbox{reference for the construction.} Note that since $\xi_t(1)=1$, we have $\sigma_t(1)=0$.
In the context of filtration enlargement, \cite{yor1985,yor2012} observed that in various cases of interest one has $\sigma_t(\varphi)=\int \varphi(x) \sigma_t(\dx)$ for a single process $\sigma=(\sigma_t)_{t\ge0}$ that takes values among the signed measures and admits a progressively measurable function $\rho_t(\omega,x)$ such that
\[
\text{$\sigma_t(\varphi)=\xi_t(\varphi \rho_t)-\xi_t(\varphi)\xi_t(\rho_t)$ for all $\varphi\in C_b$.}\footnote{Subtracting $\xi_t(\varphi) \xi_t(\rho_t)$ ensures that $\sigma_t(1) = 0$. Equivalently, by replacing $\rho_t(x)$ by $\tilde\rho_t(x)=\rho_t(x)-\xi_t(\rho_t)$, one gets $\sigma_t(\varphi)=\xi_t(\varphi \tilde\rho_t)$ and $\xi_t(\tilde\rho_t)=0$. This is for instance done in \cite{mansuy2006}; see the table of p.~34. We find it more convenient not to use this convention, in order to avoid the constraint $\xi_t(\tilde\rho_t)=0$.}
\]
% M.~Yor observed that in various cases of interest one has $\sigma_t(\varphi)=\int \varphi(x) \sigma_t(\dx)$ for a single process $\sigma=(\sigma_t)_{t\ge0}$ that takes values among the signed measures and satisfies $\sigma_t(\dx)\ll\xi_t(\dx)$ and $\int_0^t | \sigma_s |_\text{TV} ds < \infty$ for all $t$; see \cite{yor1985,yor2012}, where the context is filtration enlargement. In such situations the Radon--Nikodym theorem applied on $\mathfrak P \otimes \Bcal(\R^d)$ yields a progressively measurable function $\rho_t(\omega,x)$ such that, outside a $\P\otimes \dt$-nullset, $\xi_t( |\rho_t| ) < \infty$ and
% \[
% \text{$\sigma_t(\varphi)=\xi_t(\varphi \rho_t)-\xi_t(\varphi)\xi_t(\rho_t)$ for all $\varphi\in C_b$.}\footnote{Subtracting $\xi_t(\varphi) \xi_t(\rho_t)$ ensures that $\sigma_t(1) = 0$. Equivalently, by replacing $\rho_t(x)$ by $\tilde\rho_t(x)=\rho_t(x)-\xi_t(\rho_t)$, one gets $\sigma_t(\varphi)=\xi_t(\varphi \tilde\rho_t)$ and $\xi_t(\tilde\rho_t)=0$. This is for instance done in \cite{mansuy2006}; see the table of p.~34. We find it more convenient not to use this convention, in order to avoid the constraint $\xi_t(\tilde\rho_t)=0$.}
% \]
%With the notation
%\begin{equation}\label{eq:CovDefn}
%\Cov_\mu(\varphi,\psi) = \mu(\varphi\psi) - \mu(\varphi)\mu(\psi),
%\end{equation}
Equation \eqref{eq_mg_repr} then takes the form \eqref{eq_MVM_SDE}.
%\begin{equation}\label{eq_MVM_SDE}
%\xi_t(\varphi) = \xi_0(\varphi) + \int_0^t \Cov_{\xi_s}(\varphi,\rho_s) \d W_s.
%\end{equation}
%Our goal in this paper is to consider optimal control problems where the controlled state is an MVM $\xi$, given as a weak solution of \eqref{eq_MVM_SDE} in a sense to be made precise below, with the progressively measurable function $\rho$ acting as the control.

Let us finally mention a condition introduced by \cite{jacod1985}, also in the context of filtration enlargement: $\xi_t(\dx)\ll \xi_0(\dx)$. Under this condition there is a progressively measurable function $f_t(\omega,x)$ such that $\xi_t(\varphi)=\xi_0(\varphi f_t)$ and for every $x$, $f_t(x)$ is a martingale \cite[Lemma~1.8]{jacod1985}. In a Brownian filtration one then has a representation $f_t(x)=1+\int_0^t f_s(x) \tilde\rho_s(x) \d W_s$ for some progressively measurable function $\tilde\rho_t(x)$ \cite[Proposition~3.14]{jacod1985}. Under suitable integrability conditions it follows that Jacod's condition implies Yor's condition. Indeed, multiplying by $\varphi(x)$, integrating against $\xi_0(\dx)$, applying the stochastic Fubini theorem, and comparing with \eqref{eq_mg_repr}, one finds that $\sigma_t(\varphi)=\xi_t(\varphi \tilde\rho_t)$.
\end{remark}

%\begin{remark}
%Although we will not use it in this paper, let us mention a condition introduced by \cite{jacod1985}, also in the context of filtration enlargement: $\xi_t(\dx)\ll \xi_0(\dx)$. Under this condition there is a progressively measurable function $f_t(\omega,x)$ such that $\xi_t(\varphi)=\xi_0(\varphi f_t)$, and in Brownian filtrations one has a representation $f_t(x)=1+\int_0^t f_s(x) g_s(x) \d W_s$. Without going into detail, multiplying by $\varphi(x)$, integrating against $\xi_0(\dx)$, applying the stochastic Fubini theorem, and comparing with \eqref{eq_mg_repr}, one finds that $\sigma_t(\varphi)=\xi_t(\varphi g_t)$. Thus Jacod's condition implies Yor's condition.
%\end{remark}

\section{Control problem and dynamic programming}
\label{sec:contr-probl-dynam}

Let us first define what we mean by a weak solution of \eqref{eq_MVM_SDE}.

\begin{definition}\label{def:sol_of_SDE}
A \emph{weak solution} of \eqref{eq_MVM_SDE} is a tuple $(\Omega,\Fcal,(\Fcal_t)_{t\ge0},\P,W,\xi,\rho)$, where $(\Omega,\Fcal,(\Fcal_t)_{t\ge0},\P)$ is a filtered probability space, $W$ is a standard Brownian motion on this space, $\xi$ is a continuous MVM, and $\rho$ is a progressively measurable function on $\Omega\times\R_+\times {\R^d}$ (see Remark~\ref{R_prog_meas_fcn}) such that for every $\varphi\in C_b$, $\P\otimes \d t$-a.e.,
\[
\xi_t(|\rho_t|)<\infty, \quad \int_0^t \Cov_{\xi_s}(\varphi,\rho_s)^2\d s<\infty,
\]
and \eqref{eq_MVM_SDE} holds, that is,
\[
\xi_t(\varphi) = \xi_0(\varphi) + \int_0^t \Cov_{\xi_s}(\varphi,\rho_s) \d W_s \quad \text{for all } \varphi \in C_b.
\]
To simplify terminology, we often call $(\xi,\rho)$ a weak solution, without explicitly mentioning the other objects of the tuple.
\end{definition}

We are interested in a specific class of \emph{controlled} MVMs, specified as follows. Fix $p \in [1,\infty) \cup \{0\}$, $q \in [1,p] \cup \{0\}$, and a Polish space $\Hb$ of measurable real-valued functions on $\R^d$, the set of actions. We make the standing assumption that the evaluation map $(\rho,x) \mapsto \rho(x)$ from $\mathbb H \times \R^d$ to $\R$ is measurable. This ensures that any $\Hb$-valued progressively measurable process is also a progressively measurable function, a property which is used in the proof of the dynamic programming principle in Section~\ref{app:DPP}.
%map $\Hb\times\Pc_p\ni(\rho,\xi)\mapsto\xi(|\rho|)\in[0,\infty]$ is measurable.
The role of the parameter $p$ will be to specify the state space $\Pcal_p$ of the controlled MVMs, while $q$ will be related to the set of test functions used in the definition of viscosity solution in Section~\ref{S_viscosity}.

\begin{definition}\label{def:adm_control}
An \emph{admissible control} is a weak solution $(\xi,\rho)$ of \eqref{eq_MVM_SDE} such that
\[
\rho_t(\fdot,\omega)\in\Hb
\]
and, $\P\otimes \d t$-a.e., 
\begin{equation}\label{eqn3}
\int_0^t \left( \int_{\R^d} (1 + |x|^q) |\rho_s(x) - \xi_s(\rho_s)| \xi_s(\d x) \right)^2 \d s < \infty.
\end{equation}
\end{definition}

Condition \eqref{eqn3} will later on enable us to apply our It{\^o} formula to any admissible control; here is a sufficient condition for it to hold.

\begin{lemma}\label{rem1}
Fix $r\in[0, p-q]$ and suppose that for each $\rho \in \mathbb H$ there is a constant $c$ such that $\rho(x)\leq c(1+|x|^r)$.  Then \eqref{eqn3} holds for any weak solution $(\xi,\rho)$ of \eqref{eq_MVM_SDE} such that $\xi_0 \in \Pcal_p$ and $\rho_t(\fdot,\omega)\in\Hb$.
\end{lemma}
\begin{proof}
Note that $\xi$ takes values in $\Pcal_p$ thanks to Remark~\ref{rem:MVMCompact}\ref{rem:MVMCompact_1}. Observe that $\P\otimes \d s$-a.e.
\begin{align*}
& \int_{\R^d} (1 + |x|^q) |\rho_s(x) - \xi_s(\rho_s)| \xi_s(\dx) \\
&\qquad \leq C \bigg(\int_{\R^d} (1+|x|^{q+r}) \xi_s(\dx)+\int_{\R^d} (1+|x|^{q}) \xi_s(\dx)\int_{\R^d} (1+|x|^{r}) \xi_s(\dx)\bigg),
\end{align*}
for some $C\geq0$. Since $s\mapsto\int_{\R^d} (1+|x|^{m}) \xi_s(\dx)$ is a continuous map for each $m\leq p$, condition \eqref{eqn3} follows.
\end{proof}

We consider the following control problem. In addition to the action space $\Hb$, fix a measurable cost function
\[
c\colon\Pcal_p\times\Hb\to\R \cup \{+\infty\}
\]
and a discount rate $\beta\ge0$. The value function is given by
\begin{equation}\label{eq_value_function}
v(\mu) = \inf\left\{  \E\left[ \int_0^\infty e^{-\beta t} c(\xi_t,\rho_t) \dt \right] \colon \text{$(\xi,\rho)$ admissible control, $\xi_0=\mu$}\right\}
\end{equation}
for every $\mu\in\Pcal_p$. Note that the value function depends on $\Hb$ through the definition of admissible control. Because $\xi_0=\mu$ lies in $\Pcal_p$, so does $\xi_t$ for all $t$. Thus $c(\xi_t,\rho_t)$ is well-defined. We will also want to ensure that the control problem is itself well-defined, in the sense that the expectation appearing in the expression above is well-defined for all admissible controls. To ensure this, we assume that
\begin{equation}\label{eqn9}
\int_0^\infty e^{-\beta t} \E\left[c(\xi_t,\rho_t)_-\right] \dt<\infty
\end{equation}
holds for every admissible control $(\xi, \rho)$, where $x_- = \max\{0, -x\}$ denotes the negative part of $x$.
This is trivially true if we suppose that $c(\xi,\rho)$ is bounded below. More generally, if there exists a non-negative, uniformly integrable martingale $M_t$ such that $c(\xi_t,\rho_t)_- \le M_t$, then \eqref{eqn9} is satisfied.

\begin{remark}
It would be natural to assume that $c(\mu,\rho)=c(\mu,\rho')$ for any $\mu\in\Pcal_p$ and $\rho,\rho'\in\Hb$ such that $\rho-\rho'$ is constant on $\supp(\mu)$. This is natural because  equation \eqref{eq_MVM_SDE} cannot detect any difference between $\rho$ and $\rho'$, since $\Cov_\mu(\varphi,\rho)=\Cov_\mu(\varphi,\rho')$. It is then reasonable that two such controls should produce the same cost. Our arguments do not require this assumption however, so we do not impose it. %, and there are examples where this additional flexibility is convenient.
\end{remark}

\begin{remark}\label{rem2}
In view of Lemma~\ref{rem1} a natural choice for the set $\Hb$ arising in applications is 
$$\Hb:=\{\rho\in C(\R^d)\colon \rho(x)\leq c(1+|x|^r)\}$$
for some fixed $c>0$ and $r\in[0,p-q]$. In some of our applications it will however be convenient to include an additional \emph{state-dependent} constraint on the controls. Specifically, it would be desirable to assume in addition that the control $\rho_t$ belongs to $\Hb(\xi_t)$, a state-dependent subset of $\Hb$. Instances in this sense are $\Hb(\mu):=\{\rho\in \Hb\colon \Var_\mu(\rho)\leq \Var(\mu)\}$ (see Example~\ref{ex92}) or $\Hb(\mu)=\left\{\rho\in\Hb:\Cov_{\mu}(\mathrm{id},\rho) \in (1-\kappa,1+\kappa) \right\}$ (see Section~\ref{sec:SEP}).
Rather than formulating this condition directly in the definition of an admissible strategy we enforce the state dependence in a weak formulation. Specifically, suppose there is a set $A \subseteq \Pcal_p \times \Hb$ which we wish our process and the corresponding control to remain within, for example, $A = \{(\xi,\rho)\colon \xi \in \Pcal_p, \rho \in \Hb(\xi)\}$. Then it is natural to only optimise over solutions for which $\int_0^\infty \bm1_{\{(\xi_t,\rho_t) \in A^{\complement} \}} \dt = 0$ almost surely. This can be achieved in the existing framework by ensuring that the cost function $c$ takes the value $+\infty$ on the set $A^{\complement}$. In the subsequent arguments, we will allow cost functions of this form, although our main assumptions will impose some properties on $A$ (typically that $A$ is open).
\end{remark}

\begin{remark}
As noted in the introduction, our class of MVMs appears to most closely relate to problems of controlled, partially observed diffusions, however the connections are largely conceptual, rather than exact. We here comment on these connections.

One key property of the class of measure-valued processes that we study is that the support of the process will always decrease. In this sense, the class of processes that we consider certainly does not include the full class of processes that arise in partially observed filtering, where, if the process is known to start inside some interval, the posterior measure will in general not be confined to that interval. See, for example, \cite{fabbri_stochastic_2017} for a discussion of partially observed control problems.

However, the class of problems we consider can be interpreted as a type of controlled observation process, where there is a constant signal $Y$ which is being observed with some noise, and has initial prior $\xi_0$, say. A typical filtering problem might then observe a signal process $Z_t = \int_0^t h(Y) \, ds + W_t$, for some (independent) Brownian motion $W$. The MVM $\xi$ would then be defined as the current posterior measure $\xi_t :=\mathrm{Law} (Y|\mathcal{F}^Z_t)$, where $\mathcal{F}^Z$ is the filtration generated by $Z$. In the context of partial observation, we believe that our MVM problems correspond to examples where the function $h$ may be allowed to depend on some control variable, the choice of which may incur some cost depending on the current posterior belief of the true state. In addition, our running cost $c$ may depend in a non-trivial way on both the control and the posterior measure, in a way that is far more general than in e.g. \cite{fabbri_stochastic_2017} and \cite{bandini_randomized_2019}. For example, our cost function permits control problems where the cost depends on the variance of the posterior measure, and more generally may be a non-linear function of the current posterior measure. We note that by the MVM property, in the case where the cost $c$ is independent of the control and linear in the measure, $c(\mu, \rho) = \mu(\tilde c)$ for some $\tilde c$, the problem degenerates completely for then $\E[c(\xi_t)] = \xi_0(\tilde c)$ by the martingale property, and the optimisation problem becomes trivial.

In most of the literature on controlled partially observed diffusions, control is allowed only in the behaviour of the diffusion, so the overlap between our problem and the problems considered in these parts of the partial observation literature are essentially only trivial cases where the control has no impact on the problem. In a limited number of papers, e.g. \cite{bandini_backward_2018}, some control of the observations are allowed. Here our results would potentially overlap with their setting under the assumption that the controlled process $Y$ is constant. The most general version of this approach that we are aware of appears in the book \cite{nisio_book_2015}, which covers examples where there may be overlap with the control problems we consider. However our results are not directly applicable. Our approach works directly with probability measures; Nisio works on a Sobolev space of (unnormalised) density functions, and we do not restrict our state process in this manner. Our setting also includes cases where there may be no dominating probability measure, and thus no regularity requirements on the densities, which are crucial to Nisio's approach.
\end{remark}

The following result states that the value function satisfies a dynamic programming principle. Let $C(\R_+,\Pcal_p)$ be the set of continuous functions from $\R_+$ to $\Pcal_p$. We say that $\tau$ is a \emph{stopping time on $C(\R_+,\Pcal_p)$} if $\tau\colon{C(\R_+,\Pcal_p)}\to\R_+$ is a stopping time with respect to the (raw) filtration generated by the coordinate process on $C(\R_+,\Pcal_p)$. In this case, for any admissible control $(\xi,\rho)$, $\tau(\xi)$ is a stopping time with respect to the filtration generated by the admissible control, where $\tau(\xi)$ is given by $\omega\mapsto\tau(\xi_\cdot(\omega))$. The proof of the following result is given in Appendix~\ref{app:DPP}.

\begin{theorem}\label{thm_dpp_main_text}
Let $\tau$ be a bounded stopping time on $C(\R_+,\Pcal_p)$. For any $\mu\in\Pcal_p$, the value function $v$ defined in \eqref{eq_value_function} satisfies
%$$
%\begin{aligned}
%v(\mu) =\inf\bigg\{&  \E\bigg[ e^{-\beta\tau(\xi)}v(\xi_{\tau(\xi)}) + \int_0^{\tau(\xi)} e^{-\beta t}c(\xi_t,\rho_t) \dt \bigg]\colon\\
%&\qquad  \text{$(\xi,\rho)$ admissible control, $\xi_0=\mu$}\bigg\}.
%\end{aligned}
%$$
$$v(\mu) = \inf_{(\xi,\rho)} \E\left[ e^{-\beta\tau(\xi)}v(\xi_{\tau(\xi)}) + \int_0^{\tau(\xi)} e^{-\beta t}c(\xi_t,\rho_t) \dt \right]$$
where the infimum extends over all admissible controls $(\xi,\rho)$ with $\xi_0=\mu$.
\end{theorem}

To ensure that the control problem \eqref{eq_value_function} is
nontrivial, we need to confirm that for any initial point
$\mu\in\Pcal_p$, there exists some admissible control. In the following
result, we prove this fact. 

\begin{theorem}\label{thm5}
For any measurable function $\bar\rho\colon {\R^d}\to\R$ and any $\mu\in\Pcal$, there exists a weak solution $(\xi,\rho)$ of \eqref{eq_MVM_SDE} such that $\xi_0=\mu$ and $\rho_t=\bar\rho$ for all $t$.
\end{theorem}

\begin{proof}
Let $\Omega=C(\R_+,\R)$ be the canonical path space of continuous functions. Let $X$ be the coordinate process, $\Fb$ the right-continuous filtration generated by $X$, $\Fc=\Fc_\infty$, and $\Q$ the Wiener measure. Thus $X$ is a standard Brownian motion under $\Q$. Let $\bar\rho\colon {\R^d}\to\R$ be a measurable function. For each fixed $x\in {\R^d}$, the process $\Ec(\bar\rho(x)X)$ is geometric Brownian motion and in particular a martingale. Define a strictly positive process $Z$ by
\[
Z_t = \int_{\R^d} \Ec(\bar\rho(x)X)_t \xi_0(\dx).
\]
This is finite, because
\begin{equation}\label{eq_P_existence_1_new_01}
\Ec(\bar\rho(x)X)_t=\exp\left(\bar\rho(x)X_t-\frac12 \bar\rho(x)^2 t \right) \le \exp\left(\frac{X_t^2}{2t} \right)
\end{equation}
for $t>0$, independently of $x$. We now define the desired process $\xi$ by
\[
\xi_t(\dx) = \frac{1}{Z_t}\Ec(\bar\rho(x)X)_t \xi_0(\dx).
\]
This is clearly probability measure valued, but it may not be an MVM. However, by replacing $\Q$ with another probability measure $\P$, we can turn $\xi$ into an MVM with the required properties. This is done in a number of steps.

\underline{Step~1.}
The conditional version of Tonelli's theorem gives
\[
\E_\Q[Z_t\mid\Fc_s] = \int_{\R^d} \E_\Q[\Ec(\bar\rho(x)X)_t\mid\Fc_s] \xi_0(\dx) = \int_{\R^d}  \Ec(\bar\rho(x)X)_s \xi_0(\dx) = Z_s
\]
for all $s\le t$. Thus $Z$ is a martingale with $Z_0=1$. For each $n\in\mathbb N$, define an equivalent probability $\P_n\sim \Q|_{\Fc_n}$ on $\Fc_n$ by using $Z_n$ as Radon--Nikodym derivative. The $\P_n$ are consistent in the sense that $\P_{n+1}|_{\Fc_n}=\P_n$ for all $n$, and we have $\Fc=\bigvee_{n\ge1}\Fc_n$. A standard argument now gives a probability measure $\P$ on $\Fc$ such that $\P|_{\Fc_n}=\P_n$ for all $n$; see \cite[Section~3.5A]{kar_shr_91}.

It is now clear that $\xi$ is an MVM under $\P$. Indeed, for $\varphi\in C_b$, the product $Z \xi(\varphi)=\int_{\R^d} \varphi(x)\Ec(\bar\rho(x)X) \xi_0(\dx)$ is a martingale under $\Q$. Therefore $\xi(\varphi)$ is a martingale under $\P$, showing that $\xi$ is an MVM.

\underline{Step~2.}
We claim that
\begin{equation}\label{eq_P_existence_1_new_001001}
\text{$\int_0^t \xi_s(|\bar\rho|)\ds<\infty$ for all $t$,}
\end{equation}
and that the process
\begin{equation}\label{eq_P_existence_1_new_001}
W_t = X_t - \int_0^t \xi_s(\bar\rho) \ds
\end{equation}
is a Brownian motion under $\P$. Suppose for now that \eqref{eq_P_existence_1_new_001001} holds. Integration by parts then gives
\begin{equation}\label{eq_P_existence_1_new_001004}
Z_t W_t = Z_t X_t - \int_0^t Z_s \xi_s(\bar\rho) \ds - \int_0^t \left( \int_0^s \xi_u(\bar\rho)\d u\right) \d Z_s.
\end{equation}
Moreover, integration by parts and the stochastic Fubini theorem \cite[Theorem~2.2]{ver_12} give
\begin{equation}\label{eq_P_existence_1_new_001005}
\begin{aligned}
Z_tX_t &= \int_{\R^d} X_t \Ec(\bar\rho(x)X)_t \xi_0(\dx) \\
&= \int_{\R^d} \int_0^t ( 1 + \bar\rho(x)X_s )\Ec(\bar\rho(x)X)_s \d X_s \xi_0(\dx) \\
&\quad + \int_{\R^d} \int_0^t \bar\rho(x) \Ec(\bar\rho(x)X)_s \ds \xi_0(\dx)\\
&= \int_0^t  \int_{\R^d} ( 1 + \bar\rho(x)X_s )Z_s\xi_s(\dx) \d X_s + \int_0^t Z_s \xi_s(\bar\rho) \ds,
\end{aligned}
\end{equation}
and the first term on the right-hand side is a local martingale under $\Q$. Note that the use of the stochastic Fubini theorem will be justified in the next step. Combining \eqref{eq_P_existence_1_new_001004} and \eqref{eq_P_existence_1_new_001005}, we conclude that $ZW$ is a local martingale under $\Q$. Thus $W$ is a local martingale under $\P$, hence Brownian motion under $\P$, as claimed.

\underline{Step~3.}
We must still prove \eqref{eq_P_existence_1_new_001001} and justify our use of the stochastic Fubini theorem. The latter amounts to checking that
\begin{equation}\label{eq_P_existence_1_new_2}
\int_{\R^d} \int_0^t |\bar\rho(x)|\Ec(\bar\rho(x)X)_s \ds \xi_0(\dx) < \infty
\end{equation}
and
\begin{equation}\label{eq_P_existence_1_new_3}
\int_{\R^d} \left( \int_0^t (1+\bar\rho(x)X_s)^2\Ec(\bar\rho(x)X)_s^2\ds\right)^{1/2} \xi_0(\dx) < \infty
\end{equation}
for all $t$. Then \eqref{eq_P_existence_1_new_001001} follows from \eqref{eq_P_existence_1_new_2} and the fact that $\inf_{s\in[0,t]}Z_s>0$ for all $t$. We now prove \eqref{eq_P_existence_1_new_2}. The elementary inequality
\[
|a| \exp\left( ab - \frac12 a^2s\right) \le \left(\frac{|b|}{s} + \frac{1}{s^{1/2}} \right) \exp\left( \frac{b^2}{2s} \right),
\]
valid for all $a,b\in\R$ and $s>0$, gives
\begin{equation}\label{eq_P_existence_1_new_6}
|\bar\rho(x)|\Ec(\bar\rho(x)X)_s \le \left(\frac{|X_s|}{s} + \frac{1}{s^{1/2}} \right) \exp\left( \frac{X_s^2}{2s} \right).
\end{equation}
The law of the iterated logarithm shows that for some $\delta\in(0,e^{-e})$ (depending on $\omega$), we have $|X_s| \le \sqrt{3s\log\log(1/s)}$ for all $s<\delta$. We use this bound to get
\begin{align*}
\int_0^\delta \left(\frac{|X_s|}{s} + \frac{1}{s^{1/2}} \right) \exp\left( \frac{X_s^2}{2s} \right) \ds &\le  \int_0^\delta 2 \sqrt{\frac1s\log\log\frac1s} \left(\log\frac1s\right)^{3/2} \ds \\
&= \int_{-\log\delta}^\infty 2 (\log s)^{1/2}  s^{3/2} e^{-s/2} \ds \\
&< \infty.
\end{align*}
Since the right-hand side of \eqref{eq_P_existence_1_new_6} is continuous on $[\delta,t]$, the integral over this interval is also finite. It follows that \eqref{eq_P_existence_1_new_2} holds.

We now verify \eqref{eq_P_existence_1_new_3}. From \eqref{eq_P_existence_1_new_01} and \eqref{eq_P_existence_1_new_6}, along with two applications of the inequality $(a+b)^2\le 2a^2+2b^2$, we get
\[
(1+\bar\rho(x)X_s)^2\Ec(\bar\rho(x)X)_s^2 \le  \left( 2 + \frac{4 X_s^4}{s^2} + \frac{4 X_s^2}{s} \right) \exp\left( \frac{X_s^2}{s} \right).
\]
Using the law of the iterated logarithm as above, we find that the integral of the right-hand side over $(0,t]$ is finite. Thus \eqref{eq_P_existence_1_new_3} holds.

\underline{Step~4.}
It remains to argue that \eqref{eq_MVM_SDE} holds. To this end, define the measure-valued process $\eta_t(\dx) = \Ec(\bar\rho(x)X)_t\xi_0(\dx)$. Thus in particular, $\xi_t(\dx)=\eta_t(\dx)/\eta_t(1)$. Pick any $\varphi\in C_b$ and $0<s\le t$. Using the stochastic Fubini theorem \cite[Theorem~2.2]{ver_12} we get
\begin{align*}
\eta_t(\varphi) - \eta_s(\varphi) &= \int_{\R^d} \varphi(x) \left( \Ec(\bar\rho(x) X)_t - \Ec(\bar\rho(x) X)_s \right) \xi_0(\dx) \\
&= \int_{\R^d}  \int_s^t \varphi(x)  \bar\rho(x) \Ec(\bar\rho(x)X)_u \d X_u \xi_0(\dx) \\
&= \int_s^t \int_{\R^d}  \varphi(x)  \bar\rho(x) \Ec(\bar\rho(x)X)_u \xi_0(\dx) \d X_u \\
&= \int_s^t \eta_u( \varphi \bar\rho )  \d X_u.
\end{align*}
The stochastic Fubini theorem is applicable because $\varphi$ is bounded and since by \eqref{eq_P_existence_1_new_6} it holds
\[
\int_{\R^d} \int_s^t \bar\rho(x)^2 \Ec(\bar\rho(x)X)_u^2 \du \xi_0(\dx)
\leq\sup_{u\in[s,t]}  \left(\frac{|X_u|}{u} + \frac{1}{u^{1/2}} \right)^2 \exp\left( \frac{X_u^2}{u} \right),
\]
which is finite since $s>0$.
%This is a consequence of \eqref{eq_P_existence_1_new_6}, which produces the bound
%\[
%\bar\rho(x)^2 \Ec(\bar\rho(x)X)_u^2 \le \sup_{u\in[s,t]}  \left(\frac{|X_u|}{u} + \frac{1}{u^{1/2}} \right)^2 \exp\left( \frac{X_u^2}{u} \right), \quad u\in[s,t],
%\]
%which is finite since $s>0$. 
An application of It\^o's formula now gives
\begin{align}
\xi_t(\varphi)-\xi_s(\varphi) &= \int_s^t \d\left( \frac{\eta_u(\varphi)}{\eta_u(1)}\right) \notag\\
&= \int_s^t (\xi_u(\varphi\bar\rho) - \xi_u(\varphi)\xi_u(\bar\rho))(\d X_u - \xi_u(\bar\rho)\du) \notag\\
&= \int_s^t \Cov_{\xi_u}(\varphi,\bar\rho)\d W_u, \label{eq_P_existence_1_new_30}
\end{align}
recalling the definition \eqref{eq_P_existence_1_new_001} of $W$. We now extend this to $s=0$. Observe that
\begin{align*}
\int_0^t \Cov_{\xi_u}(\varphi,\bar\rho)^2 \du &= \lim_{s\downarrow 0} \int_s^t \Cov_{\xi_u}(\varphi,\bar\rho)^2 \du \\
&= \lim_{s\downarrow 0} \Big( \langle \xi(\varphi) \rangle_t -  \langle \xi(\varphi) \rangle_s \Big) = \langle \xi(\varphi) \rangle_t < \infty,
\end{align*}
where we use that $\xi(\varphi)$ is a continuous process that we have already shown to be a martingale and we denote by $ \langle \xi(\varphi) \rangle$ its quadratic variation process. The dominated convergence theorem for stochastic integrals now allows us to send $s$ to zero in \eqref{eq_P_existence_1_new_30} and obtain \eqref{eq_MVM_SDE}.
\end{proof}

% \begin{corollary}
% Fix a measurable function $\bar\rho\colon {\R^d}\to\R$, a measure $\bar\mu\in\Pcal$, and let $(\xi,\rho)$ be the solution of \eqref{eq_MVM_SDE} with $\xi_0=\bar\mu$ and $\rho_t=\bar\rho$ given by Theorem~\ref{thm5}.
%   If the cost function $c$ is upper semi-continuous in its first argument and $c(\bar\mu,\bar{\rho}) < \infty$, then there exists a strictly positive stopping time $\tau$ such that
%   \begin{equation*}
%     \E\left[\int_0^\tau \me^{-\beta t} c(\xi_t, \bar{\rho}) \, \dt\right] < \infty.
%   \end{equation*}
% \end{corollary}
% \begin{proof}
%   This follows immediately from the fact that the open sub-level sets $\{ \mu \colon c(\mu, \bar\rho) < y\}$ of $c(\fdot,\bar\rho)$ are open for all $y \in \R$. Since the paths of solutions to the SDE \eqref{eq_MVM_SDE} are continuous, the exit time from a sufficiently small ball centred on $(\bar\mu,\bar{\rho})$ is strictly positive, while the ball can be chosen sufficiently small that $c(\xi_t,\bar{\rho})$ is bounded from above on this ball.
% \end{proof}

\section{Differential calculus}
\label{sec:diff-calc}

We now develop the differential calculus required to formulate It\^o's formula in Section~\ref{S_Ito} and the HJB equation in Section~\ref{S_viscosity}. The derivatives used here are essentially what is called \emph{linear functional derivatives} in \cite[Section~5.4]{car_del_18_I}.

\subsection{First order derivatives}

\begin{definition}\label{D_C1r}
Let $p \in [1,\infty) \cup \{0\}$. A function $f\colon\Pcal_p\to\R$ is said to belong to $C^1(\Pcal_p)$ if there is a continuous function $(x,\mu)\mapsto\frac{\partial f}{\partial\mu}(x,\mu)$ from  ${\R^d}\times\Pcal_p$ to $\R$, called (a version of) the \emph{derivative} of $f$, with the following properties.
\begin{itemize}
\item \textbf{locally uniform $p$-growth:} for every compact set $K\subseteq\Pcal_p$, there is a constant $c_K$ such that for all  $x\in {\R^d}$ and $\mu\in K$,
\begin{equation}\label{eq_C1r_p_growth}
\left| \frac{\partial f}{\partial\mu}(x,\mu) \right| \le c_K (1+|x|^p),
\end{equation}
\item \textbf{fundamental theorem of calculus:} for every $\mu,\nu\in\Pcal_p$,
\begin{equation}\label{eq_C1r_FTC}
f(\nu) - f(\mu) = \int_0^1 \int_{\R^d} \frac{\partial f}{\partial\mu}(x,t\nu + (1-t)\mu)(\nu-\mu)(\dx)\dt.
\end{equation}
\end{itemize}
\end{definition}

\begin{remark}
  This is called \emph{linear functional derivative} by \cite[Definition~5.43]{car_del_18_I} although they require the stronger property that \eqref{eq_C1r_p_growth} hold uniformly on \emph{bounded} rather than \emph{compact} subsets of $\Pcal_p$. This notion of derivative, including its second-order analogue, has long been used in the context of measure-valued processes, sometimes implicitly; see e.g.~\cite{MR542340,dawson1993}. Note that  if  $(x,\mu)\mapsto\frac{\partial f}{\partial\mu}(x,\mu)$ is a version of the derivative of $f$, then the same holds for $(x,\mu)\mapsto\frac{\partial f}{\partial\mu}(x,\mu)+a(\mu)$ for each continuous map  $\mu\mapsto a(\mu)$. Modulo additive terms of this form, the derivative is uniquely determined. Note also that if $f:\Pcal_p\to\R$ belongs to $C^1(\Pcal_p)$, it is automatically continuous. For more details on these properties see Appendix~\ref{secC}.
 \end{remark}

\begin{remark} \label{R_C1Pq_in_C1Pp}
If $q < p$, then $C^1(\Pcal_q) \subset C^1(\Pcal_p)$ in the sense that if $g \in C^1(\Pcal_q)$ and $f$ is the restriction of $g$ to $\Pcal_p$, then $f \in C^1(\Pcal_p)$ and $\frac{\partial f}{\partial\mu}(x,\mu)=\frac{\partial g}{\partial\mu}(x,\mu)$. Indeed, the restriction is well-defined because $\Pcal_p \subset \Pcal_q$. Moreover, the topology on $\Pcal_p$ is stronger than that on $\Pcal_q$, so $(x,\mu)\mapsto\frac{\partial g}{\partial\mu}(x,\mu)$ remains continuous on ${\R^d}\times\Pcal_p$. If $K$ is compact in $\Pcal_p$ it is also compact in $\Pcal_q$, and a $q$-growth bound implies a $p$-growth bound. This gives the locally uniform $p$-growth condition. The fundamental theorem of calculus carries over as well, as it is now only required for $\mu,\nu$ in the smaller set $\Pcal_p$.
\end{remark}

Consider a function $f$ of the form
\begin{equation}\label{eq_C1_cylinder}
f(\mu)=\tilde f(\mu(\varphi_1),\ldots,\mu(\varphi_n)),
\end{equation}
where $n\in\N$, $\tilde f\in C^1(\R^n)$, and $\varphi_1,\ldots,\varphi_n\in C_b({\R^d})$. We refer to such a function as a \emph{$C^1$ cylinder function}. A version of its derivative is
\begin{equation}\label{eq_C1r_der_cyl_ex}
\frac{\partial f}{\partial\mu}(x,\mu) = \sum_{i=1}^n \partial_i\tilde f(\mu(\varphi_1),\ldots,\mu(\varphi_n)) \varphi_i(x),
\end{equation}
where $\partial_i\tilde f$ denotes partial derivative with respect to the $i$-th variable.

Any $C^1$ cylinder function belongs to $C^1(\Pcal_p)$ for every $p$. The following result gives a kind of approximate converse: every function belonging to $C^1(\Pcal_p)$ can be approximated by $C^1$ cylinder functions. This is crucial in our proof of the It\^o formula.

\begin{theorem}\label{T_approx_cyl}
Let $f \in C^1(\Pcal_p)$ for some $p \in [1,\infty) \cup \{0\}$. Then there exist $C^1$ cylinder functions $f_n$ such that one has the pointwise convergence
\begin{equation}\label{T_approx_cyl_1}
f_n(\mu) \to f(\mu) \quad\text{and}\quad \frac{\partial f_n}{\partial\mu}(x,\mu) \to \frac{\partial f}{\partial\mu}(x,\mu)
\end{equation}
for all $\mu \in \Pcal_p$, $x \in {\R^d}$, and for every compact set $K\subset\Pcal_p$ there is constant $c_K$ such that
\begin{equation}\label{T_approx_cyl_2}
|f_n(\mu)| \le c_K \quad\text{and}\quad \left|  \frac{\partial f_n}{\partial\mu}(x,\mu) \right| \le c_K (1 + |x|^p)
\end{equation}
for all $\mu\in K$, $x \in {\R^d}$, $n\in\N$.
\end{theorem}

The proof relies on the following construction, which leads to a useful way of `discretising' probability measures in $\Pcal_p$. Fix $n\in \N$, and cover the compact ball  $B_n:=\{x\in {\R^d}\colon |x|\leq n\}$ by finitely many open sets of diameter at most $1/n$, denoted by $U^n_i$, $i=1,\ldots,N_n$. Append $U^n_0={\R^d}\setminus B_n$ to get an open cover of ${\R^d}$. Finally, fix points $x^n_i$ in $\overline{U^n_i}$ with minimal norm. We have achieved that
\begin{equation}\label{eq_part_unit_diam}
\text{$\diam(U^n_i) \le \frac1n$, $i=1,\ldots,N_n$}
\end{equation}
and
\begin{equation}\label{eq_part_unit_p_growth}
\text{$|x^n_i| \le |x|$ for all $x\in U^n_i$, $i=0,\ldots,N_n$.}
\end{equation}
Now let $\{\psi^n_i\}$ be a partition of unity subordinate to $\{U^n_i\}$: that is, each $\psi_i^n$ is a continuous function, supported on $U_i^n$, and such that $\sum_{i=0}^{N_n} \psi_i^n(x) = 1$ for all $x \in {\R^d}$. For any function $\varphi$ on $\R^d$, define a new function $T_n\varphi$ by
\[
T_n\varphi(x) = \sum_{i=0}^{N_n} \varphi(x^n_i)\psi^n_i(x).
\]
Observe that $T_n\varphi$ is always continuous. Moreover, taking $\varphi(x)=h(|x|)$ for any nonnegative increasing function $h$, we have from \eqref{eq_part_unit_p_growth} that
\begin{equation}\label{eq_Tn_growth_pres}
T_n\varphi(x) = \sum_{i=0}^{N_n} h(|x^n_i|) \psi^n_i(x) \le \sum_{i=0}^{N_n} h(|x|) \psi^n_i(x) = \varphi(x).
\end{equation}
In particular, if $\varphi$ satisfies a $p$-growth bound on ${\R^d}$ of the form $|\varphi(x)|\le c(1+|x|^p)$, it follows that $T_n\varphi$ satisfies the same bound.

The operator $T_n$ admits an `adjoint' $T_n^*$ that acts on probability measures by the formula
\[
T_n^*\mu = \sum_{i=0}^{N_n} \mu(\psi^n_i) \delta_{x^n_i}.
\]
Note that $T_n^*\mu$ is again a probability measure. The terminology and notation is motivated by the identity
\begin{equation}\label{eq_Tn_adjoint}
\mu(T_n\varphi) = \sum_{i=0}^{N_n}\varphi(x^n_i) \mu(\psi^n_i) = (T_n^*\mu)(\varphi).
\end{equation}
In particular, applying this with $\varphi(x)=|x|^p$ and using \eqref{eq_Tn_growth_pres} shows that $T_n^*$ maps $\Pcal_p$ to itself.

\begin{lemma}\label{L_Tn_prop}
The operators $T_n$ satisfy the following basic properties.
\begin{enumerate}
\item\label{L_Tn_prop_a1} if $K\subset\Pcal_p$ is a compact set, one can find another compact set $K'\subset\Pcal_p$, containing $K$, such that $T_n^*$ maps $K'$ into itself for all $n$,
\item\label{L_Tn_prop_a2} if $h\colon\R\to\R$ is a convex function, then $h \circ (T_n\varphi) \le T_n(h\circ\varphi)$,
\item\label{L_Tn_prop_0} if $|x|\le n$, then $|T_n\varphi(x)| \le \sup\{|\varphi(y)|\colon y\in {\R^d}, |x-y|<1/n\}$,
\item\label{L_Tn_prop_1} if $\varphi$ is continuous at $x\in {\R^d}$, then $T_n\varphi(x)\to\varphi(x)$,
\item\label{L_Tn_prop_2} if $\varphi$ is continuous everywhere, then $T_n\varphi\to\varphi$ locally uniformly,
\item\label{L_Tn_prop_3} if $\varphi_n\to\varphi$ locally uniformly and $\varphi$ is continuous at $x\in {\R^d}$, then $T_n\varphi_n(x)\to\varphi(x)$,
\item\label{L_Tn_prop_4} $T_n^*\mu \to \mu$ in $\Pcal_p$ for every $\mu\in\Pcal_p$.
\end{enumerate}
\end{lemma}

\begin{proof}
\ref{L_Tn_prop_a1}: We apply Remark~\ref{rem:MVMCompact}\ref{dlvp}. If $K$ is compact, then there exists a positive increasing function $h$ with $\lim_{t\to\infty}h(t)/(1+t^p)=\infty$ such that the constant $c=\sup_{\mu\in K}\mu(\varphi)$ is finite, where $\varphi(x)=h(|x|)$. The set $K'=\{\mu\in\Pcal_p\colon \mu(\varphi)\le c\}$ is then compact and contains $K$. Moreover, \eqref{eq_Tn_adjoint} and \eqref{eq_Tn_growth_pres} yield $(T_n^*\mu)(\varphi)=\mu(T_n\varphi)\le\mu(\varphi)$, which shows that $T_n$ maps $K'$ into itself.

\ref{L_Tn_prop_a2}: By definition of partition of unity, $(\psi^n_0(x),\ldots,\psi^n_{N_n}(x))$ forms a vector of probability weights for any fixed $x\in {\R^d}$. Thus by Jensen's inequality, 
$$h(T_n\varphi(x))\le\sum_{i=0}^{N_n}h(\varphi(x^n_i))\psi^n_i(x)=T_n(h\circ\varphi)(x).$$

\ref{L_Tn_prop_0}: If $x\in B_n$ then $x\in U^n_i$ for some $i\ne0$. These sets all have diameter at most $1/n$, so
$$|T_n\varphi(x)|\le\sum_{i=1}^{N_n} |\varphi(x^n_i)|\psi^n_i(x)\le\sup\{|\varphi(y)|\colon y\in {\R^d}, |x-y|<1/n\}.$$

\ref{L_Tn_prop_1}: Let $\omega_x(\delta)$ be an increasing modulus of continuity for $\varphi$ at $x$. Then $|\varphi(x^n_i)-\varphi(x)|\le\omega_x(|x^n_i-x|)\le\omega_x(n^{-1})$ whenever $x$ lies in $U^n_i$ and $i\ne0$. Because $x\notin U^n_0$ for all large $n$, it follows that
\[
|T_n\varphi(x) - \varphi(x)| \le \sum_{i=1}^{N_n} |\varphi(x^n_i) - \varphi(x)|\psi^n_i(x) \le \omega_x(n^{-1}) \to 0.
\]

\ref{L_Tn_prop_2}: Fix a compact set $J\subset {\R^d}$ and let $\omega(\delta)$ be a uniform modulus of continuity for $\varphi$ on $J$. Because $J$ and $U^n_0$ are disjoint for all large $n$, the same computation as above gives $|T_n\varphi(x) - \varphi(x)| \le \omega(n^{-1})$ for all $x\in J$.

\ref{L_Tn_prop_3}: Write $|T_n\varphi_n(x)-\varphi(x)| \le |T_n(\varphi_n-\varphi)(x)| + |T_n\varphi(x)-\varphi(x)|$, and denote the two terms on the right-hand side by $A_n$ and $B_n$, respectively. We have from \ref{L_Tn_prop_0} that $A_n\le
\sup\{|\varphi_n(y)-\varphi(y)|\colon y\in {\R^d}, |x-y|<1/n\}
$ for all large $n$, so that $A_n\to0$. Moreover, thanks to \ref{L_Tn_prop_1}, $B_n\to0$.

\ref{L_Tn_prop_4}: Applying \ref{L_Tn_prop_a1} with $K=\{\mu\}$ shows that the sequence $\{T_n^*\mu\colon n\in\N\}$ is relatively compact in $\Pcal_p$. Its only limit point is $\mu$, because \ref{L_Tn_prop_1} and the bounded convergence theorem yield $(T_n^*\mu)(\varphi)=\mu(T_n\varphi)\to\mu(\varphi)$ for all $\varphi\in C_b$.
\end{proof}

\begin{lemma}\label{L_C1r_der_fn}
Suppose $f$ belongs to $C^1(\Pcal_p)$ and define $f_n(\mu)=f(T_n^*\mu)$. Then $f_n$ is a $C^1$ cylinder function, and a version of its derivative is given by
\begin{equation}\label{eq_L_C1r_der_fn_1}
\frac{\partial f_n}{\partial\mu}(x,\mu) = T_n \frac{\partial f}{\partial\mu}(\fdot,T_n^*\mu)(x).
\end{equation}
\end{lemma}

\begin{proof}
We first show that $f_n$ is a $C^1$ cylinder function. To this end, write $f_n(\mu) = \tilde f(\mu(\psi^n_0),\ldots,\mu(\psi^n_{N_n}))$, where we define
\begin{equation}\label{eq_L_C1r_der_fn_ftilde}
\tilde f(p) = f(p_0\delta_{x^n_0} + \ldots + p_{N_n}\delta_{x^n_{N_n}})
\end{equation}
for all $p$ in the standard $N_n$-simplex $\Delta^{N_n}$ in $\R^{N_n+1}$ given by
\begin{equation}
\Delta^{N_n} = \{(p_0,\ldots,p_{N_n}) \in [0,1]^{N_n+1}\colon p_0+\cdots+p_{N_n}=1\}.
\end{equation}
We now argue that $\tilde f$ satisfies a fundamental theorem of calculus. Pick any $p,q\in\Delta^{N_n}$. Writing $\nu=p_0\delta_{x^n_0} + \ldots + p_{N_n}\delta_{x^n_{N_n}}$ and $\eta=q_0\delta_{x^n_0} + \ldots + q_{N_n}\delta_{x^n_{N_n}}$, and using that $f$ satisfies the fundamental theorem of calculus \eqref{eq_C1r_FTC} by assumption, we get
\begin{align*}
\tilde f(q)-\tilde f(p) &= f(\eta) - f(\nu) \\
&= \int_0^1 \int_{\R^d} \frac{\partial f}{\partial\mu}(x,t\eta + (1-t)\nu)(\eta-\nu)(\dx)\dt \label{eq_L_C1r_der_fn_2} \\
&= \int_0^1 \sum_{i=0}^{N_n}  \frac{\partial f}{\partial\mu}(x^n_i,t\eta + (1-t)\nu)(q_i-p_i)\dt \\
&= \int_0^1 \sum_{i=0}^{N_n} \partial_i\tilde f(tq+(1-t)p)(q_i-p_i)\dt,
\end{align*}
where we define $\partial_i\tilde f(p)=\frac{\partial f}{\partial\mu}(x^n_i,p_0\delta_{x^n_0}+\cdots+p_{N_n}\delta_{x^n_{N_n}})$. Since $f$ belongs to $C^1(\Pcal_p)$, the functions $\partial_i\tilde f$ are continuous on $\Delta^{N_n}$.
The above implies that $\tilde f$ is $C^1$ on $\Delta^{N_n}$ is the sense that the tangential derivatives exist and are uniformly continuous on the relative interior of $\Delta^{N_n}$. Using Whitney's extension theorem, see e.g.\ \citep[Appendix, Corollary~6.3]{MR838085}, we deduce that $\tilde f$ can be extended to a $C^1$ function on all of $\R^{N_n+1}$.
%This allows us to use Whitney's extension theorem to deduce that $\tilde f$ can be extended to a $C^1$ function on all of $\R^{N_n+1}$.
This confirms that $f_n$ is a $C^1$ cylinder function. To verify \eqref{eq_L_C1r_der_fn_1}, it now suffices to note that $\partial_i\tilde f(\mu(\psi^n_0),\ldots,\mu(\psi^n_{N_n}))=\frac{\partial f}{\partial\mu}(x^n_i,T_n^*\mu)$ and apply formula \eqref{eq_C1r_der_cyl_ex}.
\end{proof}

\begin{proof}[Proof of Theorem~\ref{T_approx_cyl}]
Take $f_n(\mu)=f(T_n^*\mu)$, which are $C^1$ cylinder functions due to Lemma~\ref{L_C1r_der_fn}. We need to verify \eqref{T_approx_cyl_1} and \eqref{T_approx_cyl_2}.

First, continuity of $f$ and Lemma~\ref{L_Tn_prop}\ref{L_Tn_prop_4} yield $f_n(\mu)=f(T_n^*\mu)\to f(\mu)$. Next, to simplify notation, write $g(x,\mu)=\frac{\partial f}{\partial\mu}(x,\mu)$ and $g_n(x,\mu)=\frac{\partial f_n}{\partial\mu}(x,\mu)$. Then for each fixed $\mu\in\Pcal_p$, Lemma~\ref{L_Tn_prop}\ref{L_Tn_prop_4} and joint continuity of $g$ imply that $g(\fdot,T_n^*\mu)\to g(\fdot,\mu)$ locally uniformly. Therefore, by the expression \eqref{eq_L_C1r_der_fn_1} and Lemma~\ref{L_Tn_prop}\ref{L_Tn_prop_3}, $g_n(x,\mu)=T_ng(\fdot,T_n^*\mu)(x)\to g(x,\mu)$ for every $x\in {\R^d}$. We have proved \eqref{T_approx_cyl_1}.

To prove \eqref{T_approx_cyl_2}, let $K \subset \Pcal_p$ be an arbitrary compact set. Lemma~\ref{L_Tn_prop}\ref{L_Tn_prop_a1} gives a possibly larger compact set $K'$ such that $T_n^*\mu\in K'$ for all $n$ and all $\mu\in K$. Thus $|f_n(\mu)| = |f(T_n^*\mu)| \le \max_{K'} |f| < \infty$ for $\mu \in K$. Moreover, since $f$ belongs to $C^1(\Pcal_p)$, it satisfies the locally uniform $p$-growth bound
\[
\left| \frac{\partial f}{\partial\mu}(x,T_n^*\mu) \right| \le c_{K'} (1+|x|^p)
\]
for some constant $c_{K'}$ and all $\mu \in K$ and $x \in {\R^d}$. Combining this with \eqref{eq_L_C1r_der_fn_1}, Lemma~\ref{L_Tn_prop}\ref{L_Tn_prop_a2} (with $h(x)=|x|$), and the fact that $T_n$ preserves growth bounds, we obtain
\[
\left| \frac{\partial f_n}{\partial\mu}(x,\mu) \right|
= \left| T_n \frac{\partial f}{\partial\mu}(\fdot,T_n^*\mu)(x) \right|
\le T_n \left| \frac{\partial f}{\partial\mu}(\fdot,T_n^*\mu) \right|(x)
\le c_{K'} (1+|x|^p)
\]
for all $\mu \in K$, $x \in {\R^d}$, $n\in\Nb$. Setting $c_K = c_{K'} \vee \max_{K'} |f|$ gives \eqref{T_approx_cyl_2}.
\end{proof}

\subsection{Second order derivatives}

\begin{definition}\label{D_C2r}
Let $p \in [1,\infty) \cup \{0\}$. A function $f\in C^1(\Pcal_p)$ is said to belong to $C^2(\Pcal_p)$ if there is a continuous function $(x,y,\mu)\mapsto\frac{\partial^2 f}{\partial\mu^2}(x,y,\mu)$ from ${\R^d}\times {\R^d}\times\Pcal_p$ to $\R$, called (a version of) the \emph{second derivative} of $f$, such that $\frac{\partial^2 f}{\partial\mu^2}$ is symmetric in its first two arguments and  the following properties hold.
\begin{itemize}
\item \textbf{locally uniform $p$-growth:} for every compact set $K\subset\Pcal_p$, there is a constant $c_K$ such that for all $x,y\in {\R^d}$ and $\mu\in K$,
\begin{equation}\label{eq_C2r_p_growth}
\left| \frac{\partial^2 f}{\partial\mu^2}(x,y,\mu) \right| \le c_K (1+|x|^p+|y|^p),
\end{equation}
\item \textbf{fundamental theorem of calculus:} for every $\mu,\nu\in\Pcal_p$,
\begin{equation}\label{eq_C2r_FTC}
\begin{aligned}
&f(\nu) - f(\mu) - \int_{\R^d} \frac{\partial f}{\partial\mu}(x,\mu)(\nu-\mu)(\dx) \\
&\quad=\int_0^1 \int_0^t \int_{{\R^d}\times {\R^d}} \frac{\partial^2 f}{\partial\mu^2}(x,y,s\nu + (1-s)\mu)(\nu-\mu)^{\otimes 2}(\dx,\dy)\ds \dt.
\end{aligned}
\end{equation}
Here $(\nu-\mu)^{\otimes 2}$ is shorthand for the product measure $(\nu-\mu) \otimes (\nu-\mu)$ on $\R^d \times \R^d$.
\end{itemize}
\end{definition}

\begin{remark}
Observe that the imposed symmetry permits to avoid unnecessary redundancies. One can indeed see that adding a term of the form
$(x,y,\mu)\mapsto c(x,\mu)-c(y,\mu)$ to a version of the second derivative of $f$ does not change the value of
the integral term on the right hand side of \eqref{eq_C2r_FTC}.
Moreover note that  if  $(x,y,\mu)\mapsto\frac{\partial^2 f}{\partial\mu^2}(x,y,\mu)$ is a version of the second derivative of $f$, then the same holds for $(x,y,\mu)\mapsto\frac{\partial^2 f}{\partial\mu^2}(x,y,\mu)+a(x,\mu)+a(y,\mu)$ for each continuous map  $(x,\mu)\mapsto a(x,\mu)$.
 Modulo additive terms of this form, the second derivative is uniquely determined. For more details on this property see Appendix~\ref{secC}.
\end{remark}

\begin{remark}
If $q < p$, then $C^2(\Pcal_q) \subset C^2(\Pcal_p)$ in the sense described in Remark~\ref{R_C1Pq_in_C1Pp}. The reasoning for verifying this is the same.
\end{remark}

Consider a function $f$ of the form \eqref{eq_C1_cylinder}, now with $\tilde f\in C^2(\R^n)$. We refer to such a function as a \emph{$C^2$ cylinder function}. A version of its first derivative is given by \eqref{eq_C1r_der_cyl_ex}, and a version of its second derivative is
\begin{equation}\label{eq_C2r_2n_der_cyl_ex}
\frac{\partial^2 f}{\partial\mu^2}(x,y,\mu) = \sum_{i,j=1}^n \partial^2_{ij}\tilde f(\mu(\varphi_1),\ldots,\mu(\varphi_n)) \varphi_i(x)\varphi_j(y).
\end{equation}
Any $C^2$ cylinder function belongs to $C^2(\Pcal_p)$ for every $p$. The following result extends Theorem~\ref{T_approx_cyl} in the case of $C^2$ functions.

\begin{theorem}\label{T_approx_cyl_C2}
Let $f \in C^2(\Pcal_p)$ for some $p \in [1,\infty) \cup \{0\}$. Then there exist $C^2$ cylinder functions $f_n$ such that one has the pointwise convergence \eqref{T_approx_cyl_1} as well as
\begin{equation}\label{T_approx_cyl_1_C2}
\frac{\partial^2 f_n}{\partial\mu^2}(x,y,\mu) \to \frac{\partial^2 f}{\partial\mu^2}(x,y,\mu)
\end{equation}
for all $\mu \in \Pcal_p$, $x,y \in {\R^d}$, and for every compact set $K\subset\Pcal_p$ there is constant $c_K$ such that \eqref{T_approx_cyl_2} holds along with
\begin{equation}\label{T_approx_cyl_2_C2}
\left|  \frac{\partial^2 f_n}{\partial\mu^2}(x,y,\mu) \right| \le c_K (1 + |x|^p + |y|^p)
\end{equation}
for all $\mu\in K$, $x,y \in {\R^d}$, $n\in\N$.
\end{theorem}

To prove this result we introduce $T_n^{\otimes 2}$ acting on functions $(x,y)\mapsto\varphi(x,y)$ of two variables by
\[
T_n^{\otimes 2}\varphi(x,y) = \sum_{i,j=0}^{N_n} \varphi(x^n_i,x^n_j)\psi^n_i(x)\psi^n_j(y).
\]
Observe that $T_n^{\otimes 2}\varphi$ is always continuous. Moreover, taking $\varphi(x,y)=h(|x|,|y|)$ for any nonnegative function $h$ increasing in both of its arguments, by applying \eqref{eq_Tn_growth_pres} twice, we obtain
\begin{align*}
T_n^{\otimes 2}\varphi (x,y)
=T_n\left(y\mapsto T_n h(|\cdot|,|y|)(x)\right)(y)
\le T_n h(|x|,|\cdot|)(y)
\le \varphi(x,y).
\end{align*}
In particular, if $\varphi$ satisfies a $p$-growth bound of the form $|\varphi(x,y)|\le c(1+|x|^p+|y|^p)$ on $\R^d\times\R^d$, it follows that $T_n^{\otimes 2}\varphi$ satisfies the same bound.

\begin{lemma}\label{L_Tn2_prop}
The operators $T_n^{\otimes 2}$ satisfy the following basic properties.
\begin{enumerate}
\item\label{L_Tn2_prop_0} if $|x|\vee|y|\le n$, then 
$$|T_n^{\otimes 2}\varphi(x,y)| \le \sup\{|\varphi(u,z)|\colon (u,z)\in\R^d\times\R^d, |(u,z)-(x,y)|<\sqrt{2}/n\},$$
\item\label{L_Tn2_prop_1} if $\varphi$ is continuous at $(x,y)\in\R^d\times\R^d$, then $T_n^{\otimes 2}\varphi(x,y)\to\varphi(x,y)$,
\item\label{L_Tn2_prop_2} if $\varphi$ is continuous everywhere, then $T_n^{\otimes 2}\varphi\to\varphi$ locally uniformly,
\item\label{L_Tn2_prop_3} if $\varphi_n\to\varphi$ locally uniformly and $\varphi$ is continuous at $(x,y)\in\R^d\times\R^d$, then $T_n^{\otimes 2}\varphi_n(x,y)\to\varphi(x,y)$.
\end{enumerate}
\end{lemma}

\begin{proof}
\ref{L_Tn2_prop_0}: If $x,y\in B_n$ then $x\in U^n_i$ and $y\in U^n_j$ for some (possibly many) $i,j\ne0$. These sets all have diameter at most $1/n$, so
\begin{align*}
|T_n^{\otimes 2}\varphi(x,y)|
&\le\sum_{i,j=1}^{N_n}|\varphi(x^n_i,x^n_j)|\psi^n_i(x)\psi^n_j(y)\\
&\le\sup\{|\varphi(u,z)|\colon u,z\in {\R^d},|u-x|\vee|z-y|<1/n\}.
\end{align*}

\ref{L_Tn2_prop_1}: Let $\omega_{(x,y)}(\delta)$ be an increasing modulus of continuity for $\varphi$ at $(x,y)$. Then $|\varphi(x^n_i,x^n_j)-\varphi(x,y)|\le\omega_{(x,y)}(|(x^n_i,x^n_j)-(x,y)|)\le\omega_{(x,y)}(\sqrt{2}n^{-1})$ whenever $x\in U^n_i$ and $y\in U^n_j$ for $i,j\ne0$. Because $x,y\notin U^n_0$ for all large $n$, it follows that
\[
|T_n^{\otimes 2}\varphi(x,y) - \varphi(x,y)| \le \sum_{i,j=1}^{N_n} |\varphi(x^n_i,x^n_j) - \varphi(x,y)|\psi^n_i(x)\psi^n_j(y) \le \omega_{(x,y)}(\sqrt{2}n^{-1}) \to 0.
\]

\ref{L_Tn2_prop_2}: Fix a compact set $J\subset\R^d\times\R^d$ and let $\omega(\delta)$ be a uniform modulus of continuity for $\varphi$ on $J$. Because it holds for all large $n$ that $x,y\not\in U^n_0$ for all $(x,y)\in J$, the same computation as above gives $|T_n^{\otimes 2}\varphi(x,y) - \varphi(x,y)| \le \omega(\sqrt{2}n^{-1})$ for all $(x,y)\in J$.

\ref{L_Tn2_prop_3}: Write $|T_n^{\otimes 2}\varphi_n(x,y)-\varphi(x,y)| \le |T_n^{\otimes 2}(\varphi_n-\varphi)(x,y)| + |T_n^{\otimes 2}\varphi(x,y)-\varphi(x,y)|$, and denote the two terms on the right-hand side by $A_n$ and $B_n$, respectively. We have from \ref{L_Tn2_prop_0} that $A_n\le
\sup\{|\varphi_n-\varphi|(u,z)\colon (u,z)\in\R^d\times\R^d, |(u,z)-(x,y)|<\sqrt{2}/n\}
$ for all large $n$, so that $A_n\to0$. Moreover, thanks to \ref{L_Tn2_prop_1}, $B_n\to0$.
\end{proof}

\begin{lemma}\label{L_C2r_der_fn}
Suppose $f$ belongs to $C^2(\Pcal_p)$ and define $f_n(\mu)=f(T_n^*\mu)$. Then $f_n$ is a $C^2$ cylinder function, a version of its first derivative is given by \eqref{eq_L_C1r_der_fn_1}, and a version of its second derivative is given by
\begin{align}\label{eq_L_C2r_der_fn_2}
\frac{\partial^2f_n}{\partial\mu^2}(x,y,\mu) = T_n^{\otimes 2} \frac{\partial^2f}{\partial\mu^2}(\fdot,\fdot,T_n^*\mu)(x,y).
\end{align}
\end{lemma}

\begin{proof}
We first show that $f_n$ is a $C^2$ cylinder function. To this end, write $f_n(\mu) = \tilde f(\mu(\psi^n_0),\ldots,\mu(\psi^n_{N_n}))$, where we define $\tilde f$ as in \eqref{eq_L_C1r_der_fn_ftilde} on the $N_n$-simplex $\Delta^{N_n}$ in $\R^{N_n+1}$.
We now argue that $\tilde f$ satisfies a fundamental theorem of calculus. Pick any $p,q\in\Delta^{N_n}$. Writing $\nu=p_0\delta_{x^n_0} + \ldots + p_{N_n}\delta_{x^n_{N_n}}$ and $\eta=q_0\delta_{x^n_0} + \ldots + q_{N_n}\delta_{x^n_{N_n}}$, and using that $f$ satisfies the fundamental theorem of calculus \eqref{eq_C2r_FTC} by assumption, we get
\begin{align*}
\tilde f(q)- & \tilde f(p) 
= f(\eta) - f(\nu) \\
=& \int_{\R^d} \frac{\partial f}{\partial\mu}(x,\nu)(\eta-\nu)(\dx)\\
&+\int_0^1\int_0^t \int_{\R^d\times\R^d} \piann{f}{\mu}(x,y,s\eta + (1-s)\nu)(\eta-\nu)^{\otimes 2}(\dx,\dy)\ds\dt \\
=& \sum_{i=0}^{N_n}  \pian{f}{\mu}(x^n_i,\nu)(q_i-p_i)\\
&+\int_0^1\int_0^t \sum_{i,j=0}^{N_n}  \piann{f}{\mu}(x^n_i,x^n_j,s\eta + (1-s)\nu)(q_i-p_i)(q_j-p_j)\ds\dt \\
=& \sum_{i=0}^{N_n} \partial_i\tilde f(p)(q_i-p_i)
+ \int_0^1\int_0^t \sum_{i,j=0}^{N_n} \partial^2_{ij}\tilde f(sq+(1-s)p)(q_i-p_i)(q_j-p_j)\ds\dt\\
=& \sum_{i=0}^{N_n} \partial_i\tilde f(p)(q_i-p_i)
+\frac 1 2 \sum_{i,j=0}^{N_n} \partial^2_{ij}\tilde f(p)(q_i-p_i)(q_j-p_j)\\
&\qquad+ \sum_{i,j=0}^{N_n}\int_0^1\int_0^t ( \partial^2_{ij}\tilde f(sq+(1-s)p)-  \partial^2_{ij}\tilde f(p))(q_i-p_i)(q_j-p_j)\ds\dt,
\end{align*}
where we define $\partial_i\tilde f(p)=\frac{\partial f}{\partial\mu}(x^n_i,p_0\delta_{x^n_0}+\cdots+p_{N_n}\delta_{x^n_{N_n}})$ as in Lemma~\ref{L_C1r_der_fn}, and $\partial^2_{ij}\tilde f(p)=\frac{\partial^2 f}{\partial\mu^2}(x^n_i,x^n_j,p_0\delta_{x^n_0}+\ldots+p_{N_n}\delta_{x^n_{N_n}})$.
Since $f$ belongs to $C^2(\Pcal_p)$, the functions $\partial_i\tilde f$ and $\partial^2_{ij}\tilde f$ are continuous on $\Delta^{N_n}$.
The above implies that $\tilde f$ is $C^2$ on $\Delta^{N_n}$ is the sense that the tangential derivatives exist and are uniformly continuous on the relative interior of $\Delta^{N_n}$. Using Whitney's extension theorem, see e.g.\ \citep[Appendix, Corollary~6.3]{MR838085}, we deduce that $\tilde f$ can be extended to a $C^2$ function on all of $\R^{N_n+1}$.
%This allows us to use the $C^2$ case of Whitney's extension theorem to deduce that $\tilde f$ can be extended to a $C^2$ function on all of $\R^{N_n+1}$.
This confirms that $f_n$ is a $C^2$ cylinder function. 
To verify \eqref{eq_L_C2r_der_fn_2}, it now suffices to note that $\partial^2_{ij}\tilde f(\mu(\psi^n_0),\ldots,\mu(\psi^n_{N_n}))=\frac{\partial^2f}{\partial\mu^2}(x^n_i,x^n_j,T_n^*\mu)$ and apply \eqref{eq_C2r_2n_der_cyl_ex}.
\end{proof}

\begin{proof}[Proof of Theorem~\ref{T_approx_cyl_C2}]
Take $f_n(\mu)=f(T_n^*\mu)$, which are $C^2$ cylinder functions due to Lemma~\ref{L_C2r_der_fn}. Thanks to Theorem~\ref{T_approx_cyl}, only \eqref{T_approx_cyl_1_C2} and \eqref{T_approx_cyl_2_C2} need to be argued. 

To simplify notation, write $g(x,y,\mu)=\frac{\partial^2 f}{\partial\mu^2}(x,y,\mu)$ and $g_n(x,y,\mu)=\frac{\partial^2 f_n}{\partial\mu^2}(x,y,\mu)$. Then for each fixed $\mu\in\Pcal_p$, Lemma~\ref{L_Tn_prop}\ref{L_Tn_prop_4} and joint continuity of $g$ imply that $g(\fdot,\fdot,T_n^*\mu)\to g(\fdot,\fdot,\mu)$ locally uniformly. Therefore, by the expression \eqref{eq_L_C2r_der_fn_2} and Lemma~\ref{L_Tn2_prop}\ref{L_Tn2_prop_3}, $g_n(x,y,\mu)=T_n^{\otimes 2}g(\fdot,\fdot,T_n^*\mu)(x,y)\to g(x,y,\mu)$ for every $(x,y)\in\R^d\times\R^d$. We have proved \eqref{T_approx_cyl_1_C2}.

To prove \eqref{T_approx_cyl_2_C2}, let $K \subset \Pcal_p$ be an arbitrary compact set. Lemma~\ref{L_Tn_prop}\ref{L_Tn_prop_a1} gives a possibly larger compact set $K'$ such that $T_n^*\mu\in K'$ for all $n$ and all $\mu\in K$. Since $f$ belongs to $C^2(\Pcal_p)$, it satisfies the locally uniform $p$-growth bound
\[
\left| \frac{\partial^2 f}{\partial\mu^2}(x,y,T_n^*\mu) \right| \le c_{K'} (1+|x|^p+|y|^p)
\]
for some constant $c_{K'}$ and all $\mu \in K$ and $x,y \in {\R^d}$.
Combining this with \eqref{eq_L_C2r_der_fn_2}, the fact that $|T_n^{\otimes 2}\varphi|\le T_n^{\otimes 2}|\varphi|$ due to the triangle inequality, and the fact that $T_n^{\otimes 2}$ preserves growth bounds, we obtain
\begin{align*}
\left|  \frac{\partial^2 f_n}{\partial\mu^2}(x,y,\mu) \right| 
&=\left|T_n^{\otimes 2} \frac{\partial^2f}{\partial\mu^2}(\fdot,\fdot,T_n^*\mu)(x,y)\right|\\
&\le T_n^{\otimes 2} \left|\frac{\partial^2f}{\partial\mu^2}(\fdot,\fdot,T_n^*\mu)\right|(x,y)
%\le T_n^{\otimes 2}\left(c_{K'}(1 + |\cdot|^p + |\cdot|^p)\right)(x,y)
\le c_{K'}(1 + |x|^p + |y|^p)
\end{align*}
for all $\mu \in K$, $x,y \in {\R^d}$, $n\in\Nb$, which gives \eqref{T_approx_cyl_2_C2}.
\end{proof}

\section{It\^o's formula} \label{S_Ito}
\label{sec:itos-formula}

We now establish the following It\^o formula, which is a crucial tool in this paper. Most importantly, it is used to prove the viscosity sub- and super-solution properties in Sections~\ref{sec:visc-subs-prop} and~\ref{sec:visc-supers-prop}.

\begin{theorem}\label{T_Ito}
Let $(\xi,\rho)$ be a \emph{weak solution} of \eqref{eq_MVM_SDE}, where $\xi$ takes values in $\Pcal_p$ for some fixed $p \in [1,\infty) \cup \{0\}$.
%Let $(\Omega,\Fcal,(\Fcal_t)_{t\ge0},P)$ be a filtered probability space, $W$ a standard Brownian motion on this space, $\xi$ a continuous MVM with values in $\Pcal_p$ for some fixed $p\ge0$, and $\rho$ a progressively measurable function on $\Omega\times\R_+\times\R$.
Let $q \in [1,p] \cup \{0\}$ and assume that, $\P\otimes \dt$-a.e., 
\begin{equation}\label{eq_Ito_assumption}
\int_0^t \left( \int_{\R^d} (1 + |x|^q) |\rho_s(x) - \xi_s(\rho_s)| \xi_s(\dx) \right)^2 \ds < \infty.
\end{equation}
%and that for every $\varphi\in C_b$, $P\otimes \dt$-a.e.,
%\begin{equation}\label{eq:ito_general_sde_NEW}
%\xi_t(\varphi) = \xi_0(\varphi) + \int_0^t \Cov_{\xi_s}(\varphi,\rho_s) \d W_s.
%\end{equation}
Then, for every $f$ in $C^2(\Pcal_q)$ we have the It\^o formula
\begin{equation}\label{eq:ito_general_NEW}
\begin{aligned}
f(\xi_t) &= f(\xi_0) + \int_0^t \int_{\R^d} \frac{\partial f}{\partial \mu}(x,\xi_s)\sigma_s(\dx) \d W_s \\
&\quad + \frac12 \int_0^t \int_{{\R^d}\times {\R^d}} \frac{\partial^2 f}{\partial \mu^2}(x,y,\xi_s) \sigma_s(\dx)\sigma_s(\dy) \ds,
\end{aligned}
\end{equation}
where we write $\sigma_s(\dx)=(\rho_s(x)-\xi_s(\rho_s))\xi_s(\dx)$.
\end{theorem}

\begin{remark}
Note that \eqref{eq_Ito_assumption} is the same as condition \eqref{eqn3}. A sufficient condition for it to hold is given in Lemma~\ref{rem1}.
\end{remark}
%\begin{example}
%\fbox{To be worked out.} Suppose $(\xi,\rho,W)$ satisfies \eqref{eq:ito_general_sde_NEW}, where $\rho_t(x) \equiv \rho(x)$ is constant in time and satisfies the bound $|\rho(x)| \le c(1+|x|^p)$. If $p = q = 0$, then \eqref{eq_Ito_assumption} is clearly satisfied. If $\xi_0 \in \Pcal_p$ and $0 \le q < p$, then \eqref{eq_Ito_assumption} is satisfied.
%\end{example}

\begin{remark}
The formula \eqref{eq:ito_general_NEW} cannot easily be expressed in terms of the Lions derivative. Indeed, the second Lions derivative of $f$ coincides with $\nabla_x \nabla_y \frac{\partial^2 f}{\partial\mu^2}(x,y,\mu)$. Using this object to express the last term of \eqref{eq:ito_general_NEW} would require us to first undo the two gradient operations. %The absence of Lions derivatives makes our It\^o formula appear different than those often seen in the mean-field games literature.
\end{remark}

The proof of Theorem~\ref{T_Ito} proceeds by first proving the result for $C^2$ cylinder functions and then for more general functions by an approximation argument. A similar strategy was used by \cite{guo_pha_wei_20} in the context of McKean--Vlasov equations. The first step is straightforward and only requires real-valued It\^o calculus. The approximation argument is slightly more delicate, and builds on Theorem~\ref{T_approx_cyl_C2}. We begin with the first step.

\begin{lemma}\label{L_Ito_D}
Let $(\xi,\rho)$ be as in Theorem~\ref{T_Ito}. Then It\^o's formula \eqref{eq:ito_general_NEW} holds for all $C^2$ cylinder functions.
\end{lemma}

\begin{proof}
Let $f(\mu) = \tilde f(\mu(\varphi_1),\ldots,\mu(\varphi_n))$ be a $C^2$ cylinder function as in \eqref{eq_C1_cylinder}. Using \eqref{eq_MVM_SDE} and It\^o's formula for real-valued processes we get
\begin{align*}
\d f(\xi_t) &= \sum_{i=1}^n \partial_i \tilde f(\xi_t(\varphi_1),\ldots,\xi_t(\varphi_n)) \Cov_{\xi_t}(\varphi_i,\rho_t) \d W_t \\
&\quad + \frac12 \sum_{i,j=1}^n \partial^2_{ij}\tilde f(\xi_t(\varphi_1),\ldots,\xi_t(\varphi_n)) \Cov_{\xi_t}(\varphi_i,\rho_t) \Cov_{\xi_t}(\varphi_j,\rho_t) \dt \\
&= \int_{\R^d} \sum_{i=1}^n \partial_i \tilde f(\xi_t(\varphi_1),\ldots,\xi_t(\varphi_n)) \varphi_i(x) \sigma_t(\dx) \d W_t \\
&\quad + \frac12 \int_{{\R^d}\times {\R^d}} \sum_{i,j=1}^n \partial^2_{ij}\tilde f(\xi_t(\varphi_1),\ldots,\xi_t(\varphi_n)) \varphi_i(x) \varphi_j(y)\sigma_t(\dx)\sigma_t(\dy) \dt,
\end{align*}
where we write $\sigma_t(\dx)=(\rho_t(x)-\xi_t(\rho_t))\xi_t(\dx)$. In view of the expressions \eqref{eq_C1r_der_cyl_ex} and \eqref{eq_C2r_2n_der_cyl_ex} for the derivatives of $C^2$ cylinder functions, the above expression is precisely \eqref{eq:ito_general_NEW}.
\end{proof}

We now proceed with the second step. Fix $q \in [1,\infty) \cup \{0\}$. We consider triplets $(f,g,H)$ of measurable functions $f\colon\Pcal_q\to\R$, $g\colon {\R^d}\times\Pcal_q\to\R$, $H\colon {\R^d}\times {\R^d}\times\Pcal_q\to\R$ that satisfy the following growth bound: for every compact set $K \subset \Pcal_q$ there is a constant $c_K$ such that
\[
|f(\mu)| \le c_K, \quad | g(x,\mu) | \le c_K (1 + |x|^q), \quad | H(x,y,\mu) | \le c_K (1 + |x|^q + |y|^q)
\]
for all $\mu \in K$, $x,y \in {\R^d}$. We define a notion of convergence for such triplets as follows. We say that $(f_n,g_n,H_n)\to(f,g,H)$ in the sense of \emph{local b.p.~(bounded pointwise) convergence} if the functions $f_n,g_n,H_n$ converge pointwise to $f,g,H$, and the above growth bounds hold uniformly in $n$; that is, for every compact set $K \subset \Pcal_q$ there is a constant $c_K$ such that
\[
|f_n(\mu)| \le c_K, \quad | g_n(x,\mu) | \le c_K (1 + |x|^q), \quad | H_n(x,y,\mu) | \le c_K (1 + |x|^q + |y|^q)
\]
holds for all $\mu \in K$, $x,y \in {\R^d}$, and all $n \in \N$. Given any collection $\Acal$ of such triplets $(f,g,H)$, the \emph{local b.p.~closure} of $\Acal$ is the smallest set that contains $\Acal$ and is closed with respect to local b.p.~convergence. Observe that the notions of local b.p.~convergence and closure depend on the parameter $q$, both through the domain of definition of $f,g,H$, through the exponent in the growth bounds, and through the meaning of compactness in $\Pcal_q$.

\begin{lemma}\label{L_Ito_bpclosure}
Let $p$, $q$, and $(\xi,\rho)$ be as in Theorem~\ref{T_Ito}. Consider a collection $\Acal$ of triplets as above (using the given $q$), and assume that
\begin{equation}\label{eq_L_Ito_bpclosure}
\begin{aligned}
f(\xi_t) &= f(\xi_0) + \int_0^t \int_{\R^d} g(x,\xi_s)\sigma_s(\dx) \d W_s \\
&\quad + \frac12 \int_0^t \int_{{\R^d}\times {\R^d}} H(x,y,\xi_s)\sigma_s(\dx)\sigma_s(\dy) \ds
\end{aligned}
\end{equation}
holds for every $(f,g,H)\in\Acal$, where we write $\sigma_s(\dx)=(\rho_s(x)-\xi_s(\rho_s))\xi_s(\dx)$. Then \eqref{eq_L_Ito_bpclosure} also holds for all $(f,g,H)$ in the local b.p.~closure of $\Acal$.
\end{lemma}

\begin{proof}
It suffices to consider $(f_n,g_n,H_n)\in\Acal$ converging to some $(f,g,H)$ in the local b.p.~sense, and show that \eqref{eq_L_Ito_bpclosure} holds for any fixed $t$. By localisation we may assume that the left-hand side of \eqref{eq_Ito_assumption} is bounded by a constant, and in particular
\begin{equation}\label{eq_Ito_assumption_v2}
\E\left[ \int_0^t \left( \int_{\R^d} (1 + |x|^q) |\rho_s(x) - \xi_s(\rho_s)| \xi_s(\dx) \right)^2 \ds \right] < \infty.
\end{equation}
By further localisation based on Lemma~\ref{L_lin_fcn} and  Remark~\ref{rem:MVMCompact}\ref{page:mvm_localise_compact}, we may additionally assume that $\{\xi_s\colon s\in[0,t]\}$ remains inside some compact set $K\subset\Pcal_p$. Since $q \le p$, $K$ is also a compact subset of $\Pcal_q$.

Clearly $f_n(\xi_t)\to f(\xi_t)$ and $f_n(\xi_0)\to f(\xi_0)$. Next, we claim that
\begin{equation}\label{eq_Ito_proof_2}
\E\left[ \int_0^t \left( \int_{\R^d} (g_n-g)(x,\xi_s)\sigma_s(\dx)\right)^2 \ds \right] \to 0.
\end{equation}
To see this, first observe that $g_n \to g$ pointwise. Moreover, recall that $\sigma_s(\dx) = (\rho_s(x) - \xi_s(\rho_s))\xi_s(\dx)$ and note that
\begin{equation}\label{eq_Ito_proof_3}
| (g_n - g)(x,\xi_s) (\rho_s(x) - \xi_s(\rho_s)) | \le 2 c_K (1 + |x|^q) |\rho_s(x) - \xi_s(\rho_s)|
\end{equation}
since $\xi_s$ remains inside $K$. Due to \eqref{eq_Ito_assumption_v2} we have, with probability one, that
\[
\int_{\R^d} (1 + |x|^q) |\rho_s(x) - \xi_s(\rho_s)| \xi_s(\dx) < \infty
\]
for Lebesgue-a.e.~$s \in [0,t]$, so the dominated convergence theorem gives
\[
\int_{\R^d} (g_n-g)(x,\xi_s)\sigma_s(\dx) \to 0
\]
for all such $s$. Moreover, using again \eqref{eq_Ito_proof_3} we have
\[
\left( \int_{\R^d} (g_n-g)(x,\xi_s)\sigma_s(\dx)\right)^2 \le 4 c_K^2 \left( \int_{\R^d} (1 + |x|^q) |\rho_s(x) - \xi_s(\rho_s)| \xi_s(\dx) \right)^2,
\]
which is $\P \otimes \ds$-integrable thanks to \eqref{eq_Ito_assumption_v2}. One more application of dominated convergence now gives \eqref{eq_Ito_proof_2}. With this in hand, we obtain $\int_0^t \int_{\R^d} g_n(x,\xi_s)\sigma_s(\dx) \d W_s\to\int_0^t \int_{\R^d} g(x,\xi_s)\sigma_s(\dx) \d W_s$ in $L^2(\P)$, by use of the It\^o isometry.

It only remains to argue that
\[
\E\left[ \left|\int_0^t\int_{{\R^d}\times {\R^d}}(H_n-H)(x,y,\xi_s) \sigma_s(\dx)\sigma_s(\dy)\ds\right| \right] \to 0.
\]
This follows from dominated convergence on noting that $H_n \to H$ pointwise, and making use of the bounds
\[
|H_n-H|(x,y,\xi_s) \le 2c_K (1 + |x|^q + |y|^q)
\]
and
\begin{align*}
\int_{{\R^d}\times {\R^d}} & (1 + |x|^q + |y|^q) |\rho_s(x) - \xi_s(\rho_s)|  |\rho_s(y) - \xi_s(\rho_s)| \xi_s(\dx) \xi_s(\dy) \\
&\le \left( \int_{\R^d} (1 + |x|^q) |\rho_s(x) - \xi_s(\rho_s)| \xi_s(\dx) \right)^2,
\end{align*}
which is $\P \otimes \ds$-integrable thanks to \eqref{eq_Ito_assumption_v2}. All in all, we deduce that \eqref{eq_L_Ito_bpclosure} carries over from $(f_n,g_n,H_n)$ to $(f,g,H)$.
\end{proof}

\begin{proof}[Proof of Theorem~\ref{T_Ito}]
Define $\Acal=\{(f,\frac{\partial f}{\partial\mu},\frac{\partial^2f}{\partial\mu^2})\colon \text{$f$ is a $C^2$ cylinder function}\}$. According to Lemmas~\ref{L_Ito_D} and~\ref{L_Ito_bpclosure}, \eqref{eq_L_Ito_bpclosure} holds for all elements of the local b.p.~closure of $\Acal$. In particular, by Theorem~\ref{T_approx_cyl_C2}, this closure contains all triplets $(f,\frac{\partial f}{\partial\mu},\frac{\partial^2f}{\partial\mu^2})$ with $f$ in $C^2(\Pcal_q)$. This gives the result.
\end{proof}

\section{Viscosity solutions and HJB equation}\label{S_viscosity}

Fix exponents $p \in [1,\infty) \cup \{0\}$ and $q \in [1,p] \cup \{0\}$. Using the dynamic programming principle, we will prove that the value function \eqref{eq_value_function} is a viscosity solution of the following HJB equation:
\begin{align}
\beta u(\mu) + \sup_{\rho\in\Hb}\left\{ - c(\mu,\rho) - Lu(\mu,\rho) \right\} = 0, \quad \mu\in\Pcal_p, \label{eq_HJB} 
%\\
%u(\delta_x)&=c(x)/\beta,&& \delta_x\in\Pcal^s \label{eq_BC} 
\end{align}
where
% \[
% c(x) = \inf_{\rho\in \mathbb H}  c(\delta_x,\rho)
% \]
% and
the operator $L$ is %defined in \eqref{eq_Lf}. 
given by
$$Lf(\mu,\rho)=\frac12  \int_{{\R^d}\times {\R^d}} \frac{\partial^2 f}{\partial \mu^2}(x,y,\mu) \sigma(\dx)\sigma(\dy)$$
with $\sigma(\dx)=(\rho(x)-\mu(\rho))\mu(\dx)$, for any $f \in C^2(\Pcal_q)$, $\mu \in \Pcal_p$, and $\rho \in L^1(\mu)$ such that $\frac{\partial^2 f}{\partial \mu^2}(\fdot,\fdot,\mu)$ belongs to $L^1(\sigma \otimes \sigma)$. In all other cases we set $Lf(\mu,\rho)=+\infty$ by convention.

\begin{remark}\label{rem3}
We observe that when $\mu = \delta_x \in \mathcal{P}^s$ and $\beta > 0$, then \eqref{eq_HJB} simplifies to: $u(\delta_x)=c(x)/\beta$ 
where
\[
c(x) = \inf_{\rho\in \mathbb H}  c(\delta_x,\rho).
\]
This can be interpreted as a kind of boundary condition. Since an MVM starting at a Dirac measure $\delta_x$ stays there for all times, the value function \eqref{eq_value_function} must satisfy $v(\delta_x)=c(x)/\beta$, which is exactly \eqref{eq_HJB}. Note that (up to required continuity or semi-continuity conditions), we can modify the value of $c$ only on the set of singular measures --- such an action is then equivalent to changing the boundary values of the problem, since this change will affect the behaviour before entry time to $\mathcal{P}^s$ through its change to the final value accrued after the entry time to the set $\mathcal{P}^s$.
\end{remark}

The following is the main result of this paper. The notion of viscosity solution is defined precisely below. It will be convenient to introduce the notation $\Hb_c := \Hb \cap C_c(\R^d)$. Recall also the standing assumptions in Section~\ref{sec:contr-probl-dynam} placed on $\Hb,\beta,c$. 

\begin{theorem}\label{thetheorem}
Assume that 
\begin{enumerate}
%\item\label{iti} {\color{red} $\mu\mapsto Lf(\mu,\bar\rho)$ is continuous for each $f\in C^2(\Pcal_q)$ and each $\bar\rho\in \mathbb H\cap C_c(\R)$;}
\item\label{itii0}  there is a constant $\r \in (0,\infty)$ such that $|\rho(x)| \le R( 1 + |x|^p )$ for all $x \in {\R^d}$ and $\rho \in \Hb_c$;
%\[
%\sup_{x \in \R} \frac{|\rho(x)|}{1 + |x|^p} \le \kappa \quad \text{for all $\rho \in \Hb$.}
%\]
\item\label{itii} $\mu\mapsto c(\mu,\rho)$ is upper semi-continuous for every $\rho\in \mathbb H_c$;
\item\label{itiii} for every $\mu \in \Pcal_p$ and every $f \in C^2(\Pcal_q)$,
$$\sup_{\rho\in\Hb}\left\{ - c(\mu,\rho) - Lf(\mu,\rho) \right\}
=\sup_{\rho\in\Hb_c}\left\{ - c(\mu,\rho) - Lf(\mu,\rho) \right\}.$$
%\item\label{itiii} $\Hb$ can be replaced by $\Hb\cap C_c(\R)$ in \eqref{eq_Hmuf}, i.e.
%$$\sup_{\rho\in\Hb}\left\{ - c(\mu,\rho) - Lf(\mu,\rho) \right\}
%=\sup_{\overline \rho\in\Hb\cap C_c(\R)}\left\{ - c(\mu,\overline\rho) - Lf(\mu,\overline\rho) \right\};$$
%\item\label{itiv} $c$ is bounded from below. 
\end{enumerate}
Then the value function $v\colon\Pcal_p\to\overline{\R}$ given by \eqref{eq_value_function} is a viscosity solution of \eqref{eq_HJB}.

If we additionally suppose that $\beta > 0$ and % the cost function $c$ and the action space $\Hb$ satisfy the following conditions: %for any finite set $K\subset \R$
\begin{enumerate}[resume]
\item\label{itv} $v\in C(\Pcal_p)$;
\item \label{itvi} $\mu\mapsto c(\mu,\rho)$ is continuous on $\Pc(\{x_1,...,x_N\})$ uniformly in $\rho\in\Hb_c$ for any $N\in\Nb$ and $x_1,...,x_N\in  {\R^d}$, %$\mu\mapsto c(\mu,\overline\rho)$ is continuous on $\{\mu\in \Pcal_p\colon \mu(K)=1\}$ uniformly with respect to ${\color{blue}\overline\rho\in\Hb\cap C_c(\R)}$;
% \item\label{itvii} the set $\{\rho(x) - \rho(0) \colon \rho \in \Hb\}$ is bounded for every $x \in \R$, %$\sup_{\color{blue}\overline\rho\in \Hb\cap C_c(\R)}|\sup_{x\in K}\rho(x)-\inf_{x\in K}\rho(x)|<\infty$.  
\end{enumerate}
then $v$ is the unique finite continuous viscosity solution of \eqref{eq_HJB}.
\end{theorem}

\begin{proof}
The first part of the conclusion follows by Theorem~\ref{thm2}, Theorem~\ref{thm3} and Remark~\ref{rem3}, and the second part by Theorem~\ref{T_comparison}. Note that condition \ref{itii0} implies condition \ref{T_comparison_2} of Theorem~\ref{T_comparison}, after taking $\Hb_c$ in the theorem, in place of $\Hb$.
\end{proof}

% \begin{remark}
%   The condition \ref{itiv} can be weakened provided, for example, the cost can be bounded below by a uniformly integrable martingale $M$, for each possible choice on the control. It is clear in the formulation that some lower bound on the cost is necessary to rule out degenerate solutions.
% \end{remark}

The equation \eqref{eq_HJB} above is a (degenerate) elliptic equation. To see this, write \eqref{eq_HJB} as
\[
H\Big(\mu,u(\mu),\frac{\partial^2 u}{\partial \mu^2}(\fdot,\fdot,\mu)\Big) = 0, \quad \mu\in\Pcal_p, 
\]
where the Hamiltonian $H$ is defined for measures $\mu\in\Pcal_p$, real numbers $r\in\R$, and functions $\varphi\colon {\R^d}\times {\R^d}\to\R$ by the formula
\begin{align*}
H(\mu,r,\varphi) = \beta r +  \sup_{\rho\in\Hb} \Big\{ - c(\mu,\rho) - \frac12 \int_{{\R^d}\times {\R^d}} & \varphi(x,y)(\rho(x) - \mu(\rho)) \\
&\times(\rho(y) - \mu(\rho)) \mu(\dx)\mu(\dy) \Big\},
\end{align*}
whenever this is well-defined. The Hamiltonian is (degenerate) elliptic in the sense that
\[
\varphi \succeq \psi \quad \Longrightarrow \quad H(\mu,r,\varphi) \le H(\mu,r,\psi),
\]
where the notation $\varphi \succeq \psi$ means that $\varphi-\psi$ is a positive definite function, that is,
\[
\int_{{\R^d}\times {\R^d}} (\varphi-\psi)(x,y) \nu(\dx)\nu(\dy) \ge 0
\]
for any signed measure $\nu$.

To avoid the need for any \emph{a priori} regularity of the value function, we work with a notion of viscosity solution that we now introduce. %We consider the following class of test functions:
Motivated by the fact that MVMs have decreasing support in the sense of \eqref{eq_supp_cont}, we define a partial order $\preceq$ on $\Pcal_p$ by
\[
\mu\preceq\nu \quad\Longleftrightarrow\quad \supp(\mu)\subseteq\supp(\nu).
\]
Thus Remark~\ref{rem:MVMCompact}\ref{it2iii} states that MVMs are decreasing with respect to this order. This means that the effective state space for an MVM starting at a measure $\bar\mu\in\Pcal_p$ is the set
\begin{equation}\label{eq_Dmu}
D_{\bar\mu} = \{\mu\in\Pcal_p\colon \mu\preceq\bar\mu\}.
\end{equation}
This set is weakly closed, and hence also closed in $\Pcal_p$, and it is worth mentioning that for Dirac masses, $D_{\delta_x} = \{\delta_x\}$ is a singleton. Equipped with the subspace topology inherited from $\Pcal_p$, $D_{\bar\mu}$ is a Polish space, and we may consider upper and lower semicontinuous envelopes of functions defined on $D_{\bar\mu}$. In particular, for any $u\colon\Pcal_p\to\overline\R$, the restriction of $u$ to $D_{\bar\mu}$ has semicontinuous envelopes given by
\begin{align*}
(u|_{D_{\bar\mu}})^*(\mu) &:= \limsup_{\nu\to\mu,\,\nu\preceq\bar\mu}u(\nu) \\
(u|_{D_{\bar\mu}})_*(\mu) &:= \liminf_{\nu\to\mu,\,\nu\preceq\bar\mu}u(\nu)
\end{align*}
for all $\mu\preceq\bar\mu$.

\begin{remark}
    Note that assumption \ref{itv} of Theorem~\ref{thetheorem} is a relatively strong requirement. However in some cases this can be checked directly, see for example Lemma 3.1 in \cite{cox2017}.
    On the contrary, assumption \ref{itiii} is often satisfied. For instance, this is always the case for
    $$\mathbb H:=\{\rho\in C(\R^d)\colon \rho(x)\leq  M(1+|x|^{p-q})\}$$
    for some $M>0$,
    when $\rho\mapsto c(\mu,\rho)$ is continuous along pointwise converging sequences in $\Hb$. %One may then take $\Hb_c = \Hb \cap C_c(\R^d)$.
\end{remark}

With this in mind, we now state our definition of viscosity solution. To keep things as transparent as possible, the definition is given without resorting to notation involving $D_{\bar\mu}$ and semicontinuous envelopes. Still, it is possible and technically useful to recast the definition in this language, and we will do so momentarily; see the discussion before Lemma~\ref{L_viscsol_strict} below. For any test function $f\in C^2(\Pcal_q)$, define $H(\fdot;f)\colon\Pcal_p\to\overline\R$ by
\begin{equation}\label{eq_Hmuf}
H(\mu;f) = \beta f(\mu) + \sup_{\rho\in\Hb}\left\{ - c(\mu,\rho) - Lf(\mu,\rho) \right\}.
\end{equation}
    We restrict our test functions to belong to the possibly smaller space $C^2(\Pcal_q) \subset C^2(\Pcal_p)$ in order to be able to apply the It\^o formula, Theorem~\ref{T_Ito}. This is crucial for proving that the value function is a viscosity solution.

We can now state the definition of viscosity solution.

\begin{definition}\label{D_visc_sol}
Consider a function $u\colon\Pcal_p\to\overline\R$.
\begin{itemize}
\item $u$ is a \emph{viscosity subsolution} of \eqref{eq_HJB} if
\[
\liminf_{\mu\to\bar\mu,\, \mu\preceq\bar\mu} H(\mu;f) \le 0
\]
holds for all $\bar\mu\in\Pcal_p$ and $f\in C^2(\Pcal_q)$ such that $f(\bar\mu)=\limsup_{\mu\to\bar\mu,\,\mu\preceq\bar\mu}u(\mu)$ and $f(\mu)\ge u(\mu)$ for all $\mu\preceq\bar\mu$.
\item $u$ is a \emph{viscosity supersolution} of \eqref{eq_HJB} if
\[
\limsup_{\mu\to\bar\mu,\, \mu\preceq\bar\mu} H(\mu;f) \ge 0
\]
holds for all $\bar\mu\in\Pcal_p$ and $f\in C^2(\Pcal_q)$ such that $f(\bar\mu)=\liminf_{\mu\to\bar\mu,\,\mu\preceq\bar\mu}u(\mu)$ and $f(\mu)\le u(\mu)$ for all $\mu\preceq\bar\mu$.
\item $u$ is a \emph{viscosity solution} of \eqref{eq_HJB} if it is both a viscosity subsolution and a viscosity supersolution.
\end{itemize}
\end{definition}

An equivalent way of expressing the subsolution property of $u$ is as follows: for any $\bar\mu\in\Pcal_p$ and $f\in C^2(\Pcal_q)$, one has the implication
\[
\text{$f(\bar\mu)=\hat u(\bar\mu)$ and $f|_{D_{\bar\mu}} \ge \hat u$} \quad\Longrightarrow\quad \check H(\bar\mu; f)\le0,
\]
where $\hat u = (u|_{D_{\bar\mu}})^*$ and $\check H(\fdot; f) = (H(\fdot;f)|_{D_{\bar\mu}})_*$. The analogous statement holds for supersolutions. 

\begin{remark}
If $u\in C(\Pcal_p)$ is a subsolution in the sense of Definition~\ref{D_visc_sol}, then it is also a subsolution in the sense that
$\liminf_{\mu\to\bar\mu} H(\mu;f) \le 0$
holds for all $\bar\mu\in\Pcal_p$ and $f\in C^2(\Pcal_q)$ such that $f(\bar\mu)=u(\bar\mu)$ and $f(\mu)\ge u(\mu)$ on $\Pcal_p$, and similarly for supersolutions. 
In order to obtain comparison, it is however crucial for us to work with the above definition which takes into account the partial ordering $\preceq$; \cf{} Lemma~\ref{L_HJB_N}.
\end{remark}

The following result shows that, as in finite-dimensional situations, it is enough to consider test functions that are strictly larger than $\hat u$ away from $\bar\mu$.

\begin{lemma}\label{L_viscsol_strict}
Assume that there is a constant $R\in \R_+$ such that 
\begin{equation}\label{eqn10}
|\rho(x)|\leq R(1+|x|^p)
\end{equation}
 for all $x\in \R^d$ and $\rho\in \Hb$.
Consider a function $u\colon\Pcal_p\to\overline\R$.
\begin{enumerate}
\item\label{L_viscsol_strict_1} $u$ is a viscosity subsolution of \eqref{eq_HJB} if and only if for any $\bar\mu\in\Pcal_p$ and $f\in C^2(\Pcal_q)$, one has the implication
\[
\text{$f(\bar\mu)=\hat u(\bar\mu)$ and $f(\mu) > \hat u(\mu)$ for all $\mu\in D_{\bar\mu}\setminus\{\bar\mu\}$} \quad\Longrightarrow\quad \check H(\bar\mu; f)\le0,
\]
where $\hat u = (u|_{D_{\bar\mu}})^*$ and $\check H(\fdot; f) = (H(\fdot;f)|_{D_{\bar\mu}})_*$. 
\item\label{L_viscsol_strict_2} $u$ is a viscosity supersolution of \eqref{eq_HJB} if and only if for any $\bar\mu\in\Pcal_p$ and $f\in C^2(\Pcal_q)$, one has the implication
\[
\text{$f(\bar\mu)=\check u(\bar\mu)$ and $f(\mu) < \check u(\mu)$ for all $\mu\in D_{\bar\mu}\setminus\{\bar\mu\}$} \quad\Longrightarrow\quad \hat H(\bar\mu; f)\ge0,
\]
where $\check u = (u|_{D_{\bar\mu}})_*$ and $\hat H(\fdot; f) = (H(\fdot;f)|_{D_{\bar\mu}})^*$.
\end{enumerate}
\end{lemma}

\begin{proof}
Unpacking the definitions, one finds that the properties in the lemma are weaker than the definitions of sub- and supersolution in Definition~\ref{D_visc_sol}. Therefore it is enough to prove the ``if'' statements. Consider \ref{L_viscsol_strict_1}, and assume $u$ satisfies the given property.
Note that for $\bar\mu\in\Pcal^s$, the implication trivially holds true since $D_{\bar\mu}=\{\bar\mu\}$.
Pick therefore $\bar\mu\in\Pcal_p\setminus\Pcal^s$ and $f\in C^2(\Pcal_q)$ such that $f(\bar\mu)=\limsup_{\mu\to\bar\mu,\,\mu\preceq\bar\mu}u(\mu)$ and $f(\mu)\ge u(\mu)$ for all $\mu\preceq\bar\mu$. We must show that $\liminf_{\mu\to\bar\mu,\, \mu\preceq\bar\mu} H(\mu;f) \le 0$.

To this end, for any $\varepsilon>0$ we consider the perturbed test function $f_\varepsilon = f + \varepsilon g$, where we define
\[
g(\mu) = \frac12 \int_{{\R^d}\times {\R^d}} e^{-(x-y)^2/2} (\mu - \bar\mu)(\dx)(\mu - \bar\mu)(\dy).
\]
We start by establishing some properties of $g$. First, $g$ belongs to $C^2(\Pcal_q)$ and its second derivative is $\frac{\partial^2 g}{\partial \mu^2}(x,y,\mu) = e^{-(x-y)^2/2}$. Next, using the identity
\[
e^{-x^2/2} = \int_\R e^{i\theta x} \gamma(\d \theta) \quad \text{where} \quad \gamma(\d \theta) = \frac{1}{\sqrt{2\pi}}e^{-\theta^2/2} \d \theta,
\]
we have for any finite signed measure $\nu$ that
\begin{align*}
\int_{{\R^d}\times {\R^d}} e^{-(x-y)^2/2} \nu(\dx) \nu(\dy) &= \int_{{\R^d}\times {\R^d}} \int_\R e^{i\theta (x-y)} \gamma(\d \theta) \nu(\dx) \nu(\dy) \\
&= \int_\R \Big| \int_{\R^d} e^{i\theta x} \nu(\dx) \Big|^2 \gamma(\d \theta).
\end{align*}
This implies that $g(\mu) > 0$ for every $\mu \ne \bar \mu$, and we clearly have $g(\bar\mu) = 0$. Moreover, the right-hand side is upper bounded by the squared total variation $\|\nu\|_\text{TV}^2$ of $\nu$. As a consequence, writing $\sigma(\dx) = (\rho(x) - \mu(\rho))\mu(\dx)$, we have
\[
Lg(\mu,\rho) = \frac12 \int_{{\R^d}\times {\R^d}} e^{-(x-y)^2/2} \sigma(\dx) \sigma(\dy) \le \frac12 \| \sigma \|_\text{TV}^2 \le 2 \mu(|\rho|)^2.
\]
Since condition \eqref{eqn10} is satisfied, it follows there is a constant $\r \in (0,\infty)$ such that
\[
\sup_{\rho \in \Hb} Lg(\mu,\rho) \le 2 \r^2 (1 + \mu(|\fdot|^p))^2.
\]

We now return to proving that $\liminf_{\mu\to\bar\mu,\, \mu\preceq\bar\mu} H(\mu;f) \le 0$. Using the perturbed test function $f_\varepsilon = f + \varepsilon g$ we have
\begin{align*}
H(\mu, f_\varepsilon) &= \beta f_\varepsilon(\mu) + \sup_{\rho \in \Hb} \{ - c(\mu,\rho) - Lf_\varepsilon(\mu,\rho) \} \\
&\ge \beta f(\mu) + \sup_{\rho \in \Hb} \{ - c(\mu,\rho) - Lf(\mu,\rho) \} - \varepsilon \sup_{\rho \in \Hb} Lg(\mu,\rho) \\
&\ge H(\mu; f) - 2 \varepsilon \r^2 (1 + \mu(|\fdot|^p))^2.
\end{align*}
Rearranging this gives
\[
H(\mu,f) \le H(\mu, f_\varepsilon) + 2 \varepsilon \r^2 (1 + \mu(|\fdot|^p))^2.
\]
Now, $f_\varepsilon$ satisfies $f_\varepsilon(\bar\mu) = \hat u(\bar \mu)$ and $f_\varepsilon(\mu) > \hat u(\mu)$ for all $\mu \le \bar\mu$ different from $\bar\mu$. Therefore, since $u$ satisfies the given property in \ref{L_viscsol_strict_1}, we get
\begin{align*}
\liminf_{\mu\to\bar\mu,\, \mu\preceq\bar\mu} H(\mu,f) &\le \liminf_{\mu\to\bar\mu,\, \mu\preceq\bar\mu} H(\mu, f_\varepsilon) + 2 \varepsilon \r^2 (1 + \bar\mu(|\fdot|^p))^2 \\
&\le 2 \varepsilon \r^2 (1 + \bar\mu(|\fdot|^p))^2.
\end{align*}
Since $\varepsilon>0$ was arbitrary, we obtain $\liminf_{\mu\to\bar\mu,\, \mu\preceq\bar\mu} H(\mu;f) \le 0$ as required.

The corresponding argument in the supersolution case is completely analogous, but uses the perturbed test function $f_\varepsilon = f - \varepsilon g$ instead.
\end{proof}

We next verify that with our definition of viscosity solution, every classical solutions is also a viscosity solution. The proof of this statement relies on the following positive maximum principle.

\begin{lemma}\label{PMP1}
Fix $\bar \mu\in \Pcal_p$, a measurable function $\bar\rho\colon {\R^d}\to\R$,  and  $f\in C^2(\Pcal_q)$ 
such that $Lf(\bar \mu,\bar  \rho)<\infty$. Suppose that
 $f(\bar \mu)=\max_{\mu\in D_{\bar \mu}}f(\mu)$. Then
$Lf(\bar \mu,\bar  \rho)\leq0.$
\end{lemma}

\begin{proof}
Assume first that $\bar \rho\in C_c({\R^d})$ and let $(\xi,\rho)$ be the weak solution of \eqref{eq_MVM_SDE} satisfying $\xi_0=\bar \mu$ and $\rho\equiv\bar \rho$ given by Theorem~\ref{thm5}. By Remark~\ref{rem:MVMCompact}\ref{it2iii} we know that $\xi_t\in D_{\bar \mu}$ for each $t$ almost surely.
Since \eqref{eq_Ito_assumption} is always satisfied for $\rho\in C_c({\R^d})$, an application of It\^o's formula yields
$$
\begin{aligned}
f(\xi_t) &= f(\bar\mu) + \int_0^t \int_{\R^d} \frac{\partial f}{\partial \mu}(x,\xi_s)\sigma_s(\dx) \d W_s 
+  \int_0^t  Lf(\xi_s, \bar\rho) \ds,
\end{aligned}
$$
where we write $\sigma_s(\dx)=(\bar\rho(x)-\xi_s(\bar \rho))\xi_s(\dx)$. 

Following the proof of \cite[Lemma~2.3]{FL:16}, assume that $Lf(\bar \mu, \bar\rho)>0$, consider the random time
$$\tau:=\inf\{s\geq0\colon Lf(\xi_s,\bar \rho)\leq0\},$$
and note that the continuity of $Lf(\fdot,\bar\rho)$ yields $\tau>0$. Letting $(\tau_n)_{n\in \N}$ be a localising sequence for $\int_0^\cdot \int_{\R^d} \frac{\partial f}{\partial \mu}(x,\xi_s)\sigma_s(\dx) \d W_s$ this implies
$$0\geq\E[f(\xi_{t\land\tau\land\tau_n})-f(\bar \mu)]
=\E\bigg[\int_0^{t\land\tau\land\tau_n}  Lf(\xi_s,\bar \rho) \ds
\bigg]> 0,$$
giving the necessary contradiction. A density argument allows us to extend this result to compactly supported measurable $\bar \rho$ first, and then to any measurable $\bar \rho$ such that $Lf(\bar \mu,\bar  \rho)<\infty$.
\end{proof}

\begin{definition}
A map
 $u\in  C^2(\Pcal_q)$ is called \emph{classical solution} of \eqref{eq_HJB} if $Lu( \rho,\mu)$ is well-defined and finite for each $\mu\in \Pcal_p$ and $\rho \in \Hb$, and $u$ satisfies \eqref{eq_HJB}.
\end{definition}

We can now prove that under mild conditions on $\Hb$ classical solutions are viscosity solutions.

\begin{proposition}\label{lem1}
Let $u$ be a classical solution of \eqref{eq_HJB} and assume that for each $f\in C^2(\Pcal_q)$ and $\mu\in \Pcal_p$, $Lf( \rho,\mu)< \infty$ for some $\rho \in \Hb$. Then $u$ is a viscosity solution of \eqref{eq_HJB}. 
\end{proposition}

\begin{proof}
We first prove that $u$ is a viscosity subsolution. Fix $\bar\mu\in\Pcal_p\setminus\Pcal^s$ and $f\in C^2(\Pcal_q)$ such that $f(\bar\mu)=\limsup_{\mu\to\bar\mu,\,\mu\preceq\bar\mu}u(\mu)=u(\bar \mu)$ and $f(\mu)\ge u(\mu)$ for all $\mu\preceq\bar\mu$.
Fix $\rho\in \Hb$ and note that $u-f\in  C^2(\Pcal_q)$ and attains its maximum over $D_{\bar \mu}$ at $\bar \mu$. Since $Lu( \rho,\bar \mu)<\infty$ for each $\rho \in \Hb$, Lemma~\ref{PMP1}  yields
$Lu(\bar \mu,\rho)-Lf(\bar \mu,\rho)\leq 0$ for each $\rho\in \Hb$. Using that $ H(\bar \mu;u)= 0$ we can thus compute
\begin{align*}
\liminf_{\mu\to\bar\mu,\, \mu\preceq\bar\mu} H(\mu;f) 
\leq H(\bar\mu;f) -H(\bar\mu;u)
%&=\sup_{\rho\in\Hb}\left\{ - c(\bar\mu,\rho) - Lf(\bar\mu,\rho) \right\}   -   \sup_{\rho\in\Hb}\left\{ - c(\bar\mu,\rho) - Lu(\bar\mu,\rho) \right\}\\
\leq \sup_{\rho\in\Hb}\left\{ Lu(\bar\mu,\rho) -Lf(\bar\mu,\rho)\right\}
\le 0.
\end{align*}
We now prove the supersolution property. Fix $f$ as before, replacing $f(\mu)\geq u(\mu)$ with $f(\mu)\le u(\mu)$, for all $\mu\preceq\bar\mu$. Fix $\bar \rho\in \Hb$ such that $Lf(\bar \rho,\bar \mu)<\infty$ and note that Lemma~\ref{PMP1} yields
$Lu(\bar \mu,\bar \rho)-Lf(\bar \mu,\bar\rho)\geq 0$.
Using that
$$-(\beta u(\bar \mu)-c(\bar \mu,\bar\rho)-Lu(\bar \mu,\bar\rho))\geq 0,$$
we can then compute
$$
\limsup_{\mu\to\bar\mu,\, \mu\preceq\bar\mu} H(\mu;f) 
\geq H(\bar\mu;f)\\
\geq  Lu(\bar\mu,\bar \rho) -Lf(\bar\mu,\bar\rho)\\
\geq 0,
$$
concluding the proof.
\end{proof}

We conclude this section with a verification theorem for classical solutions.
\begin{proposition}\label{prop1}
Consider a cost function $c$ satisfying condition \eqref{eqn9}.
Suppose that \eqref{eq_HJB} is satisfied for some $u\in  C^2(\Pcal_q)$
%such that $\int_0^t\E[Lu(\xi_s,\rho_s)e^{-\beta s}]ds$ is well defined for each admissible control $(\xi,\rho)$,
and let $v$ be the value function given in \eqref{eq_value_function}. Suppose that for some $\e>0$ it holds
\begin{equation}\label{eqn4}
\E[\sup_{t\geq0}|u(\xi_t)|e^{(\e-\beta) t}]<\infty
\end{equation}
 for each admissible control $(\xi,\rho)$. Then $u\leq v$. Moreover, given $\mu\in\Pcal_p$, if there exists an admissible control $(\xi^*,\rho^*)$ such that $\xi_0^*=\mu$ and
$$\rho_s^*\in{\textup{argmax}}_{\rho\in\Hb} \left\{ - c(\xi_s^*,\rho) - Lu(\xi_s^*,\rho) \right\},\quad \P\otimes \dt-a.e.,$$
then $(\xi^*,\rho^*)$ is an optimal control and $u(\mu)=v(\mu)$.
\end{proposition}
\begin{proof}
Fix an admissible control $(\xi,\rho)$ of \eqref{eq_MVM_SDE} with $\xi_0=\mu$. Define
$$\tau_n:=\inf\left\{t\geq0\colon\int_0^t  |Lu(\xi_s,\rho_s)|\ds> n\right\}$$
and note that an application of the It\^o formula yields
\begin{align*}
u(\mu)&=\int_0^t(\beta u(\xi_s)-Lu(\xi_s,\rho_s))e^{-\beta s} \ds
+e^{-\beta t}u(\xi_t)\\
&\qquad-\int_0^t \int_{\R^d} \frac{\partial u}{\partial \mu}(x,\xi_s)e^{-\beta s}(\rho_s(x)-\xi_s(\rho_s))\xi_s(\dx) \d W_s.
\end{align*}
Using $(\tau_n)_n$ as localising sequence we obtain
 $$u(\mu)=\E[e^{-\beta (t\land \tau_n)}u(\xi_{t\land \tau_n})]+\int_0^t\E[(\beta u(\xi_{s})-Lu(\xi_s,\rho_s))\bm1_{\{s\leq \tau_n\}}e^{-\beta s}] \ds,$$
which sending $t\to\infty$ yields 
 $$u(\mu)= \E[e^{-\beta  \tau_n}u(\xi_{ \tau_n})]+\int_0^\infty\E[(\beta u(\xi_{s})-Lu(\xi_s,\rho_s))\bm1_{\{s\leq \tau_n\}}e^{-\beta s}] \ds.$$
Using that $u$ satisfies \eqref{eq_HJB} we obtain
 \begin{equation}\label{eqn7}
 u(\mu)\leq \int_0^\infty\E[c(\xi_s,\rho_s)\bm1_{\{s\leq \tau_n\}}e^{-\beta s}] \ds+
\E[ u(\xi_{\tau_n})e^{-\beta\tau_n}] .
\end{equation}
Since $c$ satisfies \eqref{eqn9}, $u$ satisfies \eqref{eqn4}, and $\tau_n$ increases to infinity, the dominated convergence theorem and the monotone convergence theorem yield
 $$u(\mu)\leq \int_0^\infty\E[c(\xi_s,\rho_s)e^{-\beta s}] \ds.$$
 Since $(\xi,\rho)$ was arbitrary, we can conclude that $u(\mu)\leq v(\mu)$. Using that the inequality in \eqref{eqn7} holds with equality for $(\xi^*,\rho^*)$, the second claim follows as well.
\end{proof}

%
%
%
%\begin{remark}
%Suppose that $\Hb\subseteq \{\rho\in C(\R)\colon |\rho(x)|\leq c(1+|x|^r),\ c\in \R\}$ for some $r\leq p-q$.  Then 
%$\mu\mapsto Lf(\mu,\rho)$ is continuous for each $f\in C^2(\Pcal_q)$ and $\rho\in\Hb$. This follows directly by noting that
%$$Lf(\mu,\rho)=\frac12  \int_{{\R^d}\times {\R^d}} \frac{\partial^2 f}{\partial \mu^2}(x,y,\mu) \sigma(\dx)\sigma(\dy)$$
%for $\sigma(\dx)=(\rho(x)-\mu(\rho))\mu(\dx)$ and
%$$\rho(x),\ \frac{\partial^2 f}{\partial \mu^2}(x,y,\mu)\rho(x),\ \frac{\partial^2 f}{\partial \mu^2}(x,y,\mu)\rho(x)\rho(y)
%\leq C (1+|x|^{p}+|y|^{p}).$$
%
%\end{remark}
%
%
%
%

\section{Viscosity subsolution property}
\label{sec:visc-subs-prop}

\begin{theorem}\label{thm2}
Assume that conditions \ref{itii0}-\ref{itiii} of Theorem~\ref{thetheorem} are satisfied. Then the value function is a viscosity subsolution of \eqref{eq_HJB}.
%\
%
%********** 
%\begin{itemize}
%\item there is a constant $\r \in (0,\infty)$ such that $|\rho(x)| \le \r( 1 + |x|^p )$ holds for all $x \in \R$ and $\rho \in \Hb$;
%\item $\mu \mapsto c(\mu,\rho)$ is upper semi-continuous for every $\rho\in \Hb \cap C_c(\R)$;
%\item $c$ is bounded below;
%\item $\Hb$ can be replaced by $\Hb\cap C_c(\R)$ in \eqref{eq_Hmuf}.
%\end{itemize}
%Then the value function is a viscosity subsolution of \eqref{eq_HJB}.
\end{theorem}

\begin{proof}
Note first that for $\bar\mu\in\Pcal^s$, the subsolution property reduces to $\beta f(\bar\mu)\le\inf_{\rho\in \mathbb H}  c(\bar\mu,\rho)$, for all $f\in C^2(\Pcal_q)$ with $f(\bar\mu) = v(\bar\mu)$. If $v(\bar \mu)$ if infinite, this is vacuously satisfied. If $v(\bar \mu)$ is finite, this follows from the definition \eqref{eq_value_function} of $v$. For $\bar\mu\in\Pcal_p\setminus\Pcal^s$ we argue by contradiction, and suppose the viscosity subsolution property fails. Then, by conditions~\ref{itii0}, \ref{itiii} and Lemma~\ref{L_viscsol_strict}, there exist $f\in C^2(\Pcal_q)$ such that
\[
\text{$f(\bar\mu) = \hat v(\bar\mu)$ and $f(\mu) > \hat v(\mu)$ for all $\mu\in D_{\bar\mu}\setminus\{\bar\mu\}$}
\]
and
\[
\check H(\bar\mu; f) > 0,
\]
where $D_{\bar\mu}$ is given by \eqref{eq_Dmu}, $\hat v = (v|_{D_{\bar\mu}})^*$, and $\check H(\fdot; f) = (H(\fdot;f)|_{D_{\bar\mu}})_*$ with $H(\fdot;f)$ given by \eqref{eq_Hmuf}. In particular, we have $H(\bar\mu;f) > 0$. Therefore, due to condition \ref{itiii}, there exist $\bar\rho\in \Hb\cap C_c({\R^d})$ and $\kappa>0$ such that
\begin{equation}\label{eq_ineq1_subsol}
\beta f(\bar\mu) - c(\bar\mu,\bar\rho) - Lf(\bar\mu,\bar\rho) > \kappa.
\end{equation}
Define the set
\[
U = \{\mu\in\Pcal_p\setminus\Pcal^s \colon  \beta f(\mu) - c(\mu,\bar\rho) - Lf(\mu,\bar\rho) > \kappa \}.
\]
Thanks to \eqref{eq_ineq1_subsol} and since $f$ and $Lf(\fdot,\bar\rho)$ are continuous and $c(\fdot,\bar\rho)$ is upper semi-continuous by condition~\ref{itii}, the set $U$ is an open neighbourhood of $\bar\mu$.

Choose measures $\mu_n\in\Pcal_p$ with $\mu_n\to\bar\mu$, $\mu_n\preceq\bar\mu$, and $v(\mu_n)\to \hat v(\bar\mu)$. By discarding finitely many of the $\mu_n$, we may assume that $\mu_n\in U$ for all $n$. Since they form a convergent sequence, the $\mu_n$ together with their limit $\bar\mu$ form a compact subset of $\Pcal_p$. 
  Remark~\ref{rem:MVMCompact}\ref{dlvp} (De~la~Vall\'ee-Poussin) then
 gives the existence of a  measurable function $\varphi:\R^d\to\R_+$ such that
 \begin{equation}\label{eq_subsol_a_def}
a = \sup_n \mu_n(\varphi) \in (0,\infty)
\end{equation}
 and the set
$
K^\varphi_{2a} 
$
defined in \eqref{eqn5} is a compact subset of $\Pcal_p$ containing both $\mu_n$ and $\bar\mu$. Since $D_{\bar\mu}$ is closed in $\Pcal_p$, the set $K_\varphi:=K^\varphi_{2a} \cap D_{\bar\mu}$ is a compact subset of $D_{\bar\mu}$.

%***old***
%
%The criterion of de~la~Vall\'ee-Poussin then gives the existence of a continuous function $G: \R_+\to\R_+$ with $\lim_{t\to\infty}G(t)/t^p=\infty$ such that setting $\varphi(x):=G(|x|)$ it holds
%\begin{equation}\label{eq_subsol_a_def}
%a = \sup_n \mu_n(\varphi) \in (0,\infty).
%\end{equation}
%We define the set
%\[
%K_\varphi = \{ \mu\in D_{\bar\mu} \colon \mu(\varphi) \le 2 a\},
%\]
%which is a compact subset of $D_{\bar\mu}$. Indeed, the converse direction of de~la~Vall\'ee-Poussin shows that the set $\{ \mu\in\Pcal_p \colon \mu(\varphi) \le 2 a\}$ is compact in $\Pcal_p$. Its intersection with $D_{\bar\mu}$, which is closed in $\Pcal_p$, is therefore compact in $D_{\bar\mu}$.}

Fix $n$, and let $(\xi,\rho)$ be an admissible control with $\xi_0=\mu_n$ and $\rho_t\equiv\bar\rho$ (constant in time); this exists by Theorem~\ref{thm5} and satisfies \eqref{eqn3} because $\bar\rho$ belongs to $C_c({\R^d})$. Define the stopping time
\[
\tau = \inf\{ t\ge0\colon \text{$\xi_t \notin U$ or $\xi_t(\varphi) \ge 2a$}\} \wedge 1.
\]
Using the It\^o formula, we get that
\[
f(\xi_{t\wedge\tau}) - f(\mu_n) - \int_0^{t\wedge\tau} Lf(\xi_s,\bar\rho) \ds
\]
is a local martingale, and then so is
\begin{equation}\label{eq_subsol_supermg}
e^{-\beta t\wedge\tau} f(\xi_{t\wedge\tau}) - f(\mu_n) - \int_0^{t\wedge\tau} e^{-\beta s} ( Lf(\xi_s,\bar\rho) - \beta f(\xi_s)) \ds.
\end{equation}
In fact, \eqref{eq_subsol_supermg} is a supermartingale because it is bounded from below. To see this, note that $\xi_s \in U$ for all $s<\tau$, and that $\tau\le 1$. Therefore,
\begin{equation}\label{eq_subsol_bound2}
\begin{aligned}
e^{-\beta (t\wedge\tau)} & f(\xi_{t\wedge\tau})  - \int_0^{t\wedge\tau} e^{-\beta s} ( Lf(\xi_s,\bar\rho) - \beta f(\xi_s)) \ds \\
&\ge e^{-\beta ( t\wedge\tau)} f(\xi_{t\wedge\tau}) + \int_0^{t\wedge\tau} e^{-\beta s}c(\xi_s,\bar\rho) \ds + \kappa e^{-\beta} (t\wedge\tau).
\end{aligned}
\end{equation}
Since $\mu_n\preceq\bar\mu$, and since MVMs are decreasing with respect to $\preceq$, the process $\xi_{t\wedge\tau}$ takes values in the compact set $K_\varphi$. The right-hand side of \eqref{eq_subsol_bound2} is therefore bounded below by $\min(0,\inf_{\mu\in K_\varphi} f(\mu)) - \int_0^\infty e^{-\beta s} c(\xi_s,\bar{\rho})_- \ds$, where the second term is integrable by \eqref{eqn9}. This shows that \eqref{eq_subsol_supermg} is bounded from below and hence a supermartingale, as claimed.

The supermartingale property of \eqref{eq_subsol_supermg} and the inequality \eqref{eq_subsol_bound2} give
\begin{equation}\label{eq_subsol_bound3}
\begin{aligned}
f(\mu_n) &\ge \E\left[ e^{-\beta \tau} f(\xi_{\tau}) - \int_0^{\tau} e^{-\beta s} ( Lf(\xi_s,\bar\rho) - \beta f(\xi_s)) \ds \right] \\
&\ge \E\left[ e^{-\beta \tau} f(\xi_\tau) + \int_0^\tau e^{-\beta s} c(\xi_s,\bar\rho) \ds + \kappa e^{-\beta}  \tau \right].
\end{aligned}
\end{equation}
The definition of $\tau$ and the fact that $\xi_\tau\preceq\bar\mu$ imply that $\xi_\tau\in K_\varphi\setminus U$ on the event $A=\{\tau<1\}\cap\{\xi_\tau(\varphi) < 2 a\}$. Since $K_\varphi\setminus U$ is compact in $D_{\bar\mu}$ (and possibly empty, but then so is $A$) and does not contain $\bar\mu$, and since $f-\hat v$ is lower semicontinuous on $D_{\bar\mu}$, nonnegative, and zero only at $\bar\mu$, it follows that the quantity
\[
\varepsilon=\inf_{\mu\in K_\varphi\setminus U}(f-\hat v)(\mu)
\]
is strictly positive (infinite if $K_\varphi\setminus U$ is empty). We thus have
\[
\text{$f(\xi_\tau)\ge \hat v(\xi_\tau)+\varepsilon \ge v(\xi_\tau)+\varepsilon$ on $A$.}
\]
Moreover, $f(\mu)\ge v(\mu)$ for all $\mu\preceq\bar\mu$. Therefore, using again that $\xi_\tau\preceq\bar\mu$, we get
\begin{equation}\label{eq_subsol_bound4}
\begin{aligned}
e^{-\beta \tau}f(\xi_\tau) + \kappa e^{-\beta} \tau &\ge e^{-\beta \tau}v(\xi_\tau) + \varepsilon e^{-\beta} \bm 1_A + \kappa e^{-\beta} \bm1_{\{\tau=1\}} \\
&\ge  e^{-\beta \tau}v(\xi_\tau) + (\varepsilon\wedge\kappa) e^{-\beta} \bm1_{\{\xi_\tau(\varphi) < 2 a\}}.
\end{aligned}
\end{equation}
Combining \eqref{eq_subsol_bound3} and \eqref{eq_subsol_bound4} yields
\begin{equation}\label{eq_subsol_bound5}
\begin{aligned}
f(\mu_n) &\ge \E\left[ e^{-\beta \tau} v(\xi_\tau) + \int_0^\tau e^{-\beta s} c(\xi_s,\bar\rho) \ds \right] \\
&\qquad\qquad\qquad + (\varepsilon \wedge \kappa)e^{-\beta} \P(\xi_\tau(\varphi) < 2 a).
\end{aligned}
\end{equation}
Using Markov's inequality, the stopping theorem along with the fact that $\xi(\varphi)$ is a continuous martingale, and the choice of the constant $a$ in \eqref{eq_subsol_a_def}, we get
\[
\P(\xi_\tau(\varphi) \ge 2a) \le \frac{1}{2a} \E[\xi_\tau(\varphi)] = \frac{1}{2a} \mu_n(\varphi) \le \frac12.
\]
Combining this with \eqref{eq_subsol_bound5} and the dynamic programming principle (Theorem~\ref{thm_dpp_main_text}), we obtain
\[
f(\mu_n) \ge v(\mu_n) + \frac{\varepsilon \wedge \kappa}{2}e^{-\beta}.
\]
This holds for all $n$. Sending $n$ to infinity yields $\hat v(\bar\mu)\ge \hat v(\bar\mu) + \frac12(\varepsilon \wedge \kappa)e^{-\beta}$, which is the required contradiction.
\end{proof}

\section{Viscosity supersolution property}
\label{sec:visc-supers-prop}

\begin{theorem}\label{thm3}
Assume that conditions \ref{itii0} and \ref{itiii} of Theorem~\ref{thetheorem} are satisfied. Then the value function is a viscosity supersolution of \eqref{eq_HJB}.
%
%
%\
%
%*******
%  Assume that there is a constant $\r \in (0,\infty)$ such that $|\rho(x)| \le \r( 1 + |x|^p )$ holds for all $x \in {\R^d}$ and $\rho \in \Hb$.  Then the value function is a viscosity supersolution of \eqref{eq_HJB}.
\end{theorem}

\begin{proof}

Note first that for $\bar\mu\in\Pcal^s$, the subsolution property reduces to $\beta f(\bar\mu)\ge\inf_{\rho\in \mathbb H}  c(\bar\mu,\rho)$, for all $f\in C^2(\Pcal_q)$ with $f(\bar\mu) = v(\bar\mu)$. If $v(\bar \mu)$ if infinite, this is vacuously satisfied. If $v(\bar \mu)$ is finite, this follows from the definition \eqref{eq_value_function} of $v$. For $\bar\mu\in\Pcal_p\setminus\Pcal^s$ we argue by contradiction, and suppose the viscosity supersolution property fails. Then, by conditions~\ref{itii0}, \ref{itiii} and Lemma~\ref{L_viscsol_strict}, there exist $f\in C^2(\Pcal_q)$ such that
\[
\text{$f(\bar\mu) = \check v(\bar\mu)$ and $f(\mu) < \check v(\mu)$ for all $\mu\in D_{\bar\mu}\setminus\{\bar\mu\}$}
\]
and, for some $\kappa>0$,
\[
\hat H(\bar\mu; f) < -\kappa,
\]
where $D_{\bar\mu}$ is given by \eqref{eq_Dmu}, $\check v = (v|_{D_{\bar\mu}})_*$, and $\hat H(\fdot; f) = (H(\fdot;f)|_{D_{\bar\mu}})^*$ with $H(\fdot;f)$ given by \eqref{eq_Hmuf}. Define the set
\[
U = \{\mu\in D_{\bar\mu}\setminus\Pcal^s\colon \hat H(\mu; f) < -\kappa \}.
\]
This is an open neighborhood of $\bar\mu$ in $D_{\bar\mu}$ since $\hat H(\fdot; f)$ is upper semicontinuous on $D_{\bar\mu}$. The inequality $\hat H(\fdot; f)\ge H(\fdot;f)$ on $D_{\bar\mu}$ and the definition of $H$ imply that
\begin{equation}\label{eq_supersol_bound1}
\text{$ \beta f(\mu)-c(\mu,\rho) - Lf(\mu,\rho) < -\kappa$ for all $\mu\in U$ and all $\rho\in\Hb$.}
\end{equation}

Choose measures $\mu_n\in U$, $n\in\Nb$, with $\mu_n\to\bar\mu$ and $v(\mu_n)\to \check v(\bar\mu)$. As in the proof of the subsolution property, 
  Remark~\ref{rem:MVMCompact}\ref{dlvp} (De~la~Vall\'ee-Poussin) then
 gives the existence of a  measurable function $\varphi:\R^d\to\R_+$ such that
$$
a = \sup_n \mu_n(\varphi) \in (0,\infty)
$$
 and the set
$
K_\varphi:=K^\varphi_{2a} \cap D_{\bar\mu}
$
for $K^\varphi_{2a}$ as in \eqref{eqn5}  is a compact subset of $D_{\bar\mu}$ containing both $\mu_n$ and $\bar\mu$.

%
%de~la~Vall\'ee-Poussin gives a continuous function $G:\R_+\to\R_+$ with $\lim_{t\to\infty}G(t)/t^p=\infty$ such that setting $\varphi(x):=G(|x|)$ it holds
%\[
%a = \sup_n \mu_n(\varphi) \in (0,\infty).
%\]
%The set
%\[
%K_\varphi = \{ \mu\in D_{\bar\mu} \colon \mu(\varphi) \le 2 a\}
%\]
%is then a compact subset of $D_{\bar\mu}$. 

Fix $n\in\Nb$, and let $(\xi,\rho)$ be an arbitrary admissible control with $\xi_0=\mu_n$ and such that $\int_0^1 \left(c(\xi_s, \rho_s)\right)_+ \ds$ is integrable; in particular, \eqref{eqn3} is satisfied. Such controls exist since by assumption $v(\mu_n) < \infty$ for sufficiently large $n$. Define the stopping time
\[
\tau = \inf\{ t\ge0\colon \text{$\xi_t \notin U$ or $\xi_t(\varphi) \ge 2a$}\} \wedge 1.
\]
Using the It\^o formula, we get that
\begin{equation}\label{eq_supersol_submg}
e^{-\beta t\wedge\tau} f(\xi_{t\wedge\tau}) - f(\mu_n) - \int_0^{t\wedge\tau} e^{-\beta s} ( Lf(\xi_s,\rho_s) - \beta f(\xi_s)) \ds
\end{equation}
is a local martingale.  
In fact, \eqref{eq_supersol_submg} is a submartingale because it is bounded from above by an integrable random variable. To see this, note that $\xi_s \in U$ for all $s<\tau$ and that $\tau\le 1$. Therefore, due to \eqref{eq_supersol_bound1},
\begin{equation}\label{eq_supersol_bound2}
\begin{aligned}
e^{-\beta t\wedge\tau} & f(\xi_{t\wedge\tau})  - \int_0^{t\wedge\tau} e^{-\beta s} ( Lf(\xi_s,\rho_s) - \beta f(\xi_s)) \ds \\
&\le e^{-\beta t\wedge\tau} f(\xi_{t\wedge\tau}) + \int_0^{t\wedge\tau} e^{-\beta s}c(\xi_s,\rho_s) \ds - \kappa e^{-\beta} (t\wedge\tau).
\end{aligned}
\end{equation}
Since $\xi_{t\wedge\tau}$ takes values in the compact set $K_\varphi$, the right-hand side is bounded above by 
\begin{equation*}
\max(0,\sup_{\mu\in K_\varphi} f(\mu)) + \int_0^1 \left(c(\xi_s, \rho_s)\right)_+ \ds.
\end{equation*} 
The first term is finite since $K_\varphi$ is compact and $f$ is continuous, and the second term is finite in expectation by our assumption on the chosen control. This shows that \eqref{eq_supersol_submg} is a submartingale, as claimed.

The submartingale property of \eqref{eq_supersol_submg} and the inequality \eqref{eq_supersol_bound2} give
\begin{equation}\label{eq_supersol_bound3}
\begin{aligned}
f(\mu_n) &\le \E\left[ e^{-\beta \tau} f(\xi_{\tau})  - \int_0^{\tau} e^{-\beta s} ( Lf(\xi_s,\rho_s) - \beta f(\xi_s)) \ds \right] \\
&\le \E\left[  e^{-\beta \tau} f(\xi_\tau) + \int_0^\tau e^{-\beta s} c(\xi_s,\rho_s) \ds - \kappa e^{-\beta} \tau \right].
\end{aligned}
\end{equation}
Moreover, the same reasoning that lead to \eqref{eq_subsol_bound4}, but now using lower semicontinuity on $D_{\bar\mu}$ of $\check v-f$, gives
\begin{equation}\label{eq_supersol_bound6}
e^{-\beta \tau}f(\xi_\tau) - \kappa e^{-\beta} \tau \le e^{-\beta \tau} v(\xi_\tau) - (\varepsilon\wedge\kappa)e^{-\beta}\bm1_{\{\xi_\tau(\varphi) < 2a\}}
\end{equation}
where
\[
\varepsilon = \inf_{\mu\in K_\varphi\setminus U} (\check v-f)(\mu) \in (0,\infty].
\]
We also have, as before, the bound $\P(\xi_\tau(\varphi) < 2a)\ge\frac12$. Combining this with \eqref{eq_supersol_bound3} and \eqref{eq_supersol_bound6} yields
\[
f(\mu_n) \le \E\left[ e^{-\beta \tau} v(\xi_\tau) + \int_0^\tau e^{-\beta s} c(\xi_s,\rho_s) \ds \right] - \frac{\varepsilon\wedge\kappa}{2}e^{-\beta}.
\]
Taking the infimum over all admissible controls $(\xi,\rho)$ with $\xi_0=\mu_n$, and using the dynamic programming principle (Theorem~\ref{thm_dpp_main_text}), we obtain
\[
f(\mu_n) \le v(\mu_n) - \frac{\varepsilon\wedge\kappa}{2}e^{-\beta}.
\]
This holds for all $n$. Sending $n$ to infinity yields $\check v(\bar\mu)\le \check v(\bar\mu) - \frac12(\varepsilon \wedge \kappa)e^{-\beta}$, which is the required contradiction.
\end{proof}

\begin{remark}
An inspection of the proof shows that the assumptions of Theorem~\ref{thm3} can be relaxed to the assumptions of    Lemma~\ref{L_viscsol_strict}.
\end{remark}

\section{Comparison principle}
\label{sec:comparison-principle}

\begin{theorem}\label{T_comparison}
Let $\beta > 0$, and suppose that the cost function $c$ and the action space $\Hb$ satisfy the following conditions:
%compact set $K\subset\R$,
\begin{enumerate}
\item\label{T_comparison_1} $\mu\mapsto c(\mu,\rho)$ is continuous on $\Pc(\{x_1,...,x_N\})$ uniformly in $\rho\in\Hb$ for any $N\in\Nb$ and $x_1,...,x_N\in {\R^d}$;
%\item $\sup_{\rho \in \Hb} |\max\{\rho(x_i)\colon i=1,...,N\} - \min\{\rho(x_i)\colon i=1,...,N\}| < \infty$ 
%$|\sup_{i=1,...,N}\rho(x_i)-\inf_{i=1,...,N}\rho(x_i)|$ 
%$|\sup_{x\in K}\rho(x)-\inf_{x\in K}\rho(x)|$ 
%is uniformly bounded for $\rho\in\Hb$.
\item\label{T_comparison_2} the set $\{\rho(x) - \rho(0) \colon \rho \in \Hb\}$ is a bounded subset of $\R^d$ for every $x \in {\R^d}$.
\end{enumerate}
Let $u,v\in C(\Pcal_p)$ be a viscosity sub- and supersolution of \eqref{eq_HJB}, respectively, for some $q \in [1,p] \cup \{0\}$. Then $u\le v$ on $\Pcal_p$.
\end{theorem}

The proof of Theorem~\ref{T_comparison} proceeds by reducing the problem to a comparison result for a PDE on a finite-dimensional space. We now describe this reduction. For any $N\in\N$, denote the standard $(N-1)$-simplex in $\R^N$ by
\[
\Delta^{N-1} = \{(p_1,\ldots,p_N) \in [0,1]^N\colon p_1+\cdots+p_N=1\}.
\]
Given $N$ points $x_1,\ldots,x_N\in {\R^d}$, there is a natural bijection between measures $\mu\in\Pcal(\{x_1,\ldots,x_N\})$ and points $p\in\Delta^{N-1}$, given by
\[
\mu = p_1\delta_{x_1} + \cdots + p_N\delta_{x_N}.
\]
%Any function $f\colon\Pcal\to\R$ induces a function $\tilde f\colon\Delta^{N-1}\to\R$ defined by
%\[
%\tilde f(p_1,\ldots,p_N) = f(p_1\delta_{x_1} + \cdots + p_N\delta_{x_N}).
%\]
%If $f$ is sufficiently smooth, for example if $f\in\Acal$, then $\tilde f$ is the restriction to $\Delta^{N-1}$ of a smooth function on $\R^N$, which we again denote by $\tilde f$ (there are many ways to define $\tilde f$ outside $\Delta^{N-1}$ but the particular choice will not matter.) One verifies that the first and second order derivatives of $f$ satisfy
%\[
%\frac{\partial f}{\partial \mu}(x_i,\mu) = \frac{\partial \tilde f}{\partial p_i}(p_1,\ldots,p_N)
%\]
%and
%\[
%\frac{\partial^2 f}{\partial \mu^2}(x_i,x_j,\mu) = \frac{\partial^2 \tilde f}{\partial p_i\partial p_j}(p_1,\ldots,p_N)
%\]
%for all $\mu=p_1\delta_{x_1}+\cdots+p_N\delta_{x_N}\in\Pcal(\{x_1,\ldots,x_N\})$ and all $i,j=1,\ldots,N$.
In particular, any given function $u\colon \Pc_p\to\overline\R$ induces a function $\tilde u\colon\Delta^{N-1}\to\overline\R$ defined by
\begin{equation}\label{eq_utilde}
\tilde u(p_1,\ldots,p_N) = u(p_1\delta_{x_1} + \cdots + p_N\delta_{x_N}).
%\quad p \in \Delta^{N-1}.
\end{equation}
If $u$ is a viscosity solution of \eqref{eq_HJB}, it turns out that $\tilde u$ is a viscosity solution of a certain equation on the simplex. To specify this, for $\rho\in\Hb$ and $p\in\Delta^{N-1}$, let
\begin{align*}
\tilde c(p,\rho)=c(p_1\delta_{x_1}+\cdots+p_N\delta_{x_N},\rho).
\end{align*}
Further, for $\rho\in\Hb$, let $\tilde\rho=(\rho(x_1),\ldots,\rho(x_N))$, and consider the operator $\tilde L$ defined for $\tilde f\in C^2(\R^N)$ by
\[
\tilde L\tilde f(p,\rho) = \frac12 \sum_{i,j=1}^N \frac{\partial^2 \tilde f}{\partial p_i\partial p_j}(p)(\tilde\rho_i - p\cdot\tilde\rho)(\tilde\rho_j - p\cdot\tilde\rho)p_i p_j,
\]
where $p\in\Delta^{N-1}$ and $p\cdot\tilde\rho$ is the inner product between the two vectors. One readily verifies that $\tilde L\tilde f_1(p,\rho)=\tilde L\tilde f_2(p,\rho)$ if $\tilde f_1(x)=\tilde f_2(x)$ for each $x\in \Delta^{N-1}$.

The relevant equation on the simplex then takes the following form: 
\begin{equation}\label{eq_HJB_N}
\beta \tilde u(p) + \sup_{\rho\in\Hb}\big\{ - \tilde c(p,\rho) - \tilde L\tilde u(p,\rho)\big\} = 0,\quad p\in\Delta^{N-1}.
\end{equation} 
We note that \eqref{eq_HJB_N} equivalently can be written as 
\begin{align}\label{eq_HJB_N_H}
\tilde H\left(p,\tilde u(p),D^2\tilde u(p)\right)=0,\quad p\in\Delta^{N-1},
\end{align}
where, for any $p\in\Delta^{N-1}$, $r\in\R$ and symmetric $N\times N$-matrix $P$,
\begin{align*}
\tilde H(p,r,P)=\beta r+\sup_{\rho\in\Hb}\Big\{-\tilde c(p,\rho)-\frac12 z(p,\rho)^TPz(p,\rho)\Big\},
\end{align*}
with $z(p,\rho)\in\R^N$ given by $z(p,\rho)_i=p_i(\tilde\rho_i-p\cdot\tilde\rho)$, for $i=1,...,N$, and for all $\rho\in\Hb$.

\begin{lemma}\label{L_HJB_N}
Suppose that the assumptions of Theorem~\ref{T_comparison} hold.
Let $u\in C(\Pc_p)$ be a viscosity subsolution (resp. supersolution) of \eqref{eq_HJB} for some $q \in [1,p] \cup \{0\}$. Let $N\in\Nb$ and let $x_1,...,x_N$ be distinct points in ${\R^d}$. Define $\tilde u\in C(\Delta^{N-1})$ by \eqref{eq_utilde}. Then $\tilde u$ is a viscosity subsolution\footnote{In the sense of \cite[Definition~2.2]{crandall1992}; see also Remark~2.3 therein.} (resp. supersolution) of \eqref{eq_HJB_N}.
\end{lemma}

\begin{proof}
We consider only the subsolution case. Pick any point $\bar p\in\Delta^{N-1}$ and a function $\tilde f\in C^2(\R^N)$ such that $\tilde f(\bar p)=\tilde u(\bar p)$ and $\tilde f\ge\tilde u$ on $\Delta^{N-1}$; we first show that
\begin{align}\label{eq:proof_liminf_viscosity}
\liminf_{p\to\bar p,\, p\in\Delta^{N-1}}\beta \tilde f(p) + \sup_{\rho\in\Hb}\big\{ - \tilde c(p,\rho) - \tilde L\tilde f(p,\rho)\big\}
\le 0.
\end{align}

Define a $C^2$ cylinder function by
\[
f(\mu) = \tilde f(\mu(\varphi_1),\ldots,\mu(\varphi_N)),
\quad \mu\in\Pc_q,
\]
where the $\varphi_i\in C_b$ are chosen so that $\varphi_i(x_i)=1$ and $\varphi_i(x_j)=0$ for $j\ne i$. Define also the measure
\[
\bar\mu =\bar p_1\delta_{x_1} + \cdots + \bar p_N\delta_{x_N}\in\Pc_p.
\]
Any $\mu\preceq\bar\mu$ is then an element of $\Pcal(\{x_1,\ldots,x_N\})$ and therefore of the form $\mu=p_1\delta_{x_1}+\cdots+p_N\delta_{x_N}$ with $p=(p_1,\ldots,p_N)\in\Delta^{N-1}$. Note that $f(\mu) = \tilde f(p)$. Moreover, in view of the expression \eqref{eq_C2r_2n_der_cyl_ex} for the derivative of a $C^2$ cylinder function, we have that
\[
\frac{\partial^2 f}{\partial \mu^2}(x_i,x_j,\mu) = \frac{\partial^2 \tilde f}{\partial p_i\partial p_j}(p_1,\ldots,p_N),
\quad i,j=1,\ldots,N;
\]
hence, $Lf(\mu,\rho)=\tilde L\tilde f(p,\rho)$, $\rho\in\Hb$. Since, with the above identification, the Wasserstein distance is equivalent to the Euclidean distance on $\Delta^{N-1}$, we thus obtain
\begin{align}\label{eq:proof_reduced_visc_sln}
\liminf_{p\to\bar p,\, p\in\Delta^{N-1}}\beta \tilde f(p) + \sup_{\rho\in\Hb}\big\{ - \tilde c(p,\rho) - \tilde L\tilde f(p,\rho)\big\}
\le \liminf_{\mu\to\bar\mu,\, \mu\preceq\bar\mu} H(\mu;f).
\end{align}
Further, note that for any $\mu\preceq\bar\mu$, identified as above with a point $p\in\Delta^{N-1}$,
\[
f(\mu) = \tilde f(p) \ge \tilde u(p) = u(\mu);
\]
in particular, $f(\bar\mu)=u(\bar\mu)$. Using that $u=\hat u$, the fact that $u$ is a viscosity subsolution of \eqref{eq_HJB}, and the inequality \eqref{eq:proof_reduced_visc_sln}, we thus obtain \eqref{eq:proof_liminf_viscosity}.

Comparing \eqref{eq_HJB_N} and \eqref{eq_HJB_N_H}, we now see that in order to conclude, it suffices to establish continuity of the mapping $(p,r,P)\mapsto\tilde H(p,r,P)$. 
To this end, note first that an elementary calculation gives $\|z(p,\rho)-z(q,\rho)\|
%&\le 2\left(\|\tilde\rho\|^2\|p-q\|^2+\|(p\cdot\tilde\rho)p-(q\cdot\tilde\rho)q\|^2\right)\\
%&\le 4\left(\|\tilde\rho\|^2\|p-q\|^2+\|(p\cdot\tilde\rho)(p-q)\|^2+\|((p-q)\cdot\tilde\rho)q\|^2\right)\\
%&\le 4\left(\|\tilde\rho\|^2\|p-q\|^2+(p\cdot\tilde\rho)^2\|p-q\|^2+((p-q)\cdot\tilde\rho)^2\|q\|^2\right)\\
\le 3\|\tilde\rho\|\|p-q\|$, for any $p,q\in\Delta^{N-1}$ and $\rho\in\Hb$.
Since $z(p,\rho)$ is invariant with respect to parallel shifts of $\rho$, and thanks to assumption \ref{T_comparison_2} of Theorem~\ref{T_comparison}, this implies
\begin{align}\label{eq:revision_bound_norm_z}
\|z(p,\rho)-z(q,\rho)\| \le 3\|\tilde\rho - \rho(0) \|\|p-q\| \le \kappa\|p-q\|
\end{align}
for some constant $\kappa>0$. 
A similar argument gives that $(p,\rho)\mapsto\|z(p,\rho)\|$ is bounded on $\Delta^{N-1}\times\Hb$. 
Hence, there exists some constant $\delta>0$, such that for any $\rho\in\Hb$, $p,q\in\Delta^{N-1}$ and symmetric $N\times N$-matrices $P,Q$,
\begin{align*}
     \left|z(q,\rho)^TQz(q,\rho)-z(p,\rho)^TPz(p,\rho)\right|
     %& \le\left(\|z(p,\rho)\|+\|z(q,\rho)\|\right)\|P\|\|z(p,\rho)-z(q,\rho)\|+\|z(q,\rho)\|^2\|P-Q\|\\
     \le\delta\left(\|P\|\|p-q\|+\|P-Q\|\right),
\end{align*}
where $\|\cdot\|$ denotes the operator norm for symmetric $N\times N$-matrices.
In consequence, for any $p,q\in\Delta^{N-1}$, $r,s\in\R$ and symmetric $N\times N$-matrices $P,Q$,
\begin{align}\label{eq:revision_estimate_hamiltonian}
    \big|\tilde H(q,s,Q)-\tilde H(p,r,P)\big|
    \le & \;\beta \left|s-r\right|
    +\sup_{\rho\in\Hb}\left|\tilde c(q,\rho)-\tilde c(p,\rho)\right|\nonumber\\
    &+\frac12\sup_{\rho\in\Hb}\left|z(q,\rho)^TQz(q,\rho)-z(p,\rho)^TPz(p,\rho)\right|\\
    \le & \;\beta|r-s|+\omega(\|p-q\|)+\frac{\delta}{2}\left(\|P\|\|p-q\|+\|P-Q\|\right),\nonumber
\end{align}
where $\omega$ is a modulus of continuity which only depends on $c$. Such a modulus exists thanks to condition \ref{T_comparison_1} of Theorem~\ref{T_comparison}.
This establishes the continuity of $\tilde H$ and the proof is complete.
\end{proof}

\begin{lemma}\label{L_comparison_N}
Suppose that the assumptions of Theorem~\ref{T_comparison} hold. Let $N\in\Nb$ and let $x_1,...,x_N$ be distinct points in ${\R^d}$. Then the comparison principle holds for the PDE \eqref{eq_HJB_N}. Specifically, if $\tilde u,\tilde v\in C(\Delta^{N-1})$ are viscosity sub- and supersolutions of \eqref{eq_HJB_N}, respectively, then $\tilde u\le \tilde v$ on $\Delta^{N-1}$. 
\end{lemma}

\begin{proof}
Recall that equation \eqref{eq_HJB_N} equivalently can be written in the form \eqref{eq_HJB_N_H}. Let $\tilde u,\tilde v\in C(\Delta^{N-1})$ be viscosity sub- and supersolutions of \eqref{eq_HJB_N_H}, respectively. For any $\alpha>0$, define
	\begin{align*}
		M_\alpha=\sup_{\Delta^{N-1}\times\Delta^{N-1}}
		\Big(\tilde u(p)-\tilde v(q)-\frac{\alpha}{2}\|p-q\|^2\Big);
	\end{align*}
since $\tilde u-\tilde v$ is continuous and $\Delta^{N-1}$ is compact, $M_\alpha<\infty$ is attained for some $(p_\alpha,q_\alpha)$. According to \cite[Lemma~3.1~(i)]{crandall1992}, $\alpha\|p_\alpha-q_\alpha\|^2\to 0$ as $\alpha\to\infty$. 
	
Recall from the proof of Lemma \ref{L_HJB_N} that $\tilde H$ is continuous. Applying \cite[Theorem~3.2; see also Remark~2.4 and equation~(3.10)]{crandall1992}, and using that $\tilde u$ and $\tilde v$ are viscosity sub- and supersolutions of \eqref{eq_HJB_N_H}, we deduce the existence of two symmetric $N\times N$-matrices $P_\alpha,Q_\alpha$ such that
	\begin{align}\label{eq_ellipticity_ishii}
		\tilde H\left(p_\alpha,\tilde u(p_\alpha),P_\alpha\right) 
		\le 0 \le 
		\tilde H\left(q_\alpha,\tilde v(q_\alpha),Q_\alpha\right)
	\end{align}
and
    \begin{align*}
        z(p_\alpha,\rho)^TP_\alpha z(p_\alpha,\rho)-z(q_\alpha,\rho)^TQ_\alpha z(q_\alpha,\rho)\le 3\alpha \|z(p_\alpha,\rho)-z(q_\alpha,\rho)\|^2,
        \textrm{ for all $\rho\in\Hb$}.
    \end{align*}
Making use of \eqref{eq:revision_bound_norm_z} and estimates similar to \eqref{eq:revision_estimate_hamiltonian}, we obtain from the latter property that for each $r\in\R$,
\begin{align}\label{eq_assump_comparison}
	\tilde H(q_\alpha,r, & Q_\alpha)-\tilde H(p_\alpha,r,P_\alpha)\nonumber\\
	&\le \sup_{\rho\in\Hb}\Big\{\tilde c(p_\alpha,\rho)-\tilde c(q_\alpha,\rho)+\frac12 \Big(z(p_\alpha,\rho)^TP_\alpha z(p_\alpha,\rho)-z(q_\alpha,\rho)^TQ_\alpha z(q_\alpha,\rho)\Big)\Big\}\nonumber\\
	&\le \omega(\|p_\alpha-q_\alpha\|)+3\alpha\kappa\|p_\alpha-q_\alpha\|^2,
\end{align}
where $\kappa>0$ is a constant and $\omega$ is a modulus of continuity which only depends on $c$. 

In order to conclude, suppose contrary to the claim that there exists some $\bar p\in\Delta^{N-1}$ with $\tilde u(\bar p)>\tilde v(\bar p)$. Then, there exists $\delta>0$ such that for all $\alpha>0$,
	\begin{align*}
		M_\alpha\ge \tilde u(\bar p)-\tilde v(\bar p)>\delta.
	\end{align*}
For each $\alpha>0$, using \eqref{eq_ellipticity_ishii} and, in turn, \eqref{eq_assump_comparison}, we thus obtain
\begin{align*}
	\beta\delta 
	\le \beta\left(\tilde u(p_\alpha)-\tilde v(q_\alpha)\right)
	&= \tilde H\left(p_\alpha,\tilde u(p_\alpha),P_\alpha\right)-\tilde H\left(p_\alpha,\tilde v(q_\alpha),P_\alpha\right)\\
	&\le \tilde H\left(q_\alpha,\tilde v(q_\alpha),Q_\alpha\right)-\tilde H\left(p_\alpha,\tilde v(q_\alpha),P_\alpha\right)\\
	&\le \omega(\|p_\alpha-q_\alpha\|)+3\kappa\alpha\|p_\alpha-q_\alpha\|^2,
\end{align*}
and sending $\alpha\to\infty$ yields the desired contradiction.
\end{proof}

\begin{proof}[Proof of Theorem~\ref{T_comparison}]
Let $u,v\in C(\Pc_p)$ be a viscosity sub- and supersolution of \eqref{eq_HJB}, respectively. 
It suffices to argue that $u(\mu)\le v(\mu)$ for any finitely supported $\mu\in\Pc_p$. Indeed, since the finitely supported measures are dense in $\Pc_p$, for an arbitrary $\mu\in\Pc_p$, we can pick a sequence of finitely supported $\mu_n$ with $\mu_n\to\mu$, and then use the continuity of $u$ and $v$ to obtain 
	\begin{align*}
 		(u-v)(\mu)=\lim_{n\to\infty}(u-v)(\mu_n)\le 0.
 	\end{align*} 	
Let therefore $\mu\in\Pcal(\{x_1,\ldots,x_N\})$ for some distinct points $x_1,\ldots,x_N\in {\R^d}$, $N\in\N$. By Lemma~\ref{L_HJB_N}, the functions $\tilde u,\tilde v\in C(\Delta^{N-1})$ defined by
\begin{align*}
\tilde u(p_1,\ldots,p_N) &= u(p_1\delta_{x_1} + \cdots + p_N\delta_{x_N}),\\
\tilde v(p_1,\ldots,p_N) &= v(p_1\delta_{x_1} + \cdots + p_N\delta_{x_N}),
\end{align*}
are viscosity sub- and supersolutions of \eqref{eq_HJB_N}, respectively. Thus, by Lemma~\ref{L_comparison_N}, $\tilde u\le\tilde v$ on $\Delta^{N-1}$, or equivalently, $u\le v$ on $\Pcal(\{x_1,\ldots,x_N\})$. Hence $u(\mu)\le v(\mu)$ and we conclude. 
\end{proof}

\section{Applications}
\label{sec:applications}

We here give some concrete examples of solvable control problems which can be addressed using the framework set out in this article. In particular, we explain how our main results relate to the applications which were described in the introduction. In sections \ref{sec:SEP} to \ref{sec:ZSG}, we summarise potential applications at a general level. The results we have presented may not be directly applicable, and would potentially require modified versions of our control problems which would include e.g.~time-dependent cost functions, or cost functions which depend on additional (possibly controlled) processes. This would allow extensions of our arguments to e.g.~finite horizon examples. We anticipate that the previous results will extend to these cases with little adaptation, but we leave formal justification of these arguments to future work.

    \subsection{An abstract control problem}
The goal of this subsection is to illustrate the versatility of our methods by considering two toy examples that we solve explicitly. We rely on several results provided in this paper including the verification theorem (Proposition~\ref{prop1}), the existence theorem (Theorem~\ref{thm5}), and the comparison principle (Theorem~\ref{T_comparison}). The results are derived at the end of the subsection from a general technical result, Theorem~\ref{thm1}.

\begin{example}\label{ex91}
Fix $q=0$, a constant $C>0$, a set of actions $\Hb$ such that $|\rho(x)|\leq C(1+|x|^{p/2})$ for each $\rho\in\Hb$ and $x\in {\R^d}$,
a discount rate $\beta>0$, and  two functions $\varphi \in C_b({\R^d})$ and $\bar\rho\in\Hb$.
 For some $\alpha\geq0$ define
\begin{equation}\label{eqn1}
c(\mu,\rho):=\mu(\varphi )^2+\alpha\Var_\mu(\bar\rho-\rho)-\frac 1 \beta \Cov_\mu(\varphi ,\rho)^2.
\end{equation}
Then the corresponding stochastic optimal problem can be solved explicitly. The corresponding value function is  the unique continuous viscosity solution of \eqref{eq_HJB} and is given by
$$
\frac 1 \beta \mu(\varphi )^2 
=\inf\left\{  \E\left[ \int_0^\infty e^{-\beta t} c(\xi_t,\rho_t) \dt \right] \colon \text{$(\xi,\rho)$ admissible control, $\xi_0=\mu$}\right\}.
$$
Moreover, there exists an optimal control $(\xi^*,\rho^*)$ satisfying $\rho_s^*=\bar\rho$ for a.e.~$s\geq 0$. 
The three terms of the cost function \eqref{eqn1} can be interpreted as follows.
\begin{itemize}
\item $\mu(\varphi)^2$: If $\varphi $ is nonnegative this term penalises controls $\xi$ putting mass on regions where $\varphi $ is large. For a general $\varphi $ this term would be an incentive in choosing controls $\xi$ which are balanced with respect to $\varphi $. For example, for $d=1$, choosing $\varphi (x)=x$  penalises non-centered controls $\xi$.
\item $\alpha\Var_\mu(\bar\rho-\rho)$: This term penalises controls $\rho$ which deviate from a given target $\bar\rho$. Deviations in regions where the corresponding MVM $\xi$ is more concentrated are penalised more severely.
\item $-\frac 1 \beta \Cov_\mu(\varphi ,\rho)^2$: Since $-\frac 1 \beta \Cov_\mu(\varphi ,\rho)^2=-\frac 1 \beta \Corr_\mu(\varphi ,\rho)^2\Var_\mu(\varphi )\Var_\mu(\rho)$, we can see that this term penalises uncorrelation between $\varphi $ and $\rho$  and incentives the variance of $\rho$ with respect to $\xi$.
\end{itemize}

This example can be generalised by letting $\bar\rho$ depend on $\mu$ and requiring  $(\xi, \bar\rho_\xi)$ to be an admissible control for some continuous MVM $\xi$. The optimal control $(\xi,\rho)$ would satisfy $\rho_s=\bar\rho_{\xi_s}$. It is also possible to relax the boundedness condition on $\varphi $ by imposing a lower bound on the parameter $p$. 
\end{example}

%%% Example with identity
\begin{example}\label{ex92}
Fix $d=1$, $p\geq 4$, $q=1$, a state dependent set of actions
$$\Hb(\mu):=\{\rho\in \Hb\colon \Var_\mu(\rho)\leq\Var(\mu)\}$$
for some $\Hb$ such that $\id\in \Hb$, and
a discount rate $\beta>0$.
Define 
$$
c(\mu):= \Var(\mu)^2 - \beta\Mb(\mu)^2.
$$
Then the corresponding control problem can be solved explicitly and the associated value function is  given by
$$
 -\Mb(\mu)^2
=\inf\left\{  \E\left[ \int_0^\infty e^{-\beta t} c(\xi_t) \dt \right] \colon \text{$(\xi,\rho)$ admissible control, $\xi_0=\mu$}\right\}.
$$
Moreover, the optimal control $(\xi^*,\rho^*)$ satisfies $\rho_t^* =
\id$ $\xi_t^* \otimes \dt$-almost surely. % $\xi_s^*(\rho_s^*)=\Mb(\xi_s^*)$ for a.e.~$s\geq 0$.

If one instead considers the cost function $c(\mu)= \Var(\mu)^2 + \beta\Var(\mu)$, the optimiser remains the same but the value of the problem then equals $\Var(\mu)$.

We observe that the MVM we construct here was previously constructed by \cite[Lemma~2.2]{eldan2016}. This example provides a natural optimality criterion for this construction.
\end{example}

%\begin{theorem}\label{thm1}
%Fix a set of actions $\mathbb H$ satisfying condition \ref{itvii} of Theorem~\ref{thetheorem}, a discount rate $\beta>0$, and map $\tilde v\in \Acal$. Fix $c_1: \Pcal_p\times
%  \mathbb{H} \to \R$ and for each $\mu\in \Pcal_p$ and $\rho\in \mathbb H$ define
%$$
%    h(\mu) := \inf_{\rho \in \mathbb{H}} \left\{ c_1(\mu, \rho) +
%      L\tilde v(\mu,\rho)\right\}
%      \quad\text{and}\quad
%          c(\mu, \rho):= \beta \tilde v(\mu) + c_1(\mu, \rho) - h(\mu).
%      $$
% {\color{red} Assume that conditions \ref{iti}--\ref{itiv} and \ref{itvi} of Theorem~\ref{thetheorem} hold.}
%If the value function $v:\Pcal_p\to\R$ given by \eqref{eq_value_function} is continuous, then $v=\tilde v$.
%\end{theorem}

In order to verify the previous two examples, we first provide a general technical result.

\begin{theorem}\label{thm1}
Fix $p \in [0,\infty) \cup \{0\}$, $q \in [1,p] \cup \{0\}$, $\beta\geq0$ and a set of actions $\Hb$. Let $v\in C^2(\Pcal_q)$, $c_1: \Pcal_p\times
  \mathbb{H} \to \R\cup\{+\infty\}$, and for  $\mu\in \Pcal_p$ and $\rho\in \mathbb H$ set
$$
    h(\mu) := \sup_{\rho \in \mathbb{H}} \left\{ -c_1(\mu, \rho) 
      -Lv(\mu,\rho)\right\}
      \quad\text{and}\quad
          c(\mu, \rho):= \beta v(\mu) + c_1(\mu, \rho) +h(\mu).
      $$
Suppose that $c$ satisfies condition \eqref{eqn9} and that for each admissible control $(\xi,\rho)$ one has the inequality $\E[\sup_{t\geq0}|v(\xi_t)e^{(\e-\beta) t}|]<\infty$ for some $\e>0$. Then, for $\mu\in\Pc_p$, 
\begin{equation}\label{eqn8}
v(\mu)\leq \inf\left\{  \E\left[ \int_0^\infty e^{-\beta t} c(\xi_t,\rho_t) \dt \right] \colon \text{$(\xi,\rho)$ admissible control, $\xi_0=\mu$}\right\}.
\end{equation}
Moreover, given $\mu\in\Pcal_p$, if there exists an admissible control $(\xi^*,\rho^*)$ with $\xi^*_0=\mu$ and
$$\rho_s^*\in{\textup{argmax}}_{\rho\in\Hb} \left\{ - c_1(\xi_s^*,\rho) - Lv(\xi_s^*,\rho) \right\},\quad \P\otimes \dt-a.e.,$$
then $(\xi^*,\rho^*)$ is an optimal control and \eqref{eqn8} holds with equality.
\end{theorem}

\begin{proof}
Observe that  in this context equation \eqref{eq_HJB} reads
$$\beta u(\mu) -\beta  v(\mu) -h(\mu)+\sup_{\rho\in\Hb}\left\{-c_1(\mu,\rho) - Lu(\mu,\rho) \right\} = 0,$$
which is  satisfied by $u= v$. 
% Moreover, for each $x\in {\R^d}$ and $\rho\in \Hb$ we have $ L v(\delta_x,\rho)=0$ and thus 
% %\begin{align*}
% $$\frac 1 \beta  \inf_{\rho\in \mathbb H}c(\delta_x,\rho)=\frac 1 \beta \inf_{\rho\in \mathbb H}  \big(\beta  v(\delta_x) + c_1(\delta_x, \rho) + h(\delta_x)\big)
% %&= v(\delta_x) + \frac 1 \beta \left(\inf_{\rho \in \Hb}\big(c_1(\delta_x, \rho)+ L v(\delta_x,\rho)\big)- h(\delta_x)\right)\\
% = v(\delta_x),$$
% %\end{align*}
% which shows that \eqref{eq_HJB} is satisfied by $u=v$.
The claim then follows by Proposition~\ref{prop1}.
\end{proof}

%%%EXAMPLE 9.1
\begin{corollary}
The claimed results in Example~\ref{ex91} and Example~\ref{ex92} hold. 
\end{corollary}
\begin{proof}
Concerning Example~\ref{ex91},
observe that setting $v(\mu):=\frac 1 \beta \mu(\varphi )^2$ we have that 
$ v$ is a bounded map in $C^2(\Pcal)$ and $L v(\mu,\rho)=\frac 1 \beta \Cov_\mu(\varphi ,\rho)^2$.
We claim that the conditions of 
 Theorem~\ref{thm1} are satisfied for 
 $$c_1(\mu,\rho)=\alpha\Var_\mu(\bar\rho-\rho)-\frac 1 \beta \Cov_\mu(\varphi ,\rho)^2\quad \text{ and } \quad
h(\mu)=
%\inf_{\rho\in\mathbb H}\{c_1(\mu,\rho)+L\tilde v(\mu,\rho)\}=\inf_{\rho\in\mathbb H}\alpha\Var_\mu(g_\mu-\rho)=
0.$$
Observe that Jensen inequality yields
$$\Cov_{\xi_t}(\varphi ,\rho_t)^2
\leq 4\sup_{\R^d} |\varphi |^2\xi_t(|\rho_t|^2)
\leq 8C \sup_{\R^d} |\varphi |^2\xi_t(1+|\fdot|^{p}).$$
Since the latter is a martingale, $c$ satisfies condition \eqref{eqn9}. Finally, for any $\mu\in\Pc_p$, let $(\xi^*,\rho^*)$ be the weak solution of \eqref{eq_MVM_SDE}, with $\xi^*_0=\mu$ and $\rho_t^*=\bar\rho$ for all $t$, provided by Theorem~\ref{thm5}.
By Lemma~\ref{rem1} $(\xi^*,\rho^*)$ is an admissible control.
Since
$$\bar\rho\in{\textup{argmax}}_{\rho\in\Hb}\{-\alpha\Var_{\xi_s^*}(\bar\rho-\rho)\}= {\textup{argmax}}_{\rho\in\Hb} \left\{ - c_1(\xi_s^*,\rho) - Lv(\xi_s^*,\rho) \right\},$$
$\P$-a.s.~for almost every $s$, the claim follows.

Since the conditions of Proposition~\ref{lem1} and Theorem~\ref{T_comparison} are satisfied, we can conclude that $v$ is  the unique continuous viscosity solution of \eqref{eq_HJB}.

Concerning Example~\ref{ex92},
observe that including the state constraint in the cost function as explained in Remark~\ref{rem2}, the cost function $c$ considered here is of the form described in Theorem~\ref{thm1} for
 $v(\mu)=-\Mb(\mu)^2$ and $c_1(\mu,\rho)=\infty\bm1_{\{\Var_\mu(\rho)>\Var(\mu)\}}$; moreover, for any $\mu\in\Pc_p$,
$$\sup_{\rho\in\Hb} \left\{ - c_1(\mu,\rho) - Lv(\mu,\rho) \right\}$$ is attained by $\rho=\id$.
Indeed, by the Cauchy-Schwarz inequality, we observe that
$$\Cov_{\mu}(\id,\rho)^2 \le \Var(\mu) \Var_\mu(\rho)$$ with equality, if and only if, $\rho = \id$ $\mu$-a.s. Hence, $$h(\mu)=\sup_{\rho\in\Hb} \left\{ - c_1(\mu,\rho) - Lv(\mu,\rho) \right\}=\sup_{\rho\in\Hb(\mu)}\Cov_{\mu}(\id,\rho)^2=\Var(\mu)^2,$$ where both suprema are attained by $\rho=\id$. 

It thus suffices to verify the conditions of Theorem~\ref{thm1}. To this end, we first check that $\E[\sup_{t\geq0}|v(\xi_t)|]<\infty$ for each admissible control. Since $(\Mb_t)_{t\geq 0}$ is a square integrable martingale, by Doob's inequality, $$\E[\sup_{t\in[0,T]}\Mb(\xi_t)^2]\leq C \E[\Mb(\xi_T)^2]\leq C \E[\xi_T((\fdot)^2)]=C\mu((\fdot)^2);$$ 
sending $T$ to infinity, the claim follows by the monotone convergence theorem.
% Next, observe that by Jensen inequality, we have
% $\Var(\xi_t)^2 \leq 4\xi_t((\fdot)^4).$
% Since the latter is a martingale,
The same calculation also shows that $c$ satisfies condition \eqref{eqn9}. Finally, for any $\mu\in\Pc_p$, let $(\xi^*,\rho^*)$ be the weak solution of \eqref{eq_MVM_SDE}, with $\xi^*_0=\mu$ and $\rho_t^*=\id$ for all $t$, provided by Theorem~\ref{thm5}.  Since $(\xi^*,\rho^*)$ satisfies condition \eqref{eqn3}, and we have that
$$\id\in{\textup{argmax}}_{\rho\in\Hb} \left\{ - c_1(\xi_s^*,\rho) - Lv(\xi_s^*,\rho) \right\},$$
$\P$-a.s.~for almost every $s$, the claim follows.
%Observe that in this context $\Hb$ can be replaced by $\{\id\}$ in \eqref{eq_HJB}.
%Since the conditions of Proposition~\ref{lem1} and Theorem~\ref{T_comparison} are satisfied for $\Hb=\{\id\}$, we can conclude that $v$ is  the unique continuous viscosity solution of \eqref{eq_HJB}-\eqref{eq_BC}. 
\end{proof}

	\subsection{Optimal Skorokhod embedding problems} \label{sec:SEP}
	
	\subsubsection*{Skorokhod embedding problems and MVMs}
	
	Given $\mu\in\Pcal_1(\R)$ which is centered around zero, the classical Skorokhod embedding problem (SEP) is to find a (minimal) stopping time $\tau$ such that $B_\tau\sim\mu$ where $B$ is a Brownian motion. Since the solution is non-unique one typically looks for solutions with specific optimality properties; we refer to \cite{obloj2004} for the history of the problem and an overview of various solutions and to \cite{beiglboeck2017} for the current state of the art.
	
	The idea of connecting the SEP with MVMs goes back to \cite{eldan2016}. To specify the connection, we define as follows: we say that an MVM $\xi$ is {\it terminating in finite time}, if 
  		\begin{align}\label{def:TMVM}
  			\tau_s:= \inf \{t>0: \xi_t \in \Pc^s\}<\infty,~\as{}
  		\end{align} 
 	Via the correspondences
	\begin{align*}
		\xi_t\;\widehat =\;\mathcal{L}(B_\tau|\Fc_t),\;t\ge 0,
		\quad\textrm{and}\quad
		\tau\;\widehat=\;\tau_s,
	\end{align*}
	there is then a one-to-one correspondence between solutions $\tau$ to SEP$(\mu)$ and finitely terminating MVMs $\xi$ with $\xi_0=\mu$ and $\Mb(\xi_t)=B_t$, $t<\tau_s$, where we write $\Mb(\mu):=\mu(\mathrm{id})$.

	\subsubsection*{Formulating SEPs as stochastic control problems}
	
	Here, given a cost function, our aim is to search for solutions to the SEP which are optimal within our class of controlled MVMs. 
	Specifically, we assume that $\mu\in\Pc_2(\R)$, take $q=2$, and consider admissible controls which in addition satisfy the following state-constraint for some $\kappa\in(0,1)$:
	\begin{align}\label{eq_state-constraint}
	\rho_t\in\Hb(\xi_t),\;t<\tau_s,\;
	\textrm{with}\;\;
	\Hb(\mu)=\left\{\rho\in\Hb:\Cov_{\mu}(\mathrm{id},\rho) \in (1-\kappa,1+\kappa) \right\};
	\end{align}
	we note that such state-constraints can be handled within our framework by adding a corresponding penalisation term to the cost function.

	MVMs which satisfy this state-constraint notably terminate in finite time. Indeed, $\Var(\xi_t)+\Mb(\xi_t)^2=\xi_t(\id^2)$ is a martingale since $\xi_0\in\Pc_2$. Letting $\langle\Mb(\xi_\cdot)\rangle$ denote the quadratic variation process of $\Mb(\xi_\cdot)$ and using that $\di\langle\Mb(\xi_\cdot)\rangle_t=\Cov_{\xi_t}(\mathrm{id},\rho_t)\dt$ we thus obtain 
	\begin{align}\label{eq:var_relation}
		(1-\kappa)\E\left[t\wedge\tau_s\right]
		\le 
		\E\left[\langle\Mb(\xi_\cdot)\rangle_{t\wedge\tau_s}\right]
		= \Var(\xi_0)-\E\left[\Var\left(\xi_{t\wedge\tau_s}\right)\right],
	\end{align}
	from which it follows that $\tau_s<\infty$ a.s.
	Any admissible control thus characterises a solution to the SEP for there is a unique time-change transforming any such MVM into a terminating one whose average evolves as a Brownian motion.\footnote{Equivalently, one can consider the following scaled version of \eqref{eq_MVM_SDE}:
  	\begin{align*}%\label{eq:SDE_MVM_normalised}
		\di \xi_t(\varphi)=
		\frac{\Cov_{\xi_t}(\varphi,\rho_t)}{\Cov_{\xi_t}(\mathrm{id},\rho_t)}
		\di W_t,
		\quad
		\textrm{for all $\varphi\in C_b$},
		\;\; t<\tau_s;
	\end{align*}
	the embedding in \cite{eldan2016} was notably constructed by solving this equation for $\rho_t\equiv\mathrm{id}$, recall also Example~\ref{ex92}.}
	A similar time-change argument, combined with Theorem \ref{thm5}, ensures that the above class of state-constrained controls is non-empty. The corresponding optimisation problem is therefore well posed. 
		
	\begin{remark}\label{rem:B_sep}
		Given a (minimal) stopping time $\tau$, the MVM $\xi_t=\Lc(W_\tau|\Fc_t)$ satisfies $\Mb(\xi_t)=W_t$, $t\ge 0$. Moreover, if the filtration is Brownian, it is natural to expect $\xi$ to satisfy \eqref{eq_MVM_SDE} and thus also \eqref{eq_state-constraint}. However, if $\tau$ is not a stopping time in the Brownian filtration itself, even if $W_\tau\sim\mu$, it need not hold that $\Lc(W_\tau|\Fc^W_0)=\mu$.
		The fact that we here consider Brownian MVMs which satisfy both $\xi_0=\mu$ and \eqref{eq_state-constraint}, effectively imply that we are looking at `non-randomised' stopping times.
		Additional randomisation can be incorporated in our Brownian framework if one allows for controls for which $\Mb(\xi)$ may be constant; the Brownian motion is also then obtained by a time-change but its conditional distribution will feature a jump which is equivalent to the incorporation of additional information.
		To formalise this one needs to work with a different state-constraint (there are alternative conditions ensuring termination) or work with non-constrained solutions to \eqref{eq_MVM_SDE} and include a penalisation term or some alternative convention adapted to the problem at hand.
	\end{remark}

	\subsubsection*{An illustrating example: the Root and Rost problems}
	
	To illustrate how our control theory can be put to use, let $f:\R_+\to\R$ be a non-decreasing convex function and consider the problem of finding a (minimal) stopping time $\tau$, with $B_\tau\sim\mu$, minimising $\E[f(\langle B\rangle_\tau)]$. 
	It is well known that the general solution to this problem is given by the Root embedding; \cite{root1969} (see also \cite{kiefer1972,rost1976}). The corresponding problem where one maximises this expression is solved by the R{\"o}st embedding (see \cite{obloj2004}).
	
	Here, we are then looking for an admissible control, with $\xi_0=\mu$, which minimises $\E[f(\langle\Mb(\xi_\cdot)\rangle_{\tau_s})]$ among all such controls (since the quadratic variation is invariant with respect to time-changes, it does not matter that the average of our MVMs do not necessarily evolve as a Brownian motion). 
	It is clear that there is a trade-off between how much quadratic variation one has accumulated so far and how much of the terminal law that remains to be embedded; we define the value function associated with the conditional problem as follows:
	\begin{equation*}
			v(t,q,\mu):=\inf_{(\xi,\rho):\;\xi_t=\mu} \E\left[f\left(q+\int_t^{\tau_s}\Cov_{\xi_s}(\mathrm{id},\rho_s)^2\ds\right)\right],
	\end{equation*}
	where the infimum is taken over the state-constrained admissible controls.
	It is clear that $v$ is in fact independent of $t$.

	Compared to our standard framework, there is now an additional stochastic factor appearing in the value function, and the associated domain and boundary conditions are of a modified form. We expect, nevertheless, results parallel to our previous ones to hold; the associated HJB-equation takes the following form: 
	\begin{equation}\label{eq:hjb_root_general}
	-\inf_{\rho\in\Hb(\mu)}\left\{
	\Cov_{\mu}(\mathrm{id},\rho)^2\pian{v}{q}(q,\mu)+Lv(q,\cdot)(\mu,\rho)\right\}
	= 0,
	\;\;\textrm{$v(q,\mu)=f(q)$, $\mu\in\Pc^s$}.
  \end{equation}  
  %Note also that a multiplication of $\rho$ by a constant only results in a scaling of the expression within the brackets which allows the HJB to be further simplified.
  In the particular case $f=\mathrm{id}$, we have that $v(q,\mu)=q+\Var(\mu)$; indeed, for any admissible control with $\xi_0=\mu$, $\E[\langle\Mb(\xi_\cdot)\rangle_{\tau_s}]=\Var(\mu)$ (cf. \eqref{eq:var_relation}). 
  Hence, $\partial v/\partial q=1$, $\partial^2v/\partial\mu^2(x,y)=-2xy$ and $Lv(\mu,\rho)=-\Cov_{\mu}(\mathrm{id},\rho)^2$. As expected, the infimum in \eqref{eq:hjb_root_general} is therefore attained for each $\rho\in\Hb(\mu)$.

	\subsection{Robust pricing problems}

	\subsubsection*{Robust price bounds and MVMs}
	
	In mathematical finance, a central problem is to derive so-called robust price bounds. While classical approaches to option pricing rely on the specification of a market model, robust approaches acknowledge that a true model is not known. Meanwhile, there is consensus that fundamental no-arbitrage principles imply that the underlying asset prices should be martingales in any sensible (risk neutral) model. In addition, it is natural to restrict to models for which the prices of liquidly traded call options match actual market prices. Based on an old observation by Breeden and Litzenberger, the latter implies that the underlying price processes should fit certain marginal constraints. 
	
	Put together, given an exotic (path-dependent) option specified by a payoff function $\mathrm{\Psi}: C([0,T],\R)\to\R$, and a fixed marginal constraint $\mu\in\Pc$  (derived from market prices), a natural bound on the price of $\mathrm{\Psi}$ is obtained by maximising 
		\begin{align}\label{eq:robust_problem}
			\E\left[\mathrm{\Psi}\left((S_t)_{t\le T}\right)\right],
		\end{align}
	over probability spaces $(\Omega,\Hc,(\Hc_t)_{t\in[0,T]},\P)$ satisfying the usual conditions and supporting a \cadlag{} martingale $(S_t)_{t\le T}$ with $S_T\sim\mu$; we refer to \cite{hobson2003} for further motivation and an overview of some well-known bounds. 
	
	The study of this problem dates back to \cite{hobson1998} where it was solved for so-called lookback options depending on the past maximum of the underlying; the approach relied on the observation that since such payoffs are invariant with respect to time-changes, the pricing problem is equivalent to a certain optimal SEP. 
	 %The problem has later been studied for numerous different options; the vast majority of those works rely however on the same approach and are therefore restricted to options invariant with respect to time-changes. 
	 In \cite{cox2017} it was observed that the problem can be reformulated as an optimisation problem over MVMs starting off in $\mu$ and terminating at $T$. The equivalence rests on the following correspondences: 
	\begin{align*}
		\xi_t\;\widehat=\;\Lc(S_T|\Hc_t)
		\quad\textrm{and}\quad
		S_t\;\widehat=\;\Mb(\xi_t),
		\quad t\le T.
	\end{align*}
	The reformulation allows the problem to be addressed by use of dynamic programming arguments and the method thus requires neither time-invariance nor convexity of the payoff. 
	Here, the aim is to formulate this MVM-version of the pricing problem as a stochastic control problem within our framework.

	\subsubsection*{Formulating robust pricing problems as stochastic control problems}
	
	To put the problem into our framework, we choose to view it as a stochastic control problem on an (artificial) time-scale, say $r\ge 0$, on which two factor processes evolve: $(T_r)_{r\ge 0}$ governing current real time and $(\xi_r)_{r\ge 0}$ governing the law which currently remains to be embedded. The associated price process $(S_t)_{t\in[0,T]}$ is then defined via the correspondence 
	\begin{align*}%\label{eq:robust_S}
		S_{T_r}\widehat =\; \Mb(\xi_r).
	\end{align*}	
	More precisely, we consider tuples consisting of a filtered probability space $(\mathrm{\Omega},\Fc,\Fb,\P)$, a Brownian motion $W$, a continuous MVM $\xi$ taking values in $\Pcal$, a real-valued process $T$, and two progressively measurable processes $\rho$ and $\lambda$ taking values in $\Hb$ and $[0,1]$, respectively, such that for 
	$r<\tau_s:=\inf\left\{r>0:T_r\ge T\textrm{ or }\xi_r\in\Pc^s\right\}$, the following relations hold:
	 \begin{align}\label{eq:control_robust_T}
		\di T_r=\lambda_r\di r,
		\quad 
		\rho_r\in\Hb(\xi_r),
		%\quad r<\tau_s,
	\end{align}
	and
	\begin{align}\label{eq:control_robust_xi}
	\di\xi_r(\varphi)=\sqrt{1-\lambda_r}~
	\Cov_{\xi_r}(\varphi,\rho_r)~
	\di W_r,
	\quad\varphi\in C_b.
	%\;r<\tau_s,
	\end{align}		
	Given such a control, using the right-continuous inverse of $T$, we define $S_t=\Mb(\xi_\cdot)_{T^{-1}_t}$; 
	%$T^{-1}_t=\inf\{r>0:T_r>t\}$; 
	we employ the convention that if $\xi_{\tau_s}\not\in\Pc^s$ then $S$ realises a jump at $t=T$, and if $T_{\tau_s}<T$ then $S$ stays constant on $(\tau_s,T]$.
	Due to the state-constraint, $\tau_s<\infty$ a.s., and each admissible control thus defines a feasible price process $(S_t)_{t\in[0,T]}$. The problem of optimising over this class of price processes is therefore non-trivial and well posed. 
	Put into words, the controlled MVM governs how the conditional distribution of the process' terminal value -- $S_T$ --  evolves. 
	The presence of $\lambda$ allows however for a separate control of a time-change; this is convenient for it enables disentangling the control of the direction in which the MVM moves (controlled by $\rho$) from the speed at which it evolves (controlled by $\lambda$) with the extreme cases $\lambda =0$ and $\lambda=1$, respectively, corresponding to movement in the MVM only (the underlying realising a jump) or real time only (the underlying staying constant).

	\begin{remark}\label{rem:B_robust}
		Since \cadlag{} martingales can be written as time-changed Brownian motions, the robust pricing problem \eqref{eq:robust_problem} can be shown to be equivalent to an optimisation problem over time-changes and MVMs satisfying $\xi_0=\mu$.
		In general, the filtration needed for this is however bigger than the Brownian filtration itself. The fact that we here consider solutions to \eqref{eq:control_robust_T} -- \eqref{eq:control_robust_xi} with $\xi_0=\mu$, effectively means that we consider a class of potential market models for which the Brownian filtration does suffice for this procedure.
		In \cite{cox2017}, it was argued that for Asian options this restriction will not affect the robust price bounds; we expect similar arguments to apply also to other options. 
		Additional randomisation can however be incorporated within our Brownian framework by allowing for more general MVMs; see Remark \ref{rem:B_sep}. 
	\end{remark}

    \subsubsection*{An illustrating example: the Asian option}
    
	To illustrate how our control theory can be used to address this problem, we here specify the argument for the so-called Asian option. For a finitely supported $\mu$, this problem was solved by use of MVMs in \cite{cox2017} and the equations below are continuous analogues of the results derived therein.
	
	Given a function $F:\R\to\R$, the payoff of an Asian option is given by
     \begin{align*}
		\mathrm{\Psi}\left((S_t)_{t\in[0,T]}\right)=F\bigg(\int_0^TS_t\dt\bigg);
     \end{align*}
     it is notably not invariant with respect to time-changes. 
     In order to obtain a Markovian structure, it is convenient to introduce a state-variable governing the accumulated average. Hence, we introduce a factor-process $A$ with dynamics 
	\begin{align*}
		\di A_r=\lambda_r\Mb(\xi_r)\di r,
		\quad r<\tau_s.
	\end{align*}
	The problem then amounts to maximise $\E[F(A_{\tau_s}+\Mb(\xi_{\tau_s})(T-\tau_s))]$
	%\begin{align*}
	%	\E\big[F\big(A_{\tau_s}+\Mb(\xi_{\tau_s})(T-\tau_s)\big)\big],
	%\end{align*}
	over the class of admissible controls defined by \eqref{eq:control_robust_T} -- \eqref{eq:control_robust_xi}.
	The associated value function is given by
	\begin{align*}
		v(r,t,a,\mu):=\sup_{\stackrel{(\xi,\rho,T,\lambda):}{(T_r,A_r,\xi_r)=(t,a,\mu)}}
		\E\big[F\big(A_{\tau_s}+\Mb(\xi_{\tau_s})(T-\tau_s)\big)\big];
	\end{align*}
	we note that it is independent of $r$ and simply write $v(t,a,\mu)$. 
	In analogy to our previous results, we expect this value function to be linked to the equation 
	\begin{align*}
	-\sup_{(\rho,\lambda)\in\Hb(\mu)\times[0,1]}
	\left\{\lambda\left(\pian{v}{t}+\mathbb{M}(\mu)\pian{v}{a}\right)(t,a,\mu)
      +(1-\lambda)Lv(t,a,\cdot)(\mu,\rho)\right\}
      =0,
	\end{align*} 
      which, in turn, can be re-written as follows:
     	\begin{align}\label{eq:hjb_asian}
     	\left\{\begin{array}{rcl}
          		0 &=& -\max\bigg\{\left(\pian{v}{t}+\mathbb{M}(\mu)\pian{v}{a}\right)(t,a,\mu)~,~
          		\sup\limits_{\rho\in\Hb(\mu)}Lv(t,a,\cdot)(\mu,\rho)
          		\bigg\},\\ %\sup\limits
          		v(t,a,\mu) &=& F\left(a+\Mb(\mu)(T-t)\right),\quad\textrm{$\mu\in\Pc^s$ or $t=T$}.
		\end{array}
		\right.
		\end{align}		
%       	\begin{align}%\label{eq:hjb_asian}
%          		-\max\left\{\pian{v}{t}+\mathbb{M}(\xi)\pian{v}{a},
%          		\sup_{\rho\in\Hb(\xi)}
%          		\frac{1}{2}\int \piann{v}{\xi}(x,y)\sigma_{\rho,\xi}\otimes\sigma_{\rho,\xi}(\\dx,\dy)
%          		\right\}(\xi,t,a)=0,
%		\end{align}		
%	equipped with the boundary conditions that $v(\xi,t,a)=F(a)$, for $t=T$, and $v(\xi,t,a)=F(a+\Mb(\xi)(T-t)$, for $\xi\in\Pc^s$. 

	We see that for the case of Asian options, the supremum is always attained for $\lambda\in\{0,1\}$ which implies that market models attaining the price bound will be constant over certain intervals and then feature jumps. This is due to the particular structure of the Asian option and need in general not be the case.

	\subsection{Zero-sum games with incomplete information} \label{sec:ZSG}
	Our results are also closely related to results on certain two-player zero-sum games which feature asymmetry in the information available to the players. The study of such problems dates back to \cite{aumann1995}. In \cite{cardaliaguet_rainer2009,cardaliaguet2012}, such games were studied in a continuous time setup and linked to optimisation problems featuring MVMs; we briefly recall their setup. At the beginning of the game, the payoff function is randomly chosen -- according to a given distribution -- among a family of parameter-dependent payoff functions; the outcome is communicated only to the first player while the second only knows the probability distribution it was drawn from. One player is then trying to minimise and the other to maximise the expected payoff (which depends on the players' actions).  Since the actions are visible to both players, the uninformed player will try to deduce information about the actual payoff function from the actions of the first player; she will then act optimally based on this information. Since the first player is aware of this, it turns out that the problem can be formulated as an optimisation problem over the second player's beliefs about the game. In effect, the first player is controlling the game by choosing how much information to reveal in order to optimally steer the second player's beliefs. The problem is thus equivalent to an optimisation problem over the process representing the belief of the second player processes -- which are measure-valued martingales.

	Specifically, it was shown in \cite[Theorem~3.2]{cardaliaguet2012} that the value of the game admits the following equivalent formulation (we also refer to \cite[Theorem~3.1]{cardaliaguet_rainer2009} for the case of finitely many payoff functions and thus atomic MVMs):
	\begin{align*}%\label{eq:problem_games}
		\inf_{\textrm{MVMs $(\eta_t)_{t\ge 0}:\eta_0=\mu$}}
		\E\Big[\int_0^Th(t,\eta_t)\dt\Big],
	\end{align*}
	where
	\begin{align}\label{eq:cost_games}
		h(t,\mu):=\inf_{u\in\mathcal{U}}\sup_{v\in\mathcal{V}}
		\mu\big(l(\cdot,t,u,v)\big);
	\end{align}
	here $l$ is the given (parameter-dependent) payoff function and $\mathcal{U}$ and $\mathcal{V}$ are the state-spaces of the respective players' controls. These results require the Isaacs assumption, that is, the infimum and supremum in \eqref{eq:cost_games} can be interchanged.

	It is of course possible to formulate this problem within our stochastic control framework, provided we restrict to beliefs processes represented via time-changes and solutions to our SDE; that is, MVMs $\eta$ which admit the representation
	\begin{align*}
		\eta_t=\xi_{T^{-1}_t},\quad t\in[0,T],
	\end{align*}
	where $\T_\cdot$ and $\xi_\cdot$ are given by \eqref{eq:control_robust_T} and \eqref{eq:control_robust_xi} for some admissible control $(\lambda,\rho)$. Optimising (in a weak sense) over such controls, yields the following HJB-type equation (closely related to \eqref{eq:hjb_asian}; see also \cite[Section~4]{cardaliaguet2012}) for the associated value function:
	
        \begin{align*}%\label{eq:hjb_games}
          \min\bigg\{\pian{v}{t}(t,\mu)+h(t,\mu),
          \inf_{\rho\in\Hb(\mu)}Lv(t,\cdot)(\mu,\rho)
          \bigg\}=0,
          \quad v(T,\mu)=0,\;\mu\in\Pc.
        \end{align*}	 
        
        We stress that our arguments do not require convexity of the value-function and in contrast to the results in \cite{cardaliaguet2012}, they should thus apply also to generalisations of the game leading to non-convex value functions. We briefly outline one possible such extension here (although we leave details to subsequent work). Suppose in the framework of the game above, the informed player were further incentivised not to reveal information to the uninformed player through an additional cost relating to the strength of the control exerted in the uninformed player's belief process. Assuming that the analysis of \cite{cardaliaguet_rainer2009} and \cite{cardaliaguet2012} carries through in much the same manner, one might end up considering the optimisation problem:
	\begin{align*}%\label{eq:problem_games}
		\inf_{\textrm{MVMs $(\eta_t)_{t\ge 0}:\eta_0=\mu$}}
		\E\Big[\int_0^T \left (h(t,\eta_t)+c(\rho_t)\right) \dt\Big],
	\end{align*}
	where $\rho$ is the control of the MVM $\eta$, and $c$ represents the cost to the informed player of controlling the MVM in the direction $\rho$. This would formally give rise to the HJB equation
        \begin{align*}%\label{eq:hjb_games}
          \pian{v}{t}(t,\mu)+h(t,\mu)
          + \inf_{\rho\in\Hb(\mu)}\left\{Lv(t,\cdot)(\mu,\rho)+c(\rho)\right\}
          =0,
          \quad v(T,\mu)=0,\;\mu\in\Pc.
        \end{align*}	 
        The addition of the cost term in the second half of the HJB equation means that the value function is no longer required to be convex.

\begin{appendix}

\section{The dynamic programming principle}\label{app:DPP}

	In this appendix we establish the dynamic programming principle for our problem of study (cf. Theorem \ref{thm_dpp_main_text}); following e.g. \cite{karoui2013_1,karoui2013_2,zitkovic2014}, see also \cite{nutz2013} or \cite{neufeld2013}, we acknowledge that it is often easier to prove the DPP by working on a canonical path space and concatenate measures rather than processes.
	Recall that we have fixed $p \in [1,\infty) \cup \{0\}$, $q \in [1,p] \cup \{0\}$, and a Polish space $\Hb$ of measurable real functions on $\R^d$ that satisfies the standing assumption that the evaluation map $(\bar\rho,x) \mapsto \bar\rho(x)$ from $\Hb \times \R^d$ to $\R$ is measurable. 
	Writing $M$ for the set of Borel measures on $\R_+\times\Hb$, we define
	\[
		\Mb=\left\{m\in M \colon \textrm{$m(\ds,\du)=\tilde m(s,\du)\ds$ for some kernel $\tilde m$}\right\}
	\]
	and 
	\[
	\Mb_0=\left\{m\in\Mb:m(\ds,\du)=\delta_{\tilde\rho(s)}(\du)\ds\textrm{ for some measurable function $\tilde\rho$}\right\};
	\]
	we equip $\Mb$ with the same topology as in \cite[Remark~1.4]{karoui2013_2} rendering it a Polish space.
	The canonical path space is now given by the Polish space
	\begin{equation*}
		{\mathrm{\Omega}}:= C(\R_+,\R)\times C(\R_+,\Pc_p)\times\Mb.
	\end{equation*}
%	{\red since $\Pc_p$ is a Polish space,  so are both $C(\R_+,\R)$ and $C(\R_+,\Pc_p)$ under the local uniform topology. }\comment{SSF: did I interpret this sentence correctly?}
	The set of all Borel probability measures on ${\mathrm{\Omega}}$ is denoted by $\Pk$ and under the weak convergence topology it is a Polish space too. 
	A generic element of $\mathrm{\Omega}$ is denoted by $\omega=(B,\xi,m)$ and we use the same notation for the canonical random element.
%	 and we similarly denote the canonical process by $(B_s,\xi_s,\rho_s)(\omega)=(B,\xi,\rho)(s)$; we employ the convention that whenever $\rho\in\Mb_0$, $\rho$ is also used to denote the corresponding element of $\Bc(\R_+,\Hb)$. 
%	equipping $\Pc_p$ with the topology induced by the $\mathcal{W}_p$-metric renders it a Polish space and under the local uniform topology so are both $C(\R_+,\R)$ and $C(\R_+,\Pc_p)$. 
%	The set of all Borel probability measures on ${\mathrm{\Omega}}$ is denoted by $\Pk$ and under the weak convergence topology it is a Polish space too. 
%	A generic element of $\mathrm{\Omega}$ is denoted by $\omega=(B,\xi,\rho)$ and we similarly denote the canonical process by $(B_s,\xi_s,\rho_s)(\omega)=(B,\xi,\rho)(s)$; we employ the convention that whenever $\rho\in\Mb_0$, $\rho$ is also used to denote the corresponding element of $\Bc(\R_+,\Hb)$. 
	We note that since $\Hb$ is Polish, it is isomorphic to a Borel subset of $[0,1]$; we let $\psi \colon \Hb\to[0,1]$ be the bijection between $\Hb$ and $\psi(\Hb)\subseteq[0,1]$ and define $\chi:\R\to\Hb$ by
	\begin{equation*}
		\chi(x)=\left\{\begin{array}{lll}
		\psi^{-1}(x) & x\in\psi(\Hb)\\
		\bar\rho & x\not\in\psi(\Hb),
		\end{array}\right.
	\end{equation*}	
	where $\bar\rho$ is some fixed element of $\Hb$. 
	In turn, let $\rho:\mathrm{\Omega}\to\Bc(\R_+,\Hb)$ be given by
		\begin{equation*}
			\rho_t:=\chi\left(\frac{\partial}{\partial t}\int_0^t\int_\Hb\psi(u)m(\ds,\du)\right), \quad t\ge 0, 
                      \end{equation*}
         where the derivative is taken as the $\liminf$ of differences from the left.
	If $m\in\Mb_0$, and thus of the form $m(\ds,\du)=\delta_{\tilde\rho(s)}(\du)\ds$ for some $\tilde\rho\in\Bc(\R_+,\Hb)$, then $\rho_\cdot=\tilde\rho(\cdot)$ Lebesgue-a.e.
	We denote by $\Fb^0=(\Fc^0_t)_{t\ge 0}$ the canonical filtration given by 
		\begin{equation*}
			\Fc^0_t:=\sigma\left\{B_r,\xi_r, \int_0^r \int_\Hb \phi(u) \, m(\di s,\du) : \phi\in C_b(\Hb,\R_+),\; r\le t\right\}.
%			\Fc^0_t:=\sigma\left\{B_s,\xi_s, \bm 1_{\{\rho\in\Mb_0\}}\textstyle{\int}_0^s\phi(\rho_u)\du : \phi\in C_b(\Hb,\R_+),\; s\le t\right\}.
		\end{equation*}
The $\Hb$-valued process $\rho_t(\omega)$ is then progressively measurable and hence, because the evaluation map $(\bar\rho,x)\mapsto \bar\rho(x)$ from $\Hb \times \R^d$ to $\R$ is measurable, it is also a progressively measurable function.
%                \comment{AC: Why not just use $\int_0^s \int \phi(u) \, m(s,\du)\, \ds$ for third coordinate?}
	 For $\mu\in\Pc_p$, we then define $\Pk_\mu$ to be the set of measures $\Pp\in\Pk$ which satisfy the following properties; here $ C^\infty_0(\R\times\R)$ denotes the set of smooth functions in $C(\R\times \R)$ vanishing at infinity:
	\begin{enumerate}
		\item\label{it3i}  $\Pp$-a.s., $\xi_0=\mu$ and $m\in\Mb_0$, and thus $m(\ds,\du)=\delta_{\rho(s)}(\du)\ds$;
%		\item[i)]  $\P$-a.s., $\rho\in\Mb_0$ and $\xi_0=\mu$;
		\item\label{it3ii} $\Pp\otimes\dt$-a.s. $\xi_t(|\rho_t|)<\infty$ and
		\begin{equation}\label{eq:DPP_worthy}
			\int_0^t\left(\int_{\R^d} (1 + |x|^q) \left|\rho_s(x)-\xi_s(\rho_s)\right| \xi_s(\dx)\right)^2\ds<\infty;
		\end{equation}
		\item\label{it3iii} for every $f\in C^\infty_0(\R\times\R)$ and $\varphi\in C_b({\R^d})$, the following process is a $(\Fb^0,\Pp)$-local martingale, where $\sigma_t=(1,\sigma_t(\varphi))^T$ with $\sigma_t(\varphi)=\Cov_{\xi_t}(\varphi,\rho_t)$:
	\begin{equation}\label{eq:mp_simplified}
		f\left(B_t,\xi_t(\varphi)\right)-\int_0^t\frac{1}{2}\sum_{i,j=1}^2\frac{\partial^2f}{\partial x_i\partial x_j}\left(B_s,\xi_s(\varphi)\right)\left(\sigma_s\sigma_s^T\right)_{ij}\ds,
		\quad t\ge 0.
	\end{equation}	
%	\begin{equation*}
%		\sigma_t=
%		\left[\begin{array}{c}1\\ \Cov_{\xi_t}(\varphi,\rho_t) \end{array}\right].
%		\left[\begin{array}{cl}1 & \sigma_t(\varphi)\\ \sigma_t(\varphi) & \sigma_t(\varphi)^2 \end{array}\right]_{ij}.
%	\end{equation*}
	\end{enumerate}

	Our control problem then admits the following equivalent representation: 

	\begin{lemma}\label{lem:dpp_equiv_problems}
For the value function $v$ defined in \eqref{eq_value_function}, it holds that
	\begin{equation}\label{eq:vdefn}
		%\xi\mapsto 
		v(\mu)
		=\inf_{\Pp\in\Pk_\mu}\E^\Pp\left[\int_0^\infty e^{-\beta t}c(\xi_t,\rho_t)\dt\right],
		\quad \mu\in \Pc_p.
		%:=\sup_{\mu\in\Pc_x}\E^\mu\left[A_{\tau^s}+\Mb(\xi_{\tau^s})(T-T_{\tau_s})\right];
	\end{equation}
	\end{lemma}

	\begin{proof}
		First, by use of Theorem \ref{T_Ito}, we immediately obtain that any admissible control $(\mathrm{\Omega},\Fc,(\Fc_t)_{t\ge 0},\P,W,\xi,\rho)$, with $\xi_0=\mu$, $\P$-a.s., induces a measure $\Pp\in\Pk_\mu$.
		
		Conversely, given $\Pp\in\Pk_\mu$, define ${\mathrm{\Omega}}_0= C(\R_+,\R)\times C(\R_+,\Pc_p)\times\Mb_0$, $\Fc=\Bc({\mathrm{\Omega}})\cap{\mathrm{\Omega}}_0$ %$\overline\F_\infty:=\vee_{t\ge 0}\Fc^0_t$. 
		and let $\Fb=(\Fc_t)_{t\ge 0}$ be the $\Pp$-augmentation of $\Fb^0$. %; that is, the $\P$-augmentation of $(\cap_{s>t}\Fc^0_s)_{t\ge 0}$.
%                \comment{AC: What does `hull' mean here? I think we just want $\P$-augmentation, which usually(?) is the smallest $\P$-complete, right cts filtration containing $\Fb^0$?}
	On the filtered probability space $(\Omega_0,\Fc,\Fb,\Pp)$, $\rho$ then defines a progressively measurable $\Hb$-valued stochastic process and a progressively measurable function. 
	To show that the tuple $({\mathrm{\Omega}}_0,\Fc,\Fb,\Pp,B,\xi,\rho)$ is an admissible control, it only remains to show that $B$ is a Brownian motion and that \eqref{eq_MVM_SDE} holds.	
		To this end, note that the (local) martingale property is preserved when considering the augmented filtration. Hence, with $\sigma_t(\varphi)=\Cov_{\xi_t}(\varphi,\rho_t)$, the process given in \eqref{eq:mp_simplified} is a $(\Fb,\Pp)$-local martingale. 
		It follows that $\di\langle B\rangle_t=\dt$, $\di\langle B,\xi(\varphi)\rangle_t=\sigma_t(\varphi)\dt$ and $\di\langle \xi(\varphi)\rangle_t=\sigma_t(\varphi)^2\dt$, where $\langle B\rangle$ and $\langle \xi(\varphi)\rangle$ denote the quadratic variation process of $B$ and $\xi(\varphi)$, respectively, and $\langle B,\xi(\varphi)\rangle$ denotes the corresponding quadratic covariation process. In particular, $B$ is a Brownian motion. Further, defining
		\begin{equation*}
			X^\varphi_t:=\mu(\varphi)+\int_0^t \sigma_s(\varphi)\di B_s,
			\quad \varphi\in C_b({\R^d}),
		\end{equation*}
		%It follows that $(X^\varphi_t,\xi_t(\varphi))$ is then a solution to the martingale problem with coefficient matrix $[\tilde a_t^{i,j}]_{i,j=i}^2$ given by
		%\begin{equation*}
		%\big[\tilde a_t^{i,j}\big]_{i,j=i}^2
		%=
		%\left[\begin{array}{rl}\sigma_t(\varphi)^2 &\sigma_t(\varphi)^2 \\ \sigma_t(\varphi)^2 & \sigma_t(\varphi)^2 \end{array}\right],
		%\end{equation*}
		it holds that $\left(X^\varphi_t-\xi_t(\varphi)\right)^2$ is a local martingale. Hence, $X^\varphi$ and $\xi(\varphi)$ are indistinguishable which completes the proof. 
	\end{proof}

	To obtain the DPP we first establish some properties of the sets $\Pk_\mu$, $\mu\in\Pc_p$.

	\begin{lemma}\label{lem:analytic}
The graph $\{(\mu,\Pp):\mu\in\Pc_p,\Pp\in\Pk_\mu\}$ is a Borel set in $\Pc_p\times\Pk$. 
	\end{lemma} 
	
	\begin{proof}
		We may consider each property separately and show that the subset of pairs $(\mu,\Pp)$ in $\Pc_p\times\Pk$ for which the property holds is a Borel set. 
		
		\ref{it3i}:
		We have that $\Mb_0$ is a Borel subset of $\Mb$; see e.g.\ \cite[Appendix]{karoui_nguyen_jeanblanc}.
		In analogy to the above, denote by $\tilde\psi$ and $\tilde\chi$ the bijection and its inverse between $\Pc_p$ and the set $\tilde\psi(\Pc_p)\subset[0,1]$. Note that 
		\begin{align*}
			\{(\mu,\Pp) \colon \Pp(\xi_0=\mu) & =1\}\\
			=&
			\big\{(\mu,\Pp): \Var^\Pp\big[\tilde\psi(\xi_0)\big]=0 %\E^\Pp\left[\psi(\xi_0)^2\right]=\E^\Pp\left[\psi(\xi_0)\right]^2
			\big\}
			\cap\big\{(\mu,\Pp): \E^\Pp\big[\tilde\psi(\xi_0)\big]=\tilde\psi(\mu)\big\}.
		\end{align*}
		Since $\Pp\mapsto\tilde\chi(\E^\Pp[\tilde\psi(\xi_0)])$ is a measurable function, its graph is a Borel set.
		In consequence, so is $\{(\mu,\Pp)\in\Pc_p\times\Pk:\textrm{$m\in\Mb_0$ and $\xi_0=\mu$, $\Pp$-a.s.}\}$. 
		
		\ref{it3ii}:
		The mapping $(\omega,t)\mapsto\xi_t(\omega)(|\rho_t(\omega)|)$ defines an extended-valued measurable function on ${\mathrm{\Omega}}\times[0,\infty)$; hence		
%		 hence the following set is a Borel set: 
%		\begin{equation*}
%			\left\{\textrm{$\mathbb{Q}\in\Ppc({\mathrm{\Omega}}\times[0,\infty))$}:\;\mathbb{Q}(\xi(|\rho|)<\infty)=1\right\}.
%			%\bigcap
%			%\left\{\mathbb{Q}\in\Ppc\left({\mathrm{\Omega}}\times[0,\infty)\right):\mathbb{Q}({\mathrm{\Omega}}\times A)=\lambda(A)\right\},
%		\end{equation*}
%		Further, the subset thereof consisting of measures whose projection onto the second component is the Lebesgue measure is also Borel. Since projections of Borel sets are analytic it follows that $\{\Pp\in\Ppk: \textrm{$\Pp\otimes\dt$-a.e.}, \xi(|\rho|)<\infty\}$ is analytic. \\	
		\begin{equation*}
			A=
			\bigcap_{r\in\mathbb{Q}}
			\left\{\Pp\in\Pk:\Pp\left(\int_0^r
			\bm1_{\{\xi_s(|\rho_s|)=\infty\}}
			\ds=0\right)=1\right\}
		\end{equation*}
        %\comment{AC: Why in this complicated form? Why not just: $\P(\int_0^r \bm 1_{\{\xi_s(|\rho_s|)=\infty\}} \, \ds = 0) = 1$?}
		is a Borel set. In consequence, so is
		\begin{equation*}
			\bigcap_{r\in\mathbb{Q}}
			\left\{\Pp\in A:\Pp\left(\int_0^r
			\left(\int_{\R^d}(1+|x|^q)\left|\rho_s(x)-\xi_s(\rho_s)\right|\xi_s(\dx)\right)^2
			\ds<\infty\right)=1\right\}.
		\end{equation*}
		Hence, the subset of measures in $\Pk$ for which \ref{it3ii} holds is a Borel set. 
		
		\ref{it3iii}:
		Given that property {\ref{it3ii}} holds, for $(\varphi_n)$
                converging in the bounded pointwise sense to
                $\varphi$, it holds that
                $\E[\int_0^t\Cov_{\xi_s}(\varphi_n-\varphi,\rho_s)^2\ds]\to
                0$; since $C_b({\R^d})$ has a countable dense subset in
                the sense of bounded pointwise convergence, it suffices to check {\ref{it3iii}} for $\varphi$ in a countable subset of $C_b({\R^d})$. There is also a countable subset of $C^\infty_0(\R\times\R)$ (dense with respect to pointwise convergence of the first and second derivatives) such that if {\ref{it3iii}} holds for any $f$ within that set, then it holds for any $f\in C^\infty_0(\R\times\R)$. 
		
		Denote now the continuous process in \eqref{eq:mp_simplified} by $(\omega,t)\mapsto M^{\varphi,f}_{t}(\omega)$, and note that $H_{\pm n} = \inf \{ s \ge 0 : |M^{\varphi,f}_s| \ge n\}$ is an $\mathbb{F}^0$-stopping time by continuity of the paths of $M^{\varphi,f}$. For $\varphi\in C_b(\R^d)$, $f\in C^\infty_0(\R\times\R)$, $r\le s$, $A\in\Fc^0_r$ and $n\in\mathbb{N}$, it then holds that
		\begin{equation*}
			\left\{\Pp\in\Pk: \E^\Pp\left[\left(M^{\varphi,f}_{s \wedge H_{\pm n}}-M^{\varphi,f}_{r \wedge H_{\pm n}}\right)\bm 1_A\right]=0\right\}
		\end{equation*}
		is a Borel set. 
		In consequence, so is the intersection of such sets when $\varphi$ and $f$ range through the above-mentioned countable subsets, $r,s$ and $n$ through the rationals, and $A$ through a countable algebra generating $\Fc^0_r$; this is sufficient to ensure property \ref{it3iii}. 	
		\end{proof}

	We call a collection $(\Pp_\mu)_{\mu\in\Pc}$ such that $\mu\mapsto\Pp_\mu$ is universally measurable and $\Pp_\mu\in\Pk_\mu$, $\mu\in\Pc_p$, an \emph{admissible kernel}. Given $\Pp\in\Pk$ and an admissible kernel $(\Pp_\mu)_{\mu\in\Pc}$, writing
	\begin{equation*}
		(\omega\otimes_t\omega')(s)=
%		\omega(s), \;\;\textrm{for $s<t$},
%		\quad\textrm{and}\quad
%		(\omega\otimes_t\omega')(s)=
%		\omega'(s-t), \;\;\textrm{for $s\ge t$},
		\left\{\begin{array}{ll}
		\omega(s) &s<t\\
		\omega'(s-t) &s\ge t
		\end{array}\right.,
		\quad \omega,\omega'\in{\mathrm{\Omega}}, %with $\xi_0(\omega')=\xi_t(\omega)$
	\end{equation*}
	we define for any random time $\tau:{\mathrm{\Omega}}\to\R_+$,
	\begin{equation*}
		(\Pp\otimes_\tau\Pp_\cdot)(A)
		=
		\int_{{\mathrm{\Omega}}\times{\mathrm{\Omega}}} \bm 1_A\left(\omega\otimes_{\tau(\omega)}\omega'\right)\Pp_{\xi_{\tau}(\omega)}(\di\omega')\Pp(\di\omega),
		\quad A\in\Bc({\mathrm{\Omega}}).
	\end{equation*}
%	\begin{equation*}
%		(\Pp\otimes_\tau\Pp_\cdot)(A)
%		=
%		\int_{\Omega\times\Omega} 
%		\ind_A\Big(\omega(s)\ind_{\{s<\tau(\omega)\}}+\omega'(s)\ind_{\{s\ge \tau(\omega)\}}\Big)
%		\Pp_{\xi_{\tau}(\omega)}(\di\omega')\Pp(\di\omega),
%		\; A\in\Fc. 
%	\end{equation*}
	Our family $(\Pk_\mu)_{\mu\in\Pc_p}$ is then stable under disintegration and concatenation in the following sense; the proof is similar to that of \cite[Lemma~3.3]{karoui2013_2} or \cite[Proposition~2.5]{zitkovic2014} and we omit the details: 
	
	\begin{lemma}\label{lem:concatenation}
          Let $\tau$ be a
          finite $\Fb^0$-stopping time, $\bar\mu\in\Pc_p$ and $\Pp\in\Pk_{\bar\mu}$. 
          Then,
		\begin{enumerate}
			\item there exists an admissible kernel $(\Pp_\mu)_{\mu\in\Pc_p}$ such that $\Pp=\Pp\otimes_\tau\Pp_\cdot$; 
			\item conversely, given an admissible kernel $(\Pp_\mu)_{\mu\in\Pc_p}$, it holds that $\Pp\otimes_\tau\Pp_\cdot\in\Pk_{\bar\mu}$.
		\end{enumerate}
	\end{lemma}

	By use of Lemmas \ref{lem:analytic} and \ref{lem:concatenation} the following result can now be easily derived; we refer e.g.~to the proof of \cite[Theorem~2.1]{karoui2013_2} or \cite[Theorem~2.4]{zitkovic2014} for an outline of the argument.

\begin{theorem}\label{thm:dpp}
  For any $\Fb^0$ stopping time $\tau$, it holds that
  	\begin{equation*}
    	v(\mu)=\inf_{\Pp\in\Pk_\mu}\E^\Pp\left[\int_0^\tau e^{-\beta t}c(\xi_t,\rho_t)\dt+e^{-\beta \tau} v(\xi_\tau)\right],
   		\quad \mu\in\Pc_p.
  	\end{equation*}
\end{theorem}

	We conclude by noticing that Theorem \ref{thm_dpp_main_text} is an immediate consequence of the above result and (the proof of) Lemma \ref{lem:dpp_equiv_problems}.

	\section{Properties of the derivatives}\label{secC}
	In the following lemma we provide some basic properties of the derivative. The continuity result is classical and a proof in similar contexts can be found in the literature (see for instance the discussion at page 416 in \cite{car_del_18_I}). 
	\begin{lemma}
Fix $p \in [1,\infty) \cup \{0\}$ and a map $f\in C^1(\Pcal_p)$. Then $f$ is a continuous map and its derivative is uniquely determined up to a continuous additive term of the form $\mu\mapsto a(\mu)$. If $f\in C^2(\Pcal_p)$ then its second derivative is  uniquely determined up to a continuous additive term of the form $(x,y,\mu)\mapsto a(x,\mu)+b(y,\mu)$. 
	\end{lemma}
	\begin{proof}
	  To prove continuity of $f$ along a sequence $(\mu_n)_n$ converging to $ \mu$ by \eqref{eq_C1r_FTC} it suffices to show that 
  \begin{align*}
  \bigg| \int_{\R^d} \frac{\partial f}{\partial\mu}(x,t\mu_n + (1-t)\mu)-\frac{\partial f}{\partial\mu}(x,\mu)(\mu_n-\mu)(\dx)\bigg|
  + \bigg| \int_{\R^d}\frac{\partial f}{\partial\mu}(x,\mu)(\mu_n-\mu)(\dx)\bigg|
  \end{align*}
  vanishes for $n$ going to infinity.
  The second term converges to zero due to continuity of the derivative and \eqref{eq_C1r_p_growth}. To prove convergence of the first term it suffices to show that
  $$\lim_{n\to\infty}\sup_{\nu\in K}\int_{\R^d}\bigg| \frac{\partial f}{\partial\mu}(x,t\mu_n + (1-t)\mu)-\frac{\partial f}{\partial\mu}(x,\mu)\bigg|\nu(\dx)=0$$
  for $K:=\{\mu_n\colon n\in \N\}\cup\{\mu\}$. Fix $\e>0$. Since $K$ is compact the map $\nu\mapsto\nu(1+|\cdot|^p)$ is bounded on $K$ and we can find a map   $\varphi\in C_c(\R^d)$ such that $0\leq\varphi(x)\leq 1$ and 
  $$\sup_{\nu\in K}\bigg|\int(1+|x|^p)(1-\varphi(x))\nu(\dx)\bigg|<\e.$$
  Since continuous maps are uniformly continuous on compacts we can conclude that for $n$ large enough
\begin{align*}
    &\sup_{\nu\in K}\int \bigg| \frac{\partial f}{\partial\mu}(x,t\mu_n + (1-t)\mu)-\frac{\partial f}{\partial\mu}(x,\mu)\bigg| \nu(\dx)\\
  &\qquad\leq \sup_{\nu\in K}\int \bigg| \frac{\partial f}{\partial\mu}(x,t\mu_n + (1-t)\mu)-\frac{\partial f}{\partial\mu}(x,\mu)\bigg| \varphi(x)\nu(\dx)+\e
  \leq 2\e,
\end{align*}
  proving the first claim.

	Uniqueness of the derivative can be shown by proving that every version of $\frac{\partial f}{\partial\mu}$ for $f=0$ does not depend on $x$. 
   Fix $\mu\in \Pcal_p$, $\overline x \in \R^d$, and note that condition \eqref{eq_C1r_FTC} for $f=0$ and  $\nu=(1-\e)\mu+\e\delta_{\overline x}$ yields 
$$
0 = \int_0^1 \bigg(\frac{\partial f}{\partial\mu}(\overline x,\mu+t\e(\delta_{\overline x}-\mu))-\int_{\R^d} \frac{\partial f}{\partial\mu}(x,\mu+t\e(\delta_{\overline x}-\mu))\mu(\dx)\bigg)\dt,
$$
  for each $\e>0$. Since $K:=\{\mu+t(\delta_{\overline x}-\mu)\colon t\in [0,1]\}$ is a compact set, by the continuity of $\frac{\partial f}{\partial\mu}$ and \eqref{eq_C1r_p_growth} we can apply the dominated convergence theorem to conclude that $\frac{\partial f}{\partial\mu}(\overline x,\mu)=\int_{\R^d} \frac{\partial f}{\partial\mu}(x,\mu)\mu(\dx)$.

To prove uniqueness of the second derivative  set again $f=0$, $\mu\in \Pcal_p$, and $\nu=(1-2\e)\mu+\e(\delta_{\overline x}+\delta_{\overline y})$.
Proceeding as for the first order derivative  conditions \eqref{eq_C2r_FTC} and  \eqref{eq_C2r_p_growth} and the imposed symmetry yield
$$
\begin{aligned}
0&= \frac{\partial^2 f}{\partial\mu^2}(\overline x,\overline y,\mu)
-\Big(\int_{\R^d}\frac{\partial^2 f}{\partial\mu^2}(\overline x,y,\mu) \mu(\dy)
+\int_{\R^d}\frac{\partial^2 f}{\partial\mu^2}(x,\overline y,\mu) \mu(\dx)\Big)\\
&\qquad+\int_{{\R^d}\times {\R^d}} \frac{\partial^2 f}{\partial\mu^2}(x,y,\mu)\mu^{\otimes 2}(\dx,\dy),
\end{aligned}
$$
  proving the claim.
\end{proof}
\end{appendix}

\bibliographystyle{abbrvnat}
\bibliography{refs.bib}

\end{document}